\documentclass[12pt,reqno]{amsart}
\usepackage{hyperref}
\usepackage{amscd,amssymb,amsfonts,amsbsy,amsmath,verbatim,color, mathrsfs}
\usepackage{pst-node}
\usepackage{tikz-cd}
\usepackage{graphicx,cite}
\usepackage{subcaption}
\usepackage{enumitem}
\usepackage{amsmath}

\numberwithin{figure}{section}
\allowdisplaybreaks
\usepackage{float}
\usepackage{tikz,ifthen,bm}
\usepackage{verbatim,mathdots,caption,subcaption,epstopdf}
\usetikzlibrary{intersections}
\usetikzlibrary{calc}

\newtheorem{thm}{Theorem}[section]

\newtheorem{prop}[thm]{Proposition}

\theoremstyle{definition}

\newtheorem{defi}[thm]{Definition}
\newtheorem{rema}[thm]{Remark}

\numberwithin{equation}{section}

\newcommand{\E}{{\mathcal E}}

\newif\ifshowhidden  
\showhiddenfalse

\makeatletter

\newcommand{\Rmnum}[1]{\expandafter\@slowromancap\romannumeral #1@}
\makeatother

\begin{document}
	\baselineskip=17pt
	\setcounter{figure}{0}
	
	\title[Hodge-de Rham Theory on Higher-Dimensional Level-$L$ Sierpi\'{n}ski Gaskets]
{Hodge-de Rham Theory on Higher-Dimensional Level-$L$ Sierpi\'{n}ski Gaskets}

\author[S.-M. Ngai]{Sze-Man Ngai}
	\address{Beijing Institute of Mathematical Sciences and Applications, Huairou District, 101400, Beijing, China, and Key Laboratory of High Performance Computing and Stochastic Information
		Processing (HPCSIP) (Ministry of Education of China), College of Mathematics and Statistics, Hunan Normal University,
		Changsha, Hunan 410081, China.}
	\email{ngai@bimsa.cn}

\author[S.-H. Zhou]{Shui-Hong Zhou}
\address{Key Laboratory of High Performance Computing and Stochastic Information
		Processing (HPCSIP) (Ministry of Education of China), College of Mathematics and Statistics, Hunan Normal University, Changsha, Hunan 410081, China.}
\email{zsh20171101117@163.com}

\date{\today}

\thanks{The authors are supported in part by the National Natural Science Foundation of China, grant 11771136, and Construct Program of the Key Discipline in Hunan Province. }

	\begin{abstract}
	  This paper extends the Hodge-de Rham theory of Aaron \textit{et al.} [Commun. Pure Appl. Anal. {\bf 13} (2014)] to higher-dimensional level-$l$ Sierpinski gaskets $SG_{\ell}^{n},$ providing a framework for analyzing differential forms and Laplacians on these fractal structures. We construct a sequence of graphs approximating $SG_{\ell}^{n}$ and define $k$-forms, de Rham derivatives, and their duals on these graphs. We prove that the extension of a $1$-form on a generation-$m$ graph to a $1$-form on a generation-$(m+1)$ graph is harmonic. We obtain a basis for the space of harmonic $1$-forms. We also explore the properties of $2$-forms on the level-$3$ Sierpinski gasket, under the assumptions that the $2$-forms are absolutely continuous with respect to the Kusuoka measure or the standard self-similar measure and that the Radon-Nikodym derivatives are continuous.
		
	\end{abstract}
	
	\maketitle

\section{INTRODUCTION}\label{S:IN}
	\setcounter{equation}{0}

In \cite{Aaron-Conn-Strichartz-Yu_2014}, Aaron \textit{et al.} introduced Hodge-de Rham theory for post critically finite (p.c.f) self-similar sets. They constructed $k$-forms, de Rham derivative $d_{k}$ and dual de Rham derivative $\delta_{k}$ on fractal graphs, and proved that the Laplacian defined on $0$-forms by using the Hodge-de Rham theory is equivalent to the Laplacian defined by Kigami (see \cite{Kigami_2001}). Kigami (see \cite{Kigami_1989}, \cite{Kigami_1993}) defined the Laplacian on the p.c.f. self-similar sets by using a renormalized limit of finite matrix Laplacians which is equivalent to its probabilistic definition as the infinitesimal generator of Brownian motion (see \cite{Barlow-Perkins_1988}, \cite{Goldstein_1987}, \cite{Kusuoka_1987}, \cite{Telcs_1989}). Aaron \textit{et al.} proved that for the Sierpinski gasket and the $3$-dimensional Sierpinski gasket, the extension of a $1$-form on a generation-$m$ graph to a $1$-form on a generation-$(m+1)$ graph is harmonic, and analyzed harmonic $1$-forms by constructing a basis for the space of harmonic $1$-forms. They also studied two classes of $2$-forms on the Sierpinski gasket, one with absolutely continuous Radon-Nikodym derivatives, and the other absolutely continuous with respect to the Kusuoka measure.

The main goal of this paper is to extend the work in Aaron \textit{et al.} to higher-dimensional level-$l$ Sierpinski gaskets $SG_{\ell}^{n}$ $($for $n,\ell\geq 2).$ By exploring potential extensions of Hodge-de Rham theory to higher-dimensional fractals, we hope this investigation may yield generalizations of existing results while contributing to a deeper understanding of spectral and geometric properties in higher dimensions. $SG_{\ell}^{n}$ are constructed by iterated function systems (IFSs) $\{F_{i}\}_{i=1}^{N_{\ell}^{n}}$ on $\mathbb{R}^{n},$ where $N_{\ell}^{n}$ is the number of the IFS maps on $SG_{\ell}^{n}.$ Let $G_{\ell}^{n,0}$ be an initial graph of $SG_{\ell}^{n}$ in $\mathbb{R}^{n},$ $E_{\ell,0}^{n,0}:=\{q_{i}\}$ be the set of vertices, and $E_{\ell,1}^{n,0}$ be the set of edges. Let $E_{\ell,0}^{n,1}:=\{F_{i}q_{j}\}$ be the image of the points in $E_{\ell,0}^{n,0}$ under the IFS maps, let $G_{\ell}^{n,1}$ be the graph with points in $E_{\ell,0}^{n,1}$ as vertices and the images of edges of $G_{\ell}^{n,0}$ under $\{F_{i}\}$ as its edges. Iterating this procedure we obtain a sequence of graphs $G_{\ell}^{n,1}, G_{\ell}^{n,2},...,$ and $SG_{\ell}^{n}$ is the limit of the graph $G_{\ell}^{n,m}.$ We let $E_{\ell,0}^{n,m},E_{\ell,1}^{n,m},E_{\ell,2}^{n,m}$ denote the sets of vertices $e_{0}^{m},$ edges $e_{1}^{m},$ and triangles $e_{2}^{m}$ on the graph $G_{\ell}^{n,m},$ respectively.

On $SG_{\ell}^{n},$ for any $m\geq 1,$ we define
$$W_{m}:=\Big \{1,2,...,N_{\ell}^{n}\Big \}^{m}=\Big \{\omega_{1}\omega_{2}\cdots\omega_{m}:\omega_{i}\in\{1,2,...,N_{\ell}^{n}\}\Big \}.$$
An element $\omega\in W_{m}$ is called a \textit{word} of length $m,$ and we denote $|\omega|$ as the \textit{length} of $\omega.$ For any $\omega\in W_{m},$ we define
$$F_{\omega}:=F_{\omega_{1}}\circ F_{\omega_{2}}\circ\cdots\circ F_{\omega_{m}}.$$

On $SG_{\ell}^{n},$ a $0$-form is a real-valued function defined on the vertex set $E_{\ell,0}^{n,m},$ a $1$-form is a measure on the edge set, and a $2$-form is a measure on the set of triangles. The de Rham derivative $d_{k}^{m}$ is a map from the space of $k$-forms to the space of $(k+1)$-forms, and $\delta_{k}^{m}$ is a map from the space of $k$-forms to the space of $(k-1)$-forms, as defined in Section \ref{S:2} Using these two operators, one can define the Laplacian
$$-\Delta_{k}^{m}=\delta_{k+1}^{m}d_{k}^{m}+d_{k-1}^{m}\delta_{k}^{m}$$
on $k$-forms. A $k$-form $f_{k}^{m}$ satisfying the conditions $-\Delta_{k}^{m}f_{k}^{m}=0$ is called a harmonic $k$-form. Let $\mathcal{H}_{\ell,k}^{n,m}$ denote the space of all harmonic $k$-forms on the graph $G_{\ell}^{n,m}.$ Cipriani \textit{et al.} \cite{Cipriani-Guido-Isola-Sauvageot_2013} introduced the concept of $1$-forms and their integrals on the paths over the Sierpinski gasket. They proved the Hodge decomposition theorem for $1$-forms and de Rham's first theorem on the space of harmonic $1$-forms. These results will be used in our study of harmonic $1$-forms.

By using the values of harmonic $1$-forms on the edges and the harmonic extension algorithm on $SG_{\ell}^{n},$ we obtain the following theorem:

\begin{thm}\label{thm:1.1}
For any $h_{1}^{m}\in\mathcal{H}_{\ell,1}^{n,m},$ let $h_{1}^{m+1}$ be a $1$-form obtained by extending $h_{1}^{m}$ onto the graph $G_{\ell}^{n,m+1}.$ Then $h_{1}^{m+1}$ is a harmonic $1$-form on $G_{\ell}^{n,m+1}$ and extends $h_{1}^{m}$ in the sense that
\begin{equation}
h_{1}^{m}(e_{1}^{m})=\sum_{e_{1}^{m+1}\subset e_{1}^{m}}h_{1}^{m+1}(e_{1}^{m+1}).
\end{equation}
\end{thm}

Compared to Sierpinski gasket, the structure of $SG_{\ell}^{n}$ is more complex. The challenges in proving Theorem \ref{thm:1.1} arise from the structural complexity, the explicit construction of harmonic extensions, and the complexity of symbols and calculations. To address these problems, we introduce a refined classification of vertices into six distinct sets that capture the hierarchical characteristics of $SG_{\ell}^{n}.$ Furthermore, we derive a harmonic extension matrix by solving a linear system of equations. This allows us to obtain an explicit expression for a harmonic function at any point in $E_{\ell,0}^{n,1}.$ By combining these with the properties of harmonic $1$-forms, we establish exact formulas for edge values. To give a unified framework across different dimensions and levels, we break down the process into several propositions.

We define a harmonic $1$-form $h$ on the graph $G_{\ell}^{n,1},$ by using the harmonic extension algorithm and the IFS maps on $SG_{\ell}^{n}.$ Define a new harmonic $1$-form on each cell by
\begin{equation*}
h_{\omega}(e):=\begin{cases}
h\circ F_{\omega}^{-1}(e),\quad &\text{$e\in F_{\omega}E_{\ell,1}^{n,1}$};\\
0,\quad &\text{elsewhere}.
\end{cases}
\end{equation*}

The following is our second main theorem.
\begin{thm}\label{thm:1.2}
On $SG_{\ell}^{n},$ for any $h_{\omega},h_{\omega^{\prime}}\in\mathcal{H}_{\ell,1}^{n,m},$ if $\omega\neq\omega^{\prime},$ then
$$\langle h_{\omega},h_{\omega^{\prime}}\rangle_{1}^{m}=\sum_{e_{1}^{m}\in E_{1,1}^{n,m}}\mu_{1}(e_{1}^{m})h_{\omega}(e_{1}^{m})\overline{h_{\omega^{\prime}}(e_{1}^{m})}=0.$$
Moreover, $\{h_{\omega}\}$ is an orthogonal basis for $\mathcal{H}_{\ell,1}^{n,m}.$
\end{thm}

To prove Theorem \ref{thm:1.2}, we first construct a harmonic $1$-form $h_{\omega}$ from $h,$ verifying its harmonicity through the harmonic conditions. Next we introduce a detailed decomposition of the edge set to handle the interplay between dimension and level, and adapt $1$-cycle integration techniques to verify linear independence in this hierarchical setting. By showing that the sum of integrals vanishes only for trivial combinations, we prove that $\{h_{\omega}\}$ forms a basis.

This paper is organized as follows. Section \ref{S:2} provides the definitions of higher-dimensional level-$l$ Sierpinski gaskets, de Rham operators, dual de Rham operators and the Laplacian on $k$-forms along with some preliminary results. Section \ref{S:3} is devoted to the proofs of Theorems \ref{thm:1.1} and \ref{thm:1.2}. In Section \ref{S:4}, we explore the properties of $2$-forms on the level-$3$ Sierpinski gasket, extending our analysis to higher dimensional settings.

\maketitle

\section{ HIGHER-DIMENSIONAL LEVEL-$L$ SIERPI\'{N}SKI GASKET}\label{S:2}
\setcounter{equation}{0}
\setcounter{figure}{0}

The higher-dimensional level-$\ell$ Sierpinski gaskets $K=SG_{\ell}^{n}$ $($for $n,\ell\geq 2)$ are self-similar fractals in $\mathbb{R}^{n}$ constructed by $N_{\ell}^{n}$ contractive similitudes $\{F_{i}\}_{i=1}^{N_{\ell}^{n}},$ where $F_{i}(x)=x/\ell+b_{i,l,n}.$ Here, $b_{i,\ell,n}$ are translation vectors corresponding to the retained subgraph at each step, and $N_{\ell}^{n}$ is the number of the IFS maps on $SG_{\ell}^{n}$ $($\textit{e.g.} $N_{\ell}^{2}=\frac{\ell(\ell+1)}{2}$ for the level-$\ell$ Sierpinski gasket in $\mathbb{R}^{2},$ $N_{3}^{3}=10$ for the $3$-dimensional level-$3$ Sierpinski gasket$).$ $SG_{\ell}^{n}$ is the unique nonempty compact set that satisfies the self-similar identity
$$SG_{\ell}^{n}=\bigcup_{i=1}^{N_{\ell}^{n}}F_{i}(SG_{\ell}^{n}).$$
For fixed level $\ell,$ the structure of $SG_{\ell}^{n}$ is constructed by first considering $SG_{\ell}^{n-1}$ embedded in $\mathbb{R}^{n-1}$ and then adding a point in $\mathbb{R}^{n}$ orthogonal to each of its $n$-simplices, and finally connecting this point to all vertices of the corresponding in $(n-1)$-simplices. The value of $N_{\ell}^{n}$ is given by the recurrence relation:
\begin{equation}\label{eq:N1}
\begin{split}
N_{\ell}^{n}&=N_{\ell}^{n-1}+N_{\ell-1}^{n}\\
&=N_{\ell}^{2}+N_{\ell-1}^{3}+N_{\ell-1}^{4}+\cdots+N_{\ell-1}^{n},
\end{split}
\end{equation}
reflecting the combinatorial interplay between dimension and level. Figures \ref{fig:sg2} and \ref{fig:sg3} illustrate $SG_{\ell}^{n}$ for the cases $n=2,$ $3,$  $\ell=2,$ $3,$ $4,$ showing the hierarchical geometric structure.

\begin{figure}[htbp]
  \centering
  \begin{tikzpicture}[scale=2]
    \draw (0,0) --(1,0) --(0.5,0.86) --(0,0);
    \draw [black, fill=black] (0,0) --(0.5,0) --(0.25,0.43) --(0,0);
    \draw [black, fill=black](0.5,0) --(1,0) --(0.75,0.43) --(0.5,0);
    \draw [black, fill=black] (0.25,0.43) --(0.75,0.43)--(0.5,0.86);
  \end{tikzpicture}\qquad\qquad
  \begin{tikzpicture}[scale=2]
    \draw (0,0) --(1,0)--(0.5,0.86) --(0,0);
    \draw (1/3,0) --(1/6,0.28) --(0.5,0.28) --(1/3,0);
    \draw (0.5,0.28) --(5/6,0.28) --(2/3,0) --(0.5,0.28);
    \draw (1/3,0.56) --(2/3,0.56) --(0.5,0.28) --(1/3,0.56);
    \draw [black, fill=black] (0,0) --(1/3,0) --(1/6,0.28);
    \draw [black, fill=black] (1/3,0) --(2/3,0) --(0.5,0.28);
    \draw [black, fill=black] (2/3,0) --(1,0) --(5/6,0.28);
    \draw [black, fill=black] (1/6,0.28) --(0.5,0.28) --(1/3,0.56);
    \draw [black, fill=black] (0.5,0.28) --(5/6,0.28) --(2/3,0.56);
    \draw [black, fill=black] (0.5,0.86) --(2/3,0.56) --(1/3,0.56);
  \end{tikzpicture}\qquad\qquad
  \begin{tikzpicture}[scale=2]
    \draw (0,0) --(1,0)--(0.5,0.86) --(0,0);
    \draw (1/4,0) --(1/8,0.22) --(3/8,0.22) --(5/8,0.22) --(7/8,0.22) --(3/4,0) --(5/8,0.22) --(1/2,0) --(3/8,0.22) --(1/4,0);
    \draw (3/8,0.22) --(1/4,0.43) --(0.5,0.43) --(3/4,0.43) --(5/8,0.22) --(0.5,0.43) --(3/8,0.22);
    \draw (1/2,0.43) --(3/8,0.65) --(5/8,0.65) --(0.5,0.43);
    \draw [black, fill=black] (0,0) --(1/4,0) --(1/8,0.22);
    \draw [black, fill=black] (1/4,0) --(1/2,0) --(3/8,0.22);
    \draw [black, fill=black] (1/2,0) --(3/4,0) --(5/8,0.22);
    \draw [black, fill=black] (3/4,0) --(1,0) --(7/8,0.22);
    \draw [black, fill=black] (1/8,0.22) --(3/8,0.22) --(1/4,0.43);
    \draw [black, fill=black] (3/8,0.22) --(5/8,0.22) --(5/8,0.22);
    \draw [black, fill=black] (3/8,0.22) --(5/8,0.22) --(0.5,0.43);
    \draw [black, fill=black] (5/8,0.22) --(7/8,0.22) --(3/4,0.43);
    \draw [black, fill=black] (1/4,0.43) --(0.5,0.43) --(3/8,0.65);
    \draw [black, fill=black] (0.5,0.43) --(3/4,0.43) --(5/8,0.65);
    \draw [black, fill=black] (3/8,0.65) --(5/8,0.65) --(0.5,0.86);
  \end{tikzpicture}\qquad\qquad
  \caption
{First iteration of the IFS defining $SG_{\ell}^{2}$ for $\ell=2,3,4.$}
\label{fig:sg2}

\end{figure}
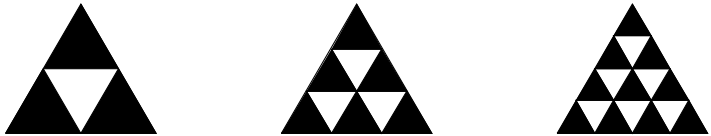

\begin{figure}[htbp]
  \centering
  \begin{subfigure}[b]{0.26\textwidth}
    \includegraphics[width=\linewidth]{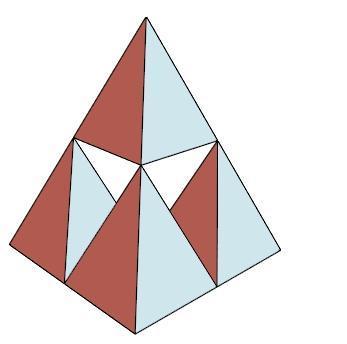}
    \label{fig:sub1}
  \end{subfigure}
  \quad
  \begin{subfigure}[b]{0.25\textwidth}
    \includegraphics[width=\linewidth]{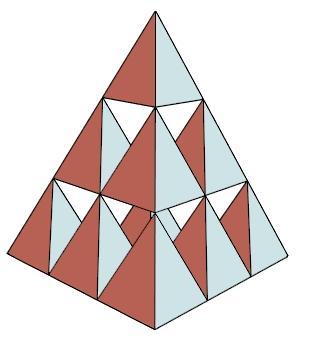}
    \label{fig:sub2}
  \end{subfigure}
  \quad
  \begin{subfigure}[b]{0.25\textwidth}
    \includegraphics[width=\linewidth]{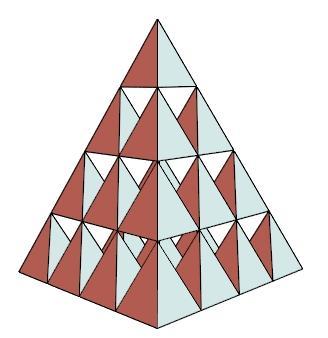}
    \label{fig:sub3}
  \end{subfigure}
\caption{First iteration of the IFS defining $SG_{\ell}^{3}$ for $\ell=2,3,4.$}
\label{fig:sg3}
\end{figure}
By equation (\ref{eq:N1}), we have the following proposition. To prove this, we will use the following standard identities (see \cite{Graham-Knuth-Patashnik_1994}, \cite{Stanley_1997})
\begin{equation}\label{eq:C1}
\sum_{i=0}^{\ell}\binom{k+i}{k}=\binom{k+\ell+1}{k+1},\text{ } \binom{n+1}{k+1}=\binom{n}{k}+\binom{n}{k+1},\text{ } \binom{n}{k}=\frac{n}{k}\binom{n-1}{k-1}.
\end{equation}
\begin{prop}\label{pro:N1}
For any $n,\ell\geq 2,$ we have
\begin{equation}\label{N1}
N_{\ell}^{n}=\binom{n+\ell-1}{n}.
\end{equation}
\end{prop}
\begin{proof}
As we all know, $N_{\ell}^{2}=\frac{\ell(\ell+1)}{2}=\binom{\ell+1}{2}$ and $N_{2}^{n}=n+1=\binom{n+1}{n}.$ When $n=3,$ by equations (\ref{eq:N1}) and (\ref{eq:C1}), we have
\begin{equation*}
\begin{split}
N_{\ell}^{3}&=N_{2}^{3}+N_{3}^{2}+N_{4}^{2}+\cdots+N_{\ell}^{2}\\
&=\sum_{j=0}^{\ell-1}\binom{2+j}{2}=\binom{2+\ell}{3}
\end{split}
\end{equation*}
Assume that for any $n\leq k,$ we have
$$N_{\ell}^{n}=\binom{n+\ell-1}{n}$$
for all $\ell\geq 2.$

When $n=k+1,$ we have
\begin{equation*}
\begin{split}
N_{\ell}^{k+1}&=N^{k+1}_{2}+N_{3}^{k}+N_{4}^{k}+\cdots+N_{\ell}^{k}\\
&=\sum_{i=0}^{\ell-1}\binom{k+i}{k}=\binom{k+\ell}{k+1}.
\end{split}
\end{equation*}
\end{proof}

When $l=0,$ we denote $G_{0}^{n,m}$ as the graph consisting of only one point in $\mathbb{R}^{n};$ when $l=1,$ we define $G_{1}^{n,m}:=G_{1}^{n,0}$ as the initial graph in $\mathbb{R}^{n}.$ For $k\geq 1,$ let $E_{\ell,k}^{n,m}$ be the collection of complete $(k+1)$-subgraphs of $G_{\ell}^{n,m}.$ For $e_{k}^{m}\in E_{\ell,k}^{n,m}$ and $e_{k+1}^{m}\in E_{\ell,k+1}^{n,m},$ $e_{k}^{m}\subset e_{k+1}^{m}$ means that the vertices of $e_{k}^{m}$ are all vertices of $e_{k+1}^{m}.$

\begin{defi}\label{def:2.1}
A \textit{parity function} ${\rm sgn}(e_{k}^{m},e_{k+1}^{m})$ is a function, defined for all $e_{k}^{m}\in E_{\ell,k}^{n,m}$ and $e_{k+1}^{m}\in E_{\ell,k+1}^{n,m}$ $(k\leq m-1),$ taking values in $\{-1,0,1\},$ and satisfying:
\begin{equation}
{\rm sgn}(e_{k}^{m},e_{k+1}^{m})\neq 0\quad\Leftrightarrow\quad e_{k}^{m}\subset e_{k+1}^{m}
\end{equation}
and
\begin{equation}\label{eq:2.1}
\sum_{e_{k}^{m}\in E_{\ell,k}^{n,m}}{\rm sgn}(e_{k-1}^{m},e_{k}^{m}){\rm sgn}(e_{k}^{m},e_{k+1}^{m})=0
\end{equation}
for all $e_{k-1}^{m}\subset e_{k+1}^{m}.$
\end{defi}

On $SG_{\ell}^{n},$ we use the right-hand rule to define the positive orientation for the edges and the triangles. Specifically, we assign $[q_{0},q_{1},q_{2}]$ as the reference positive orientation on $SG_{\ell}^{n},$  which remains consistent across all dimensions and levels of the construction.  Therefore, for any $e_{0}^{m}\in E_{\ell,0}^{n,m}$ and $e_{1}^{m}\in E_{\ell,1}^{n,m},$ if $e_{0}^{m}\subset e_{1}^{m},$ then
$${\rm sgn}(e_{0}^{m},e_{1}^{m}):=\begin{cases}
-1,\quad &\text{$e_{0}^{m}$ is the starting point of $e_{1}^{m}$};\\
1,\quad &\text{$e_{0}^{m}$ is the endpoint of $e_{1}^{m}$}.
\end{cases}$$
We define ${\rm sgn}(e_{0}^{m},e_{1}^{m})=0$ if $e_{0}^{m}$ is not an endpoint of $e_{1}^{m}.$ Assume that the edge $e_{1}^{m}=[x,y]$ is oriented from $x$ to $y.$ If $e_{1}^{m}$ lies along the boundary of the triangle $e_{2}^{m}=[x,y,z]$ and its orientation agrees with $e_{2}^{m},$ namely, the cyclic order $x\rightarrow y\rightarrow z\rightarrow x,$ then we say that the orientations of $e_{1}^{m}$ and $e_{2}^{m}$ are consistent, otherwise, their orientations are inconsistent. Similarly, for any $e_{1}^{m}\in E_{\ell,1}^{n,m}$ and $e_{2}^{m}\in E_{\ell,2}^{n,m},$ if $e_{1}^{m}\subset e_{2}^{m},$ then
$${\rm sgn}(e_{1}^{m},e_{2}^{m}):=\begin{cases}
1,\quad &\text{the orientations of $e_{1}^{m}$ and $e_{2}^{m}$ are consistent};\\
-1,\quad &\text{the orientations of $e_{1}^{m}$ and $e_{2}^{m}$ are inconsistent}.
\end{cases}$$

\begin{defi}
A set of \textit{weights} on $E_{\ell,k}^{n,m}$ is an assignment of positive values $\mu_{k}(e_{k}^{m})$ for each $e_{k}^{m}\in E_{\ell,k}^{n,m}.$
\end{defi}

\begin{defi}
A \textit{$k$-form} $f_{k}^{m}$ is a function from $E_{\ell,k}^{n,m}$ to $\mathbb C.$ We let $\mathcal{D}_{\ell,k}^{n,m}$ be the space of all $k$-forms on $G_{\ell}^{n,m}.$ Define an inner product on $\mathcal{D}_{\ell,k}^{n,m}$ as
\begin{equation}\label{eq:2.2}
\langle f_{k}^{m},g_{k}^{m}\rangle_{k}^{m}:=\sum_{e_{k}^{m}\in E_{\ell,k}^{n,m}}\mu_{k}(e_{k}^{m})f_{k}^{m}(e_{k}^{m})\overline{g_{k}^{m}(e_{k}^{m})}.
\end{equation}
\end{defi}

\begin{defi}
For a nonnegative integer $k\leq n-1,$ the \textit{de Rham derivative} $d_{k}^{m}:\mathcal{D}_{\ell,k}^{n,m}\rightarrow \mathcal{D}_{\ell,k+1}^{n,m}$ is defined as
\begin{equation}\label{eq:2.3}
d_{k}^{m}f_{k}^{m}(e_{k+1}^{m}):=\sum_{e_{k}^{m}\in E_{\ell,k}^{n,m}}{\rm sgn}(e_{k}^{m},e_{k+1}^{m})f_{k}^{m}(e_{k}^{m})
\end{equation}
and we define $d_{n}^{m}f_{n}^{m}:=0.$

For an integer $k\geq 1,$ the \textit{dual de Rham derivative} $\delta_{k}^{m}:\mathcal{D}_{\ell,k}^{n,m}\rightarrow \mathcal{D}_{\ell,k-1}^{n,m}$ is defined by
\begin{equation}
\langle\delta_{k}^{m}f_{k}^{m},g_{k-1}^{m}\rangle_{k-1}^{m}:=\langle f_{k}^{m},d_{k-1}^{m}g_{k-1}^{m}\rangle_{k}^{m}
\end{equation}
and we define $\delta_{0}^{m}f_{0}^{m}:=0.$
\end{defi}

The following explicit formula for $\delta_{k}^{m}$ follows from direct computation:
\begin{equation}\label{eq:2.6}
\delta_{k}^{m}f_{k}^{m}(e_{k-1}^{m})=\sum_{e_{k}^{m}\in E_{\ell,k}^{n,m}}\frac{\mu_{k}(e_{k}^{m})}{\mu_{k-1}(e_{k-1}^{m})}{\rm sgn}(e_{k-1}^{m},e_{k}^{m})f_{k}^{m}(e_{k}^{m}).
\end{equation}

On $SG_{\ell}^{n},$ we define a de Rham complex
$$0{\longrightarrow}\mathcal{D}_{\ell,0}^{n,m}\stackrel{d_{0}^{m}}{\longrightarrow}\cdots\stackrel{d_{n-1}^{m}}{\longrightarrow}\mathcal{D}_{\ell,n}^{n,m}\stackrel{d_{n}^{m}}{\longrightarrow}0.$$
The relation
\begin{equation}\label{eq:2.7}
d_{k}^{m}d_{k-1}^{m}\equiv 0
\end{equation}
is a consequence of $($\ref{eq:2.1}$).$

To describe the dual de Rham derivative $\delta_{k}^{m}$ and the dual de Rham complex, we need to choose the weights on $E_{\ell,k}^{n,0}$ in order to compute the weights on $E_{\ell,k}^{n,m}.$ Given a set of weights $\mu_{k}^{0}$ on $E_{\ell,k}^{n,0}$ and a collection of positive numbers $\{b_{k}^{j}\}_{j=1}^{N_{\ell}^{n}},$ we can define a set of weights on $E_{\ell,k}^{n,m}$ as: for any $e_{k}^{m}=F_{\omega}e_{k}^{0}$ with the word $\omega=\omega_{1}\cdots\omega_{m}$ of length $m,$ we have the following equation
\begin{equation}\label{eq:2.8}
\mu_{k}(e_{k}^{m}):=\begin{cases}
b_{k}^{\omega}\mu_{k}^{0}(e_{k}^{0}),\quad &\text{$k\geq 1$},\\
\sum_{F_{\omega}e_{0}^{0}=e_{0}^{m}}b_{k}^{\omega}\mu_{0}^{0}(e_{0}^{0}),\quad &\text{$k=0$},
\end{cases}
\end{equation}
where $b_{k}^{\omega}=\prod_{j=1}^{m}b_{k}^{\omega_{j}}$ with each $\omega_{j}\in\{1,2,...,N_{\ell}^{n}\}.$

According to (\ref{eq:2.6}) and (\ref{eq:2.8}), we have an explicit expression for $\delta_{k}^{m}$ and the dual de Rham complex
$$0\stackrel{\delta_{0}^{m}}{\longleftarrow}\mathcal{D}_{\ell,0}^{n,m}\stackrel{\delta_{1}^{m}}{\longleftarrow}\cdots\stackrel{\delta_{n}^{m}}{\longleftarrow}\mathcal{D}_{\ell,n}^{n,m}{\longleftarrow} 0,$$
with the operators
$$\delta_{k}^{m}f_{k}^{m}(e_{k-1}^{m})=\sum_{e_{k}^{m}\in E_{\ell,k}^{n,m}}\frac{\mu_{k}(e_{k}^{m})}{\mu_{k-1}(e_{k-1}^{m})}{\rm sgn}(e_{k-1}^{m},e_{k}^{m})f_{k}^{m}(e_{k}^{m}).$$
The relation
\begin{equation}\label{eq:2.11}
\delta_{k}^{m}\delta_{k+1}^{m}\equiv 0
\end{equation}
follows from $($\ref{eq:2.1}$).$

Thus, we can define a de Rham cohomology space
$$H^{k}:={\rm ker}{(d_{k}^{m})}/{{\rm im}{(d_{k-1}^{m})}}.$$

\begin{defi}
The \textit{energy} $\mathcal{E}_{k}^{m}$ is a symmetric bilinear form on $\mathcal{D}_{\ell,k}^{n,m}$ defined as
\begin{equation}
\mathcal{E}_{k}^{m}(f_{k}^{m},g_{k}^{m}):=\langle d_{k}^{m}f_{k}^{m},d_{k}^{m}g_{k}^{m}\rangle_{k+1}^{m}+\langle\delta_{k}^{m}f_{k}^{m},\delta_{k}^{m}g_{k}^{m}\rangle_{k-1}^{m}
\end{equation}
For the cases $k=0$ or $k=n,$ the right-hand side of the equation reduces to a single term.
\end{defi}
\begin{defi}\label{def:2.7}
For $1\leq m\leq n-1,$ \textit{Hodge Laplacian} $\Delta_{k}^{m}$ is a symmetric operator on $\mathcal{D}_{\ell,k}^{n,m},$ defined by
\begin{equation}
-\Delta_{k}^{m}:=\delta_{k+1}^{m}d_{k}^{m}+d_{k-1}^{m}\delta_{k}^{m}.
\end{equation}
In particular, when $k=0,$ $-\Delta_{0}^{m}=\delta_{1}^{m}d_{0}^{m};$ when $k=n,$ $-\Delta_{n}^{m}=d_{n-1}^{m}\delta_{n}^{m}.$
If $f_{k}^{m}\in \mathcal{D}_{\ell,k}^{n,m}$ satisfies $-\Delta_{k}^{m}f_{k}^{m}=0,$ then we call $f_{k}^{m}$ a harmonic $k$-form. We denote by $\mathcal{H}_{\ell,k}^{n,m}$ the space of all harmonic $k$-forms on $\mathcal{D}_{\ell,k}^{n,m}.$
\end{defi}
The following theorem appears as Theorem 2.2 in \cite{Aaron-Conn-Strichartz-Yu_2014}.
\begin{thm}
$-\Delta_{k}^{m}$ is the operator associated to $\mathcal{E}_{k}^{m},$ namely,
$$\langle -\Delta_{k}^{m}f_{k}^{m},g_{k}^{m}\rangle_{k}^{m}=\mathcal{E}_{k}^{m}(f_{k}^{m},g_{k}^{m})$$
and this characterizes $-\Delta_{k}^{m}.$

In particular, the following conditions are equivalent:
\begin{enumerate}[label=(\roman*),itemsep=0pt,parsep=0pt]
    \item $f_{k}^{m}$ is harmonic;
    \item $\mathcal{E}_{k}^{m}(f_{k}^{m},f_{k}^{m})=0;$
    \item
    \begin{equation*}
    \begin{cases}
        d_{k}^{m}f_{k}^{m}=0\quad\text{and}\quad\delta_{k}^{m}f_{k}^{m}=0, & \text{if } 1\leq k\leq n-1;\\
        d_{0}^{m}f_{0}^{m}=0,                                              & \text{if } k=0;\\
        \delta_{n}^{m}f_{n}^{m}=0,                                         & \text{if } k=n.
    \end{cases}
    \end{equation*}
\end{enumerate}

\end{thm}
By $(\ref{eq:2.3})$ and $(\ref{eq:2.6}),$ we can compute $-\Delta_{0}^{m}=\delta_{1}^{m}d_{0}^{m}$ explicitly as:
\begin{equation}
\begin{split}
\delta_{1}^{m}d_{0}^{m}f_{0}^{m}(e_{0}^{m})&=\sum_{e_{1}^{m}\in E_{\ell,1}^{n,m}}{\rm sgn}(e_{0}^{m},e_{1}^{m})d_{0}^{m}f_{0}^{m}(e_{1}^{m})\\
&=\sum_{e_{1}^{m}\in E_{\ell,1}^{n,m}}\frac{\mu_{1}(e_{1}^{m})}{\mu_{0}(e_{0}^{m})}{\rm sgn}(e_{0}^{m},e_{1}^{m})\sum_{e_{0}^{m^{\prime}}\in E_{\ell,0}^{n,m}}{\rm sgn}(e_{0}^{m^{\prime}},e_{1}^{m})f_{0}^{m}(e_{0}^{m^{\prime}})\\
&=\sum_{[e_{0}^{m},e_{0}^{m^{\prime}}]\in E_{\ell,1}^{n,m}}\frac{\mu_{1}(e_{1}^{m})}{\mu_{0}(e_{0}^{m})}\Big (f_{0}^{m}(e_{0}^{m})-f_{0}^{m}(e_{0}^{m^{\prime}})\Big ).
\end{split}
\end{equation}
The first equality follows from $(\ref{eq:2.6}),$ the second equality follows from $(\ref{eq:2.3}),$ and the last equality follows from the fact that each edge $e_{1}^{m}$ has two vertices $e_{0}^{m}$ and $e_{0}^{m^{\prime}}$ with ${\rm sgn}(e_{0}^{m},e_{1}^{m})=-{\rm sgn}(e_{0}^{m^{\prime}},e_{1}^{m})$. See Theorem 2.3 in \cite{Aaron-Conn-Strichartz-Yu_2014} for the following theorem.

\begin{thm}[Hodge Decomposition]
For each $k$ and $m,$ we have
\begin{equation}
\mathcal{D}_{\ell,k}^{n,m}=d_{k-1}^{m}\mathcal{D}_{\ell,k-1}^{n,m}\oplus\delta_{k+1}^{m}\mathcal{D}_{\ell,k+1}^{n,m}\oplus\mathcal{H}_{\ell,k}^{n,m}
\end{equation}
as an orthogonal direct sum. Thus $H^{k}$ is isomorphic to $\mathcal{H}_{\ell,k}^{n,m}.$
\label{thm:2.9}
\end{thm}

\begin{defi}
A \textit{$k$-chain} $C_{\ell,k}^{n,m}$ is a formal sum $C_{\ell,k}^{n,m}=\sum_{e_{k}^{m}\in E_{\ell,k}^{n,m}}a_{k}^{m}e_{k}^{m}$ with $a_{k}^{m}\in\mathbb C.$ We denote the collection of $k$-chains by $\mathcal{C}_{\ell,k}^{n,m}.$
\end{defi}
\begin{defi}\label{def:2.10}
The boundary operator $\partial_{k}:\mathcal{C}_{\ell,k}^{n,m}\rightarrow\mathcal{C}_{\ell,k-1}^{n,m}$ is defined as
\begin{equation}
\partial_{k}e_{k}^{m}:=\sum_{e_{k-1}^{m}\in E_{\ell,k}^{n,m}}{\rm sgn}(e_{k-1}^{m},e_{k}^{m})e_{k-1}^{m}
\end{equation}
and
\begin{equation}
\partial_{k}\Big (\sum_{e_{k}^{m}\in E_{\ell,k}^{n,m}}a_{k}e_{k}^{m}\Big ):=\sum_{e_{k}^{m}\in E_{\ell,k}^{n,m}}a_{k}\partial_{k}e_{k}^{m}.
\end{equation}
If $\partial_{k}C_{\ell,k}^{n,m}=0,$ then we call $C_{\ell,k}^{n,m}$ \textit{$k$-cycle}. Integration of $k$-forms along $k$-chains is defined as
\begin{equation}\label{eq:2.16}
\int_{C_{\ell,k}^{n,m}}f_{k}^{m}:=\sum_{e_{k}^{m}\in E_{\ell,k}^{n,m}}a_{k}f_{k}^{m}(e_{k}^{m}).
\end{equation}
\end{defi}
By the definition of the boundary operator, we know that $\mathcal{D}_{\ell,k}^{n,m}$ and $\mathcal{C}_{\ell,k}^{n,m}$ are dual. The following theorem holds (see \cite{Aaron-Conn-Strichartz-Yu_2014}, Theorem 2.4).
\begin{thm}[Stokes' theorem]
Let $f_{k-1}^{m}$ be a $k$-form, and $C_{\ell,k}^{n,m}$ be a $k$-chain. Then
\begin{equation}
\int_{C_{\ell,k}^{n,m}}df_{k-1}^{m}=\int_{\partial_{k}C_{\ell,k}^{n,m}}f_{k}^{m}.
\end{equation}

In particular, if $C_{\ell,k}^{n,m}$ is a $k$-cycle, then
\begin{equation}
\int_{C_{\ell,k}^{n,m}}df_{k-1}^{m}=0.
\end{equation}
\label{thm:2.10}
\end{thm}

By Definition \ref{def:2.1}, the identity $\sum_{e_{k}^{m}\in E_{\ell,k}^{n,m}}{\rm sgn}(e_{k-1}^{m},e_{k}^{m}){\rm sgn}(e_{k}^{m},e_{k+1}^{m})=0$ implies $\partial_{k-1}\partial_{k}=0.$ Therefore, the sequence
$$\mathcal{C}_{\ell,n}^{n,m}\stackrel{\partial_{n}}{\longrightarrow}\mathcal{C}_{\ell,n-1}^{n,m}\stackrel{\partial_{n-1}}{\longrightarrow}\cdots\stackrel{\partial_{1}}
{\longrightarrow}\mathcal{C}_{\ell,0}^{n,m}\stackrel{\partial_{0}}{\longrightarrow}0$$
forms a chain complex.

\begin{defi}\label{def:1}
For any $x,$ $y\in E_{\ell,0}^{n,m},$ if $x$ is connected to $y,$ we denote this as $x\sim y$ (or more precisely, $x\stackrel{G_{\ell}^{n,m}}\sim y$); otherwise, we write $x\nsim y.$ We define $\deg(x),$ the \textit{degree} of the vertex $x,$ as
$$\deg(x):=\#\{y\in E_{\ell,0}^{n,m}:y\sim x \text{ in } G_{\ell}^{n,m}\}.$$
\end{defi}

\maketitle

\section{THEORY OF $1$-FORMS }\label{S:3}
\setcounter{equation}{0}
\setcounter{figure}{0}
We denote $V_{0}$ the collection of the boundary points on $G_{\ell}^{n,0},$ and define
$$V_{\ell,1}^{n,m}:=V_{0}.$$
According to the structure of the graph $G_{\ell}^{n,m},$ when $n\geq \ell,$ the set $E_{\ell,0}^{n,m}$ can be defined as
\begin{equation}\label{E1}
E_{\ell,0}^{n,m}:=\bigcup_{k=1}^{\ell}V_{\ell,k}^{n,m},
\end{equation}
and when $\ell>n,$ the set can be defined as
\begin{equation}\label{E2}
E_{\ell,0}^{n,m}:=\bigcup_{k=1}^{n+1}V_{\ell,k}^{n,m},
\end{equation}
where for all $i\neq j, V_{\ell,i}^{n,m}\cap V_{\ell,j}^{n,m}=\emptyset.$ Moreover, for $k\geq 2,$
\begin{equation*}
\begin{aligned}
V_{\ell,k}^{n,m}:=\{e_{0}^{m}\in E_{\ell,0}^{n,m}:\exists k \text{ distinct }&F_{\omega_{j}} \text{ with }|\omega_{j}|=m\text{ such that } \\ &e_{0}^{m}=F_{\omega_{j}}q_{j_{\ell}} \text{ for some } q_{j_{\ell}}\in V_{0}\}.
\end{aligned}
\end{equation*}
A vertex $e_{0}^{m}\in V_{\ell,k}^{n,m}$ represents the intersection of $k$ distinct $e_{n}^{m}\in E_{\ell,n}^{n,m}$ in the graph $G_{\ell}^{n,m}.$ For example, when $\ell=4,$ $n=2$ and $k=3,$ we have $V_{0}=\{q_{i}\}_{i=0}^{2},$ and the set $V_{4,3}^{2,1}$ consists of all points satisfying $e_{0}^{1}=F_{\omega_{1}}q_{0}=F_{\omega_{2}}q_{1}=F_{\omega_{3}}q_{2}.$ $V_{4,3}^{2,1}$ describes all triple intersections of distinct triangles $e_{2}^{1}\in E_{4,2}^{2,1}$ in the graph $G_{4}^{2,1}.$ Formally,
\begin{equation*}
\begin{split}
V_{4,3}^{2,1}=\{e_{0}^{1}\in E_{4,0}^{2,1}: \text{ there exist } 3 \text{ distinct }&F_{\omega_{j}} \text{ with } |\omega_{j}|=1 \text{ such that } \\ &e_{0}^{1}=F_{\omega_{1}}q_{0}=F_{\omega_{2}}q_{1}=F_{\omega_{3}}q_{2}\}.
\end{split}
\end{equation*}

Let $M_{\ell}^{n,m}$ denote the cardinality of vertex set $E_{\ell,0}^{n,m}.$ By the structure of $SG_{\ell}^{n}$ and equations (\ref{E1}) and (\ref{E2}), for any $n,\ell\geq 2,$ we obtain
\begin{equation}\label{v1}
M_{\ell}^{n,m}=\begin{cases}
N_{\ell}^{n}\times M_{\ell}^{n,m-1}-\sum_{k=2}^{n+1}\binom{n+1}{k}(k-1)M_{\ell-k}^{k-1,1}, \quad & \ell>n;\\
N_{\ell}^{n}\times M_{\ell}^{n,m-1}-\sum_{k=2}^{\ell}\binom{n+1}{k}(k-1)M_{\ell-k}^{k-1,1}, \quad & \ell\leq n;
\end{cases}
\end{equation}
where we define $M_{\ell}^{n,0}:=n+1,$ $M_{\ell}^{1,1}:=\ell+1,$ $M_{0}^{n,1}:=1$ and $M_{1}^{n,1}:=n+1.$ By equation (\ref{v1}), we know it is corresponded with the $M_{\ell}^{n,m-1}.$ Hence, when $m=1,$ we have the following proposition.
\begin{prop}\label{prop:M}
On the graph $G_{\ell}^{n,1},$ for any $n\geq 2,$ we have
\begin{equation}\label{v2}
M_{\ell}^{n,1}=M_{\ell}^{n-1,1}+M_{\ell-1}^{n,1},
\end{equation}
and
\begin{equation}\label{v3}
M_{\ell}^{n,1}=\binom{n+\ell}{n}.
\end{equation}
for all $\ell\geq 1.$
\end{prop}
\begin{proof}
By definition, we obtain
\begin{equation}\label{eq:l1}
M_{0}^{n,1}=1=\binom{n}{n},
\end{equation}
\begin{equation}\label{eq:l3}
M_{1}^{n,1}=n+1=\binom{n+1}{n},
\end{equation}
and
\begin{equation}\label{eq:l2}
M_{\ell}^{1,1}=\ell+1=\binom{\ell+1}{1}.
\end{equation}

\noindent\textit{Step $1.$} Let $n=2.$ We show that for all $\ell\geq 1,$ equations (\ref{v2}) and (\ref{v3}) hold.

When $\ell=1,$ by equations (\ref{eq:l1}) and $(\ref{eq:l3}),$ equations (\ref{v2}) and (\ref{v3}) hold since
\begin{equation}\label{eq:21}
M_{1}^{1,1}+M_{0}^{2,1}=\binom{2}{1}+\binom{2}{2}=\binom{3}{2}=M_{1}^{2,1}.
\end{equation}
When $\ell=2,$ by equation $(\ref{v1})$ and Proposition \ref{pro:N1}, we have
\begin{equation*}
M_{2}^{2,1}=N_{2}^{2}\times M_{2}^{2,0}-\binom{3}{2}M_{0}^{1,1}=3\times 3-3=6.
\end{equation*}
By equations (\ref{eq:l2})and (\ref{eq:21}), we have
\begin{equation}\label{eq:22}
M_{1}^{2,1}+M_{2}^{1,1}=\binom{3}{2}+\binom{3}{1}=\binom{4}{2}=6=M_{2}^{2,1}.
\end{equation}

For $\ell>2,$ by equation $(\ref{v1}),$ we have
\begin{align*}
M_{\ell}^{2,1}&=N_{\ell}^{2}\times M_{\ell}^{2,0}-\sum_{k=2}^{3}\binom{3}{k}(k-1)M_{\ell-k}^{k-1,1}\\
&=3N_{2}^{\ell}-3(\ell-1)-2M_{\ell-3}^{2,1}\quad (\text{by } M_{\ell}^{2,0}=3, M_{\ell}^{1,1}=l+1)\\
&=3N_{\ell}^{2}-3(\ell-1)-2\Big(N_{2}^{\ell-3}\times M_{\ell-3}^{2,0}-\binom{3}{2}M_{\ell-5}^{1,1}-2\binom{3}{3}M_{\ell-6}^{2,1}\Big)\quad (\text{by equation (\ref{v1})})\\
&=3N_{\ell}^{2}-3(\ell-1)-6N_{2}^{\ell-3}+6(\ell-4)+4M_{\ell-6}^{2,1}\\
&=3N_{\ell}^{2}-3(\ell-1)-6N_{2}^{\ell-3}+6(\ell-4)+4(3N_{\ell-6}^{2}-3(\ell-7)-2M_{\ell-9}^{2,1})\quad (\text{by (\ref{v1})})\\
&=3N_{\ell}^{2}-2N_{\ell-3}^{2}+4N_{\ell-6}^{2}-3\Big((\ell-1)-2(\ell-4)+4(\ell-7)\Big)-8M_{\ell-9}^{2,1}\\
&=3\sum_{k=0}^{\lfloor \ell/3 \rfloor-1}(-2)^{k}N_{\ell-3k}^{2}-3\sum_{k=0}^{\lfloor \ell/3 \rfloor-1}(-2)^{k}(\ell-3k-1)+(-2)^{\lfloor \ell/3 \rfloor}M_{\ell-3\lfloor l/3 \rfloor}^{2,1}\\
\end{align*}
Since $N_{\ell}^{2}=\frac{\ell(\ell+1)}{2},$ for any $a\geq 1,$ we have
\begin{equation}\label{eq:2a}
M_{\ell}^{2,1}=\begin{cases}
\frac{1}{2}(2+9a+9a^{2}), \quad & \text{if $\ell=3a;$}\\
\frac{3}{2}(2+5a+3a^{2}), \quad & \text{if $\ell=3a+1;$}\\
\frac{3}{2}(4+7a+3a^{2}), \quad & \text{if $\ell=3a+2;$}\\
\end{cases}
=\frac{(\ell+1)(\ell+2)}{2}=\binom{\ell+2}{2}.
\end{equation}
Therefore, combining equations (\ref{eq:21}), (\ref{eq:22}) and (\ref{eq:2a}), for all $\ell\geq 1,$ we have
\begin{equation}\label{eq:2l}
M_{\ell}^{2,1}=\binom{\ell+2}{2}.
\end{equation}
From equations (\ref{eq:l2}) and (\ref{eq:2l}), we obtain
$$M_{\ell}^{1,1}+M_{\ell-1}^{2,1}=\binom{\ell+1}{1}+\binom{\ell+1}{2}=\binom{\ell+2}{2}=M_{\ell}^{2,1}.$$

\noindent\textit{Step $2.$} Assume that for some $k\geq 2,$ equations (\ref{v2}) and (\ref{v3}) hold for all $n$ satisfying $2\leq n\leq k$ and all $\ell\geq 1.$ Let $n=k+1.$ We claim that for any $\ell\geq 1,$ equations (\ref{v2}) and (\ref{v3}) hold.

We use induction on $\ell.$ In fact, when $\ell=1,$ by equations (\ref{eq:l1}) and (\ref{eq:l3}), we have $M_{1}^{k+1,1}=k+2$ and $M_{0}^{k+1,1}=1.$ By assumption in Step $2,$ we obtain $M_{1}^{k,1}=\binom{k+1}{k}=k+1.$ Hence, $M_{0}^{k+1}+M_{1}^{k,1}=1+k+1=k+2=M_{1}^{k+1,1}.$

Assume that for any $p\geq 1,$ we have
$$M_{\ell}^{k+1,1}=M_{\ell}^{k,1}+M_{\ell-1}^{k+1,1}$$
for all $\ell$ satisfying $\ell\leq p.$

When $\ell=p+1,$ we consider the following cases to prove the claim.

\noindent\textit{Case $1.$} $p\leq k-1.$ By equation (\ref{eq:N1}), we have
\begin{align*}
M_{p+1}^{k,1}+M_{p}^{k+1,1}&=N_{p+1}^{k}\times M_{p+1}^{k,0}-\sum_{i=2}^{p+1}\binom{k+1}{i}\times M_{(p+1)-i}^{i-1,1}\times (i-1)\\
&\quad+N_{p}^{k+1}\times M_{p}^{k+1,0}-\sum_{i=2}^{p}\binom{k+2}{i}\times M_{p-i}^{i-1,1}\times (i-1)\\
&=(k+1)N_{p+1}^{k}-\sum_{i=2}^{p+1}\binom{k+1}{i}\times M_{(p+1)-i}^{i-1,1}\times (i-1)\\
&\quad+(k+2)N_{p}^{k+1}-\sum_{i=2}^{p}\binom{k+2}{i}\times M_{p-i}^{i-1,1}\times (i-1)\\
&=(k+2)N_{p+1}^{k+1}-N_{p+1}^{k}-\sum_{i=2}^{p+1}\binom{k+2}{i}\times M_{(p+1)-i}^{i-1,1}\times (i-1)\\
&\quad+\sum_{i=2}^{p+1}\binom{k+2}{i}\times M_{(p+1)-i}^{i-1,1}\times (i-1)-\sum_{i=2}^{p+1}\binom{k+1}{i}\times M_{(p+1)-i}^{i-1,1}\times (i-1)\\
&\quad-\sum_{i=2}^{p}\binom{k+2}{i}\times M_{p-i}^{i-1,1}\times (i-1)\\\
&=M_{p+1}^{k+1,1}-N_{p+1}^{k}\\
&\quad+\sum_{i=2}^{p}(i-1)\Bigg(\Big(\binom{k+2}{i}-\binom{k+1}{i}\Big)\times M_{(p+1)-i}^{i-1,1}-\binom{k+2}{i}\times M_{p-i}^{i-1,1}\Bigg)\\
&\quad+p\Big(\binom{k+2}{p+1}-\binom{k+1}{p+1}\Big)\times M_{(p+1)-(p+1)}^{p,1}\\
&=M_{p+1}^{k+1,1}-N_{p+1}^{k}\\
&\quad+\sum_{i=2}^{p}(i-1)\Bigg(\binom{k+1}{i-1}\times (M_{(p+1)-i}^{i-1,1}-M_{p-i}^{i-1,1})-\binom{k+1}{i}\times M_{p-i}^{i-1,1}\Bigg)\\
&\quad+p\Big(\binom{k+2}{p+1}-\binom{k+1}{p+1}\Big)\times M_{(p+1)-(p+1)}^{p,1}\quad(\text{by equation (\ref{eq:C1})})\\
&=M_{p+1}^{k+1,1}-N_{p+1}^{k}\\
&\quad+\sum_{i=2}^{p}(i-1)\Bigg(\binom{k+1}{i-1}\times M_{(p+1)-i}^{i-2,1}-\binom{k+1}{i}\times M_{p-i}^{i-1,1}\Bigg)\\
&\quad+p\Big(\binom{k+2}{p+1}-\binom{k+1}{p+1}\Big)\times M_{(p+1)-(p+1)}^{p,1}\quad(\text{by assumption in Step $2$})\\
&=M_{p+1}^{k+1,1}-N_{p+1}^{k}\\
&\quad+\sum_{i=2}^{p}(i-1)\binom{k+1}{i-1}\times M_{(p+1)-i}^{i-2,1}-\sum_{i=2}^{p}(i-1)\binom{k+1}{i}\times M_{p-i}^{i-1,1}\\
&\quad+p\Big(\binom{k+2}{p+1}-\binom{k+1}{p+1}\Big)\times M_{(p+1)-(p+1)}^{p,1}\\
&=M_{p+1}^{k+1,1}-N_{p+1}^{k}\\
&\quad+(k+1)\sum_{i=2}^{p}\binom{k}{i-2}\times M_{(p+1)-i}^{i-2,1}-(k+1)\sum_{i=2}^{p}\binom{k}{i-1}\times M_{p-i}^{i-1,1}\\
&\quad+\sum_{i=2}^{p}\binom{k+1}{i}\times M_{p-i}^{i-1,1}+p\Big(\binom{k+2}{p+1}-\binom{k+1}{p+1}\Big)\times M_{(p+1)-(p+1)}^{p,1}\\
\end{align*}
The last equality follows from equation (\ref{eq:C1}). By equations (\ref{v3}) and (\ref{eq:l1}), the above expression equals
\begin{align*}
&\quad M_{p+1}^{k+1,1}-N_{p+1}^{k}\\
&\quad+(k+1)\sum_{i=2}^{p}\binom{k}{i-2}\times\binom{p-1}{j}-(k+1)\sum_{i=2}^{p}(k+1)\binom{k}{i-1}\times\binom{p-1}{j}\\
&\quad+\sum_{i=2}^{p}\binom{k+1}{i}\times\binom{p-1}{i-1}+p\Big(\binom{k+2}{p+1}-\binom{k+1}{p+1}\Big)\\
&=M_{p+1}^{k+1,1}-N_{p+1}^{k}\\
&\quad+(k+1)\sum_{i=0}^{p-2}\binom{k}{j}\times\binom{p-1}{j}-(k+1)\sum_{i=1}^{p-1}\binom{k}{i-1}\times\binom{p-1}{j}\\
&\quad+\sum_{i=2}^{p}\binom{k+1}{i}\times\binom{p-1}{i-1}+p\Big(\binom{k+2}{p+1}-\binom{k+1}{p+1}\Big)\\
&=M_{p+1}^{k+1,1}-N_{p+1}^{k}\\
&\quad+(k+1)\Bigg(-\binom{k}{p-1}+\binom{k+p-1}{k}\Bigg)-(k+1)\Bigg(-1+\binom{k+p-1}{k}\Bigg)\\
&\quad+\Bigg(-(k+1)+\binom{k+p}{k}\Bigg)+p\binom{k+1}{p}\quad(\text{by Vandermonde's Identity})\\
&=M_{p+1}^{k+1,1}-N_{p+1}^{k}-(k+1)\binom{k}{p-1}+\binom{k+p}{p}+(k+1)\binom{k}{p-1}\quad(\text{by (\ref{eq:C1})})\\
&=M_{p+1}^{k+1,1}-\binom{k+p}{p}+\binom{k+p}{p}\quad(\text{by Proposition \ref{pro:N1}})\\
&=M_{p+1}^{k+1,1}.
\end{align*}

\noindent\textit{Case $2.$} $p=k.$ The proof is essentially identical to that of Case $1.$

\noindent\textit{Case $3.$} $p=k+1.$ The proof is essentially identical to that of Case $1.$

\noindent\textit{Case $4.$} $p>k+1.$ We use induction on $p.$

\noindent(\textrm{i}) Let $k+1<p\leq 2(k+1).$ By equation (\ref{v1}), we have
\begin{align*}
&\qquad M_{p+1}^{k,1}+M_{p}^{k+1,1}\\&=N_{p+1}^{k}\times M_{p+1}^{k,0}-\sum_{i=2}^{k+1}\binom{k+1}{i}\times M_{(p+1)-i}^{i-1,1}\times (i-1)\\
&\quad+N_{p}^{k+1}\times M_{p}^{k+1,0}-\sum_{i=2}^{k+2}\binom{k+2}{i}\times M_{p-i}^{i-1,1}\times (i-1)\\
&=(k+1)N_{p+1}^{k}-\sum_{i=2}^{k+1}\binom{k+1}{i}\times M_{(p+1)-i}^{i-1,1}\times (i-1)\\
&\quad+(k+2)N_{p}^{k+1}-\sum_{i=2}^{k+2}\binom{k+2}{i}\times M_{p-i}^{i-1,1}\times (i-1)\\
&=M_{p+1}^{k+1,1}-N_{p+1}^{k}+\sum_{i=2}^{k+1}(i-1)\Bigg(\Big(\binom{k+2}{i}-\binom{k+1}{i}\Big)\times M_{(p+1)-i}^{i-1,1}-\binom{k+2}{i}\times M_{p-i}^{i-1,1}\Bigg)\\
&\quad+(k+1)\Big(\binom{k+2}{k+2}\times M_{(p+1)-(k+2)}^{k+1,1}-\binom{k+2}{k+2}\times M_{p-(k+2)}^{k+1,1}\Big)\quad(\text{by equation (\ref{v1})})\\
&=M_{p+1}^{k+1,1}-N_{p+1}^{k}+\sum_{i=2}^{k}(i-1)\Bigg(\Big(\binom{k+2}{i}-\binom{k+1}{i}\Big)\times M_{(p+1)-i}^{i-1,1}-\binom{k+2}{i}\times M_{p-i}^{i-1,1}\Bigg)\\
&\quad+k\Bigg(\Big(\binom{k+2}{k+1}-\binom{k+1}{k+1}\Big)\times M_{(p+1)-(k+1)}^{k,1}-\binom{k+2}{k+1}\times M_{p-(k+1)}^{k,1}\Bigg)\\
&\quad+(k+1)\Big(\binom{k+2}{k+2}\times M_{(p+1)-(k+2)}^{k+1,1}-\binom{k+2}{k+2}\times M_{p-(k+2)}^{k+1,1}\Big)\\
&=M_{p+1}^{k+1,1}-N_{p+1}^{k}\\
&+\sum_{i=2}^{k}(i-1)\Bigg(\Big(\binom{k+2}{i}-\binom{k+1}{i}\Big)\times M_{(p+1)-i}^{i-1,1}-\binom{k+2}{i}\times M_{p-i}^{i-1,1}\Bigg)\\
&+k\Big(\binom{k+2}{k+1}-\binom{k+1}{k+1}\Big)\times M_{(p+1)-(k+1)}^{k,1}-\Big(k\binom{k+2}{k+1}-(k+1)\binom{k+2}{k+2}\Big)\times M_{p-(k+1)}^{k,1}\\
&+(k+1)\binom{k+2}{k+2}\times M_{(p+1)-(k+2)}^{k+1,1}-(k+1)\binom{k+2}{k+2}\times M_{p-(k+1)}^{k+1,1}.\\
\end{align*}
The last equality follows from Cases $1-3$ in Step $2.$ Since the coefficients of $M_{(p+1)-(k+2)}^{k+1,1}$ and $M_{p-(k+1)}^{k+1,1}$ are equal, we have
\begin{align*}
&\quad M_{p+1}^{k+1,1}-N_{p+1}^{k}+\sum_{i=2}^{k-1}(i-1)\Bigg(\Big(\binom{k+2}{i}-\binom{k+1}{i}\Big)\times M_{(p+1)-i}^{i-1,1}-\binom{k+2}{i}\times M_{p-i}^{i-1,1}\Bigg)\\
&\quad+(k-1)\Bigg(\Big(\binom{k+2}{k}-\binom{k+1}{k}\Big)\times M_{(p+1)-k}^{k-1,1}-\binom{k+2}{k}\times M_{p-k}^{k-1,1}\Bigg)\\
&\quad+k\Big(\binom{k+2}{k+1}-\binom{k+1}{k+1}\Big)\times M_{(p+1)-(k+1)}^{k,1}-\Big(k\binom{k+2}{k+1}-(k+1)\binom{k+2}{k+2}\Big)\times M_{p-(k+1)}^{k,1}\\
&=M_{p+1}^{k+1,1}-N_{p+1}^{k}+\sum_{i=2}^{k-1}(i-1)\Bigg(\Big(\binom{k+2}{i}-\binom{k+1}{i}\Big)\times M_{(p+1)-i}^{i-1,1}-\binom{k+2}{i}\times M_{p-i}^{i-1,1}\Bigg)\\
&\quad+(k-1)\Big(\binom{k+2}{k}-\binom{k+1}{k}\Big)\times M_{(p+1)-k}^{k-1,1}\\
&\quad-\Bigg((k-1)\binom{k+2}{k}-\Big(k\binom{k+2}{k+1}-(k+1)\binom{k+2}{k+2}\Big)\Bigg)\times M_{p-k}^{k-1,1}\\
&\quad+k\Big(\binom{k+2}{k+1}-\binom{k+1}{k+1}\Big)\times M_{(p+1)-(k+1)}^{k,1}-\Big(k\binom{k+2}{k+1}-(k+1)\binom{k+2}{k+2}\Big)\times M_{p-k}^{k,1}.\\
\end{align*}
The last equality follows from the assumption in Step $2.$ Iterating the procedure under the given hypothesis in Step $2,$ we find that the last expression equals
\begin{align*}
&\quad M_{p+1}^{k+1,1}-N_{p+1}^{k}+\sum_{i=2}^{k-2}(i-1)\Bigg(\Big(\binom{k+2}{i}-\binom{k+1}{i}\Big)\times M_{(p+1)-i}^{i-1,1}-\binom{k+2}{i}\times M_{p-i}^{i-1,1}\Bigg)\\
&\quad+(k-2)\Bigg(\Big(\binom{k+2}{k-1}-\binom{k+1}{k-1}\Big)\times M_{(p+1)-(k-1)}^{k-2,1}-\binom{k+2}{k-1}\times M_{p-(k-1)}^{k-2,1}\Bigg)\\
&\quad+(k-1)\Big(\binom{k+2}{k}-\binom{k+1}{k}\Big)\times M_{(p+1)-k}^{k-1,1}\\
&\quad-\Bigg((k-1)\binom{k+2}{k}-k\binom{k+2}{k+1}+(k+1)\binom{k+2}{k+2}\Bigg)\times M_{p-k}^{k-1,1}\\
&\quad+\Big(-k\binom{k+1}{k+1}+(k+1)\binom{k+2}{k+2}\Big)\times M_{(p+1)-(k+1)}^{k,1}\\
&=M_{p+1}^{k+1,1}-N_{p+1}^{k}+\Bigg(\binom{k+2}{2}-\binom{k+1}{2}\Bigg)\times M_{(p+1)-2}^{1,1}-\Bigg(\binom{k+2}{2}-2\binom{k+2}{3}\\
&\qquad+3\binom{k+2}{4}-4\binom{k+2}{5}+\cdots+(-1)^{k+2}(k+1)\binom{k+2}{k+2}\Bigg)\times M_{p-2}^{1,1}\\
&\quad+\Bigg(-2\binom{k+1}{3}+3\binom{k+2}{4}-4\binom{k+2}{5}+\cdots+(-1)^{k+2}(k+1)\binom{k+2}{k+2}\Bigg)\times M_{p-2}^{2,1}\\
&\quad+\Bigg(-3\binom{k+1}{4}+4\binom{k+2}{5}-5\binom{k+2}{6}+\cdots+(-1)^{k+1}(k+1)\binom{k+2}{k+2}\Bigg)\times M_{p-3}^{3,1}\\
&\quad+\cdots+\Bigg(-k\binom{k+1}{k+1}+(k+1)\binom{k+2}{k+2}\Bigg)\times M_{p-k}^{k,1}\\
&=M_{p+1}^{k+1,1}-N_{p+1}^{k}+\Bigg(\binom{k+2}{2}-\binom{k+1}{2}\Bigg)\times M_{(p+1)-2}^{1,1}-\Bigg(\binom{k+2}{2}-2\binom{k+2}{3}\\
&\qquad+3\binom{k+2}{4}-4\binom{k+2}{5}+\cdots+(-1)^{k+2}(k+1)\binom{k+2}{k+2}\Bigg)\times M_{p-2}^{1,1}\\
&\quad+\sum_{\substack{i\in\{3,...k+1\}\\\text{$i$ is odd}}}\Big[-(i-1)\binom{k+1}{i}+\sum_{j=i+1}^{k+2}(-1)^{j}(j-1)\binom{k+2}{j}\Big]\times M_{(p+1)-i}^{i-1,1}\\
&\quad+\sum_{\substack{i\in\{3,...k+1\}\\\text{$i$ is even}}}\Big[-(i-1)\binom{k+1}{i}+\sum_{j=i+1}^{k+2}(-1)^{j+1}(j-1)\binom{k+2}{j}\Big]\times M_{(p+1)-i}^{i-1,1}\\
&=M_{p+1}^{k+1,1}-N_{p+1}^{k}+(k+1)\times M_{(p+1)-2}^{1,1}-M_{p-2}^{1,1}+\sum_{i=3}^{k+1}\binom{k}{i-1}\times M_{(p+1)-i}^{i-1,1}\\
&=M_{p+1}^{k+1,1}-N_{p+1}^{k}+(k+1)\binom{p}{1}-\binom{p-1}{1}+\sum_{i=3}^{k+1}\binom{k}{i-1}\times\binom{p}{i-1}\quad(\text{by (\ref{v3})})\\
&=M_{p+1}^{k+1,1}-\binom{k+p}{k}+\sum_{i=0}^{k}\binom{k}{i}\times\binom{p}{i}\quad(\text{by Proposition \ref{pro:N1}})\\
&=M_{p+1}^{k+1,1}-\binom{k+p}{k}+\binom{k+p}{k}\quad(\text{by Vandermonde's Identity})\\
&=M_{p+1}^{k+1,1}.
\end{align*}

\noindent(\textrm{ii}) Let $a\geq 1$ and $b=a+1.$ Assume that for any $p$ satisfying $a(k+1)<p\leq b(k+1),$ equations (\ref{v2}) and (\ref{v3}) hold for $n=k+1.$

\noindent(\textrm{iii}) Let $b(k+1)<p\leq (b+1)(k+1).$ Using equation (\ref{v1}) and the same method as in (\textrm{i}), we know that (\ref{v2}) holds.
By equation (\ref{v2}), we obtain
\begin{equation*}
\begin{split}
M_{\ell}^{k+1,1}&=M_{\ell}^{k,1}+M_{\ell-1}^{k,1}+\cdots+M_{2}^{k,1}+M_{1}^{k+1,1}\\
&=\binom{k+\ell}{k}+\binom{k+\ell-1}{k}+\cdots+\binom{k+2}{k}+(k+2)\\
&=\sum_{i=0}^{\ell}\binom{k+\ell-i}{k}=\binom{k+\ell+1}{k+1}.\quad(\text{by equation (\ref{eq:C1})})
\end{split}
\end{equation*}
This completes the proof for $n=k+1.$
\end{proof}

To extend a $0$-form $f_{0}^{1}$ to a $0$-form $f_{0}^{2},$ we need to use the harmonic extension algorithm associated with $SG_{\ell}^{n}.$ For example, the rule on $SG_{2}^{2}$ is $\frac{1}{5}-\frac{2}{5},$ the rule on $SG_{3}^{2}$ is $\frac{3}{15}-\frac{4}{15}-\frac{8}{15},$ and the rule on $SG_{3}^{3}$ is $\frac{11}{64}-\frac{14}{64}-\frac{28}{64}.$ Let $h$ be a harmonic function on $SG_{\ell}^{n}.$ We denote the values of the points $\{q_{i}\}_{i=0}^{n}$ by $\{h(q_{i})\}_{i=0}^{n}.$ For any $x\in E_{\ell,0}^{n,1}\setminus V_{0},$ we have
\begin{equation}\label{eq:3.1}
h(x)=\frac{1}{\deg{(x)}}\sum_{x\sim y}h(y).
\end{equation}
If the vertex $x$ is connected to $q_{i}\in V_{0},$ then $(\ref{eq:3.1})$ can be written as
\begin{equation}\label{eq:3.2}
\deg{(x)}h(x)-\sum_{\substack{x\sim y\\y\notin V_{0}}}h(y)=\sum_{\substack{x\sim y\\y\in V_{0}}}h(q_{i}).
\end{equation}
If the vertex $x$ is not connected to any point in $V_{0},$ then $(\ref{eq:3.1})$ can be written as
\begin{equation}\label{eq:3.3}
\deg{(x)}h(x)-\sum_{\substack{x\sim y\\y\notin V_{0}}}h(y)=0.
\end{equation}
For any point $x\in E_{\ell,0}^{n,1}\setminus V_{0},$ using either equation $(\ref{eq:3.2})$ or $(\ref{eq:3.3}),$ we obtain a linear system of equations
$$AX=B,$$
where $A=(a_{ij})$ is defined as follows:
\begin{equation*}
a_{ij}:=\begin{cases}
-1, \quad & i\neq j \text{ and } p_{i}\sim p_{j};\\
0, \quad & i\neq j \text{ and } p_{i}\nsim p_{j};\\
\deg{(x)}, \quad & i=j.
\end{cases}
\end{equation*}
The vector $B=(b_{i})$ is a column vector defined as follows: If the point $p_{i}$ is connected with any point $q_{i}\in V_{0},$ then $b_{i}=\sum_{p_{i}\sim q_{i}}h(q_{i});$ otherwise $b_{i}=0.$ Since $A$ is an irreducible weakly diagonally dominant matrix (recall that an $n\times n$ matrix $A=(a_{ij})$ is called \textit{weakly diagonally dominant} if $|a_{ii}|\geq\sum_{j\neq i}|a_{ij}|$ for $i=1,...,n,$ with at least one inequality be in strict), it follows that $A$ is invertible. Therefore, the linear system $AX=B$ has a unique solution. This implies that for any $x\in E_{\ell,0}^{n,1}\setminus V_{0},$ the value $h(x)$ can be expressed as a linear combination of $\{h(q_{i})\}_{i=0}^{n}$ yielding a harmonic extension algorithm on $SG_{\ell}^{n}.$

To implement this algorithm, we first derive an explicit expression for the harmonic extension at an arbitrary point  in $E_{\ell,0}^{n,1}\setminus V_{0}$ $(\text{for }\ell>n).$ For any point $e_{0}^{1}\in V_{\ell,2}^{n,1}$connected to the point $q_{i}\in V_{0},$ there exist some $q_{j}$ and $F_{k}$ such that $e_{0}^{1}=F_{i}q_{j}=F_{k}q_{i}.$ By solving the linear system $AX=B,$ we obtain
\begin{equation}\label{eq:3.4}
h(e_{0}^{1})=a_{1}h(q_{i})+a_{2}h(q_{j})+a_{3}\sum_{\tilde{\ell}=0, \tilde{\ell}\neq i,j}^{n}h(q_{\tilde{\ell}}),
\end{equation}
where $0<a_{3}<a_{2}<a_{1}<1$ and $a_{1}+a_{2}+(n-1)a_{3}=1.$

Similarly, for each point $e_{0}^{1}\in V_{\ell,2}^{n,1}$ not connected to any point in $V_{0},$ if $e_{0}^{1}$ is not the midpoint of the edge $[q_{i},q_{j}],$ there exist some $F_{k_{1}}$ and $F_{k_{2}}$ $(k_{1}<k_{2})$ such that $e_{0}^{1}=F_{k_{1}}q_{i}=F_{k_{2}}q_{j},$ then we have
\begin{equation}\label{eq:3.5}
h(e_{0}^{1})=\bar{a}_{1}h(q_{i})+\bar{a}_{2}h(q_{j})+\bar{a}_{3}\sum_{\tilde{\ell}\neq i,j}h(q_{\tilde{\ell}}),
\end{equation}
where $0<\bar{a}_{3}<\bar{a}_{2}<\bar{a}_{1}<1$ and $\bar{a}_{1}+\bar{a}_{2}+(n-1)\bar{a}_{3}=1.$

For each point $e_{0}^{1}\in V_{\ell,2}^{n,1}$ not connected to any point in $V_{0},$ if $e_{0}^{1}$ is the midpoint of the edge $[q_{i},q_{j}],$ there exist some $F_{k_{1}}$ and $F_{k_{2}}$ $(k_{1}<k_{2})$ such that $e_{0}^{1}=F_{k_{1}}q_{i}=F_{k_{2}}q_{j},$ then we have
\begin{equation}\label{eq:3.6}
h(e_{0}^{1})=\hat{a}_{1}(h(q_{i})+h(q_{j}))+\hat{a}_{2}\sum_{\tilde{\ell}\neq i,j}h(q_{\tilde{\ell}}),
\end{equation}
where $0<\hat{a}_{2}<\hat{a}_{1}<1$ and $2\hat{a}_{1}+(n-1)\hat{a}_{2}=1.$

Next we discuss the points in $V_{\ell,\alpha}^{n,1},$ where $3\leq\alpha\leq n+1.$ Before this, we give the following proposition to discuss the set $V_{\ell,\alpha}^{n,1}.$

\begin{prop}\label{1}
On $SG_{\ell}^{n},$ for any $\alpha\geq 3,$ $V_{\ell,\alpha}^{n,m}$ is the union of $\sigma_{\ell-\alpha}^{\alpha-1,m}$ vertex sets of the graph $G_{\ell-\alpha}^{\alpha-1,1},$ where
\begin{equation*}
\sigma_{\ell-\alpha}^{\alpha-1,m}=M_{\ell-\alpha}^{\alpha-1,1}\sum_{k=0}^{m-1}(N_{\ell}^{n})^{k}
=M_{\ell-\alpha}^{\alpha-1,1}\times\frac{(N_{\ell}^{n})^{m}-1}{N_{\ell}^{n}-1}.
\end{equation*}
In particular, when $\ell=\alpha=n+1,$ $V_{\ell,\alpha}^{n,m}$ is the union of $\sigma_{0}^{\alpha-1,m}$ singletons of $G_{0}^{\alpha-1,0};$
when $\ell=\alpha+1,$ $V_{\ell,\alpha}^{n,m}$ is the union of $\sigma_{1}^{\alpha-1,m}$ vertex sets of the graph $G_{1}^{\alpha-1,0}.$
\end{prop}
\begin{proof}
\noindent\textit{Step $1.$} Let $n=2.$ We first claim that for any $\ell\geq 3,$ $V_{\ell,\alpha}^{2,m}$ is the union of $\sigma_{\ell-\alpha}^{\alpha-1,m}$ vertex sets of the graph $G_{\ell-\alpha}^{\alpha-1,1}.$

Since $3\leq\alpha\leq n+1,$ when $n=2,$ we have $\alpha=3;$ therefore, we perform induction on $\ell.$ For $\ell=3,$ $SG_{3}^{2}$ is a level-$3$ Sierpinski gasket. By the iterating progress, $V_{3,3}^{2,1}$ is a singleton corresponding to $G_{0}^{2,1}$ (see Figure \ref{fig:1}). Since $N_{3}^{2}=6,$ iterating the process, we have
$$\sigma_{0}^{2,m}=1+6+6^{2}+\cdots+6^{m-1}=\frac{6^{m}-1}{5}.$$
Hence, $V_{3,3}^{2,m}$ is the union of $\sigma_{0}^{2,m}$ singleton sets.
\begin{figure}[b]
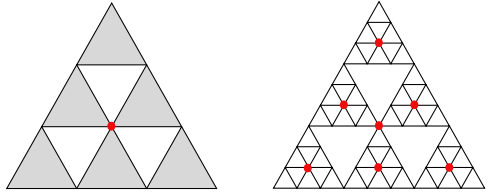

\centering


\caption{$G_{3}^{2,1}$ and $G_{3}^{2,2}$.}
\label{fig:1}
\end{figure}

For $\ell=4,$ from Figure \ref{fig:2}, we know that $V_{4,3}^{2,1}$ is the vertex set of the graph $G_{1}^{2,0}.$ By the iterating process and $N_{4}^{2}=10,$ we have
$$\sigma_{1}^{2,m}=3+3\times 10+3\times 10^{2}+\cdots+3\times 10^{m-1}=3\times\frac{10^{m}-1}{9}.$$
Hence, $V_{4,3}^{2,m}$ is the union of $\sigma_{1}^{2,m}$ copies of $E_{1,1}^{0,2}.$

\begin{figure}[htbp]
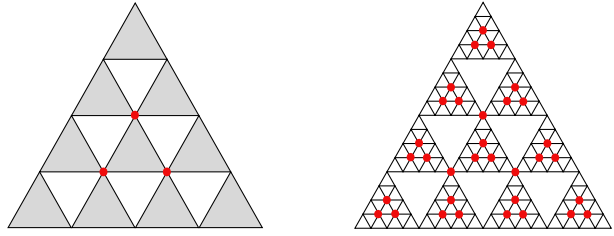

\centering

%
\caption{$G_{4}^{2,1}$ and $G_{4}^{2,2}$.}
\label{fig:2}
\end{figure}

For $\ell=5,$ from Figure \ref{fig:3}, we obtain that $V_{5,3}^{2,1}$ is the vertex set of the graph $G_{2}^{2,1}.$ Since $N_{5}^{2}=15$ and $M_{2}^{2,1}=6,$ continuing the iterating process, we have
$$\sigma_{2}^{2,m}=6+6\times 15+\cdots+6\times 15^{m-1}=6\times\frac{15^{m}-1}{14}.$$
Hence, $V_{5,3}^{2,m}$ is the union of $\sigma_{2}^{2,m}$ copies of $E_{2,0}^{2,1}.$

\begin{figure}[htbp]
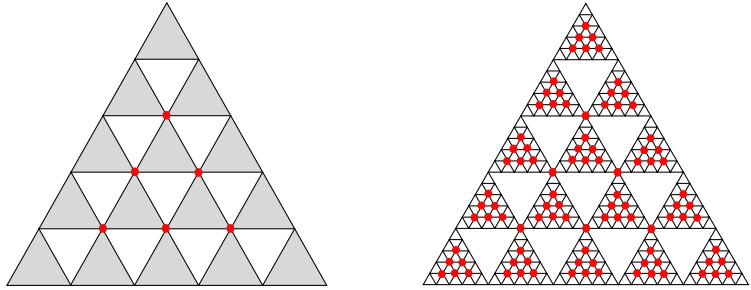

\centering

%
\caption{$G_{5}^{2,1}$ and $G_{5}^{2,2}$.}
\label{fig:3}
\end{figure}

Assume that for $\ell=p,$ $V_{p,3}^{2,m}$ is the union of $\sigma_{p-3}^{2,m}$ copies of $E_{p-3,0}^{2,1}.$

For $\ell=p+1,$ by the structure of $SG_{\ell}^{2},$ the vertex set $V_{p+1,3}^{2,1}$ consists of the vertices already present in $V_{p,3}^{2,1},$ together with new vertices generated from the bottom edge of $G_{\ell-1}^{2,0},$ excluding the boundary points of that edge. The coefficient $\sigma_{p-2}^{2,1}$ expands as
\begin{equation*}
\begin{split}
\sigma_{p-3}^{2,1}+(p+1-2)&=M_{p-3}^{2,1}+(p+1-2)=\binom{p-1}{2}+(p-1)\\
&=\frac{(p-1)p}{2}=\binom{p}{2}=M_{p-2}^{2,1}=\sigma_{p-2}^{2,1}.
\end{split}
\end{equation*}
Through $m$ iterations, $V_{p+1,3}^{2,m}$ becomes the union of $\sigma_{p-2}^{2,m}$ vertex sets of $G_{p-2}^{2,1},$ where
$$\sigma_{p-2}^{2,m}=M_{p-2}^{2,1}\sum_{k=1}^{m}\frac{(N_{p+1}^{2})^{m}-1}{N_{p+1}^{2}-1},$$
and $N_{p+1}^{2}=\frac{(\ell+1)(\ell+2)}{2}.$ This completes the proof of the claim.

\noindent\textit{Step $2.$} Assume that the proposition holds for $n=k$ and all $\alpha$ satisfying $3\leq\alpha\leq k+1.$ We show that the conclusion holds for $n=k+1.$ 

For $n=k+1,$ $\alpha$ ranges up to $k+2.$ The construction of $SG_{\ell}^{k+1}$ involves embedding $SG_{\ell}^{k}$ in $\mathbb{R}^{k}$ and adding points in the $(k+1)$-th direction. For $3\leq\alpha\leq k+1,$ the vertex set $V_{\ell,\alpha}^{k+1,1}$ inherits the recursive structure from $V_{\ell,\alpha}^{k,1},$ and $\sigma_{\ell-\alpha}^{\alpha-1,1}$ remains valid because it depends only on $N_{\ell}^{n}.$ Therefore, $V_{\ell,\alpha}^{k+1,1}$ is the union of $\sigma_{\ell-\alpha}^{\alpha-1,1}$ vertex sets of $G_{\ell-\alpha}^{\alpha-1,1},$ where $\sigma_{\ell-\alpha}^{\alpha-1,1}=M_{\ell-\alpha}^{\alpha-1,1}.$ Therefore, by the iterating process, $V_{\ell,\alpha}^{k+1,m}$ is the union of $\sigma_{\ell-\alpha}^{\alpha-1,m}$ vertex sets of $G_{\ell-\alpha}^{\alpha-1,1},$ where
$$\sigma_{\ell-\alpha}^{\alpha-1,m}=M_{\ell-\alpha}^{\alpha-1,1}\sum_{i=1}^{k}(N_{\ell}^{k+1})^{i-1}=M_{\ell-\alpha}^{\alpha-1,1}\times\frac{(N_{\ell}^{k+1})^{m}-1}{N_{\ell}^{k+1}-1}.$$
For $\alpha=k+2,$ the corresponding graph is $G_{\ell-k-2}^{k+1,1}.$ By Proposition \ref{prop:M}, we have
\begin{equation*}
\begin{split}
\sigma_{\ell-k-2}^{k+1,m}&=M_{\ell-k-2}^{k+1,1}+M_{\ell-k-2}^{k+1,1}\times N_{\ell}^{k+1}+\cdots+M_{\ell-k-2}^{k+1,1}\times(N_{\ell}^{k+1})^{m-1}\\
&=M_{\ell-k-2}^{k+1,1}\times\frac{(N_{\ell}^{k+1})^{m}-1}{N_{\ell}^{k+1}-1}.
\end{split}
\end{equation*}
Building on the previous result that $V_{\ell,3}^{2,m}$ is the union of $\sigma_{\ell-3}^{2,m}$ vertex sets of $G_{\ell-3}^{2,1},$ we extend this to show that when $\alpha=n+1=k+2,$ the set $V_{\ell,k+2}^{k+1,m}$ likewise consists of the union of $\sigma_{\ell-k-2}^{k+1,m}$ vertex sets of $G_{\ell-k-2}^{k+1,1}.$

\noindent\textit{Step $3.$} We prove two cases $\ell=\alpha=n+1$ and $\ell=\alpha+1.$

In the first case, from the proof in Steps $1$ and $2,$ we know that $V_{n+1,n+1}^{n,m}$ is the union of $\sigma_{0}^{n,m}$ vertex sets of $G_{0}^{n,0},$ where
$$\sigma_{0}^{n,m}=\frac{(N_{n+1}^{n})^{m}-1}{N_{n+1}^{n}-1}.$$
In the second case, from the proof in Steps $1$ and $2,$ $V_{\alpha+1,\alpha}^{n,m}$ is the union of $\sigma_{1}^{\alpha-1,m}$ vertex sets of $G_{0}^{\alpha-1,0}.$ For example, when $n=2,$ $\alpha=3,$ and $\ell=4,$ $V_{4,3}^{2,m}$ is the union of $\sigma_{1}^{2,m}$ vertex sets of $G_{0}^{2,0},$ where $\sigma_{1}^{2,m}=3\times\frac{10^{m}-1}{9}.$ By the definition of $M_{\ell}^{n,0},$ since $M_{1}^{\alpha-1,0}=\alpha,$ we have
$$\sigma_{1}^{\alpha-1,m}=M_{1}^{\alpha-1,0}\times\frac{(N_{\alpha+1}^{n})^{m}-1}{N_{\alpha+1}^{n}-1}.$$
\end{proof}

By Proposition $\ref{1},$ for $\ell=\alpha=n+1,$ $V_{n+1,n+1}^{n,m}$ is the union of $\sigma_{0}^{n,m}$ singleton sets. Assume that $\ell\gg\alpha,$ by Proposition \ref{1}, $V_{\ell,\alpha}^{n,m}$ can be defined as
\begin{equation*}
V_{\ell,\alpha}^{n,m}:=\bigcup_{\hat{k}=1}^{\sigma_{\ell-\alpha}^{\alpha-1,m}}E_{\ell-\alpha,0}^{\alpha-1,1}
=\bigcup_{\hat{k}=1}^{\sigma_{\ell-\alpha}^{\alpha-1,m}}\bigcup_{k=1}^{\alpha}V_{\ell-\alpha,k}^{\alpha-1,1}
=B_{\ell,\alpha}^{n,m}\cup C_{\ell,\alpha}^{n,m},
\end{equation*}
where
\begin{equation*}
\begin{aligned}
B_{\ell,\alpha}^{n,m}&:=\{x\in V_{\ell,\alpha}^{n,m}:\nexists y\in V_{\alpha-1,l}^{m,n} \text{ such that } x\sim y\}\\
&:=\bigcup_{\hat{k}=1}^{\sigma_{\ell-\alpha}^{\alpha-1,m}}\bigcup_{k=3}^{\alpha}V_{\ell-\alpha,k}^{\alpha-1,1},
\end{aligned}
\end{equation*}
\begin{equation*}
C_{\ell,\alpha}^{n,m}:=V_{\ell,\alpha}^{n,m}\setminus B_{\ell,\alpha}^{n,m}:=\bigcup_{\hat{k}=1}^{\sigma_{\ell-\alpha}^{\alpha-1,m}}(V_{\ell-\alpha,1}^{\alpha-1,1}\cup V_{\ell-\alpha,2}^{\alpha-1,1}).
\end{equation*}
By the definition of $C_{\ell,\alpha}^{n,m},$ it can be written as
$$C_{\ell,\alpha}^{n,m}:=C_{\ell,\alpha,1}^{n,m}\cup C_{\ell,\alpha,2}^{n,m},$$
where
$$C_{\ell,\alpha,1}^{n,m}:=\bigcup_{\hat{k}=1}^{\sigma_{\ell-\alpha}^{\alpha-1,m}}V_{\ell-\alpha,1}^{\alpha-1,1},\text{ }C_{\ell,\alpha,2}^{n,m}:=\bigcup_{\hat{k}=1}^{\sigma_{\ell-\alpha}^{\alpha-1,m}}V_{\ell-\alpha,2}^{\alpha-1,1}.$$
By Proposition \ref{1} and the structure of $SG_{\ell}^{n},$ $B_{\ell,\alpha}^{n,m}$ is the union of $\sigma_{\ell-2\alpha}^{\alpha-1,m}$ vertex sets of the graph $G_{\ell-2\alpha}^{\alpha-1,1}.$ Thus, the set can also be written as
\begin{equation*}
B_{\ell,\alpha}^{n,m}:=B_{\ell,\alpha,1}^{n,m}\cup B_{\ell,\alpha,2}^{n,m}\cup B_{\ell,\alpha,3}^{n,m}:=\bigcup_{\hat{k}=1}^{\sigma_{\ell-2\alpha}^{\alpha-1,m}}E_{0,l-2\alpha}^{1,\alpha-1},
\end{equation*}
where
$$B_{\ell,\alpha,1}^{n,m}:=\bigcup_{\hat{k}=1}^{\sigma_{\ell-2\alpha}^{\alpha-1,m}}V_{\ell-2\alpha,1}^{\alpha-1,1},\text{ }B_{\ell,\alpha,2}^{n,m}:=\bigcup_{\hat{k}=1}^{\sigma_{\ell-2\alpha}^{\alpha-1,m}}V_{\ell-2\alpha,2}^{\alpha-1,1},\text{ }B_{\ell,\alpha,3}^{n,m}:=\bigcup_{\hat{k}=1}^{\sigma_{\ell-2\alpha}^{\alpha-1,m}}\bigcup_{k=3}^{\ell-2\alpha}V_{\ell-2\alpha,k}^{\alpha-1,1}.$$
According to the structure of $SG_{\ell}^{n}$ and Proposition \ref{1}, $B_{\ell,\alpha,3}^{n,m}$ is the union of $\sigma_{\ell-3\alpha}^{\alpha-1,m}$ vertex sets of the graph $G_{\ell-3\alpha}^{\alpha-1,1}.$ Hence, we have
\begin{equation*}
B_{\ell,\alpha,3}^{n,m}:=B_{\ell,\alpha,3,1}^{n,m}\cup B_{\ell,\alpha,3,2}^{n,m}:=\bigcup_{\hat{k}=1}^{\sigma_{\ell-3\alpha}^{\alpha-1,m}}E_{\ell-3\alpha,0}^{\alpha-1,1},
\end{equation*}
where
$$B_{\ell,\alpha,3,1}^{n,m}:=\bigcup_{\hat{k}=1}^{\sigma_{\ell-3\alpha}^{\alpha-1,m}}V_{\ell-3\alpha,1}^{\alpha-1,1},\text{ }B_{\ell,\alpha,3,2}^{n,m}:=\bigcup_{\hat{k}=1}^{\sigma_{\ell-3\alpha}^{\alpha-1,m}}\bigcup_{k=2}^{\ell-3\alpha}V_{\ell-3\alpha,k}^{\alpha-1,1}.$$

To better understand the definitions of these sets, we take $SG_{11}^{2}$ as an example. On the left side of Figure \ref{fig:4}, the red and purple points are in $C_{11,3}^{2,1},$ with the red points in $C_{11,3,1}^{2,1}$ and the purple points in $C_{11,3,2}^{2,1}.$ The green points are in $B_{11,3}^{2,1}.$ On the right side of Figure \ref{fig:4}, the blue points are in $B_{11,3,1}^{2,1},$ the orange points are in $B_{11,3,2}^{2,1},$ and others are in $B_{11,3,3}^{2,1}.$ The yellow points are in $B_{11,3,3,1}^{2,1},$ and the brown points are in $B_{11,3,3,2}^{2,1}.$
\begin{figure}[t]
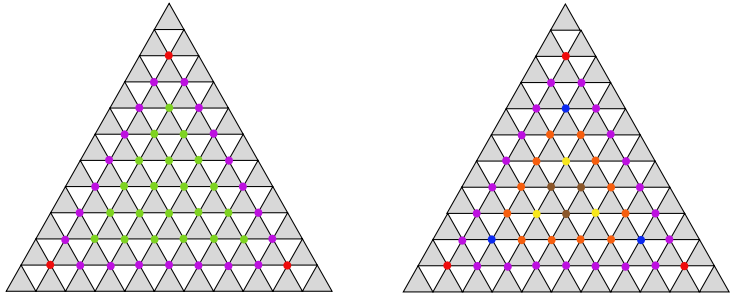

\centering


\caption{The set $V_{11,3}^{2,1}$ on $G_{11}^{2,1}$.}
\label{fig:4}
\end{figure}

We now define the value of the points in $V_{\ell,\alpha}^{n,1}$ $(3\leq\alpha\leq n+1)$ under the harmonic extension. Let $e_{0}^{1}\in C_{\ell,\alpha,1}^{n,1}\cup B_{\ell,\alpha,1}^{n,1}\cup B_{\ell,\alpha,3,1}^{n,1}.$ If $e_{0}^{1}$ is the $i$-th vertex on the graph $G_{\ell-\alpha}^{\alpha-1,1},$ $G_{\ell-2\alpha}^{\alpha-1,1}$ or $G_{\ell-3\alpha}^{\alpha-1,1},$ then there exist mappings $F_{k_{j}}$ and $q_{\bar{k}_{j}}\in V_{0}$ $(j=1,2,...\alpha)$ such that $e_{0}^{1}=F_{k_{j}}q_{\bar{k}_{j}},$ where $k_{1}<k_{2}<\cdots<k_{\alpha}$ and $\bar{k}_{1}>\bar{k}_{2}>\cdots>\bar{k}_{\alpha}.$ Therefore, we have
\begin{equation}\label{eq:3.7}
h(e_{0}^{1})=b_{1,\alpha}h(q_{i})+b_{2,\alpha}\sum_{j\in\{\bar{k}_{i}\}\setminus{i}}h(q_{j})+b_{3,\alpha}\sum_{\tilde{\ell}\notin\{\bar{k}_{i}\}}h(q_{\tilde{\ell}}),
\end{equation}
where $0<b_{3,\alpha}<b_{2,\alpha}<b_{1,\alpha}<1$ and $b_{1,\alpha}+(\alpha-1)b_{2,\alpha}+(n+1-\alpha)b_{3,\alpha}=1.$ The coefficients $b_{k,\alpha}$ $(k=1,2,3)$ vary depending on the solution of $AX=B.$

For any $e_{0}^{1}\in V_{\ell,\alpha}^{n,1}\setminus(C_{\ell,\alpha,1}^{n,1}\cup B_{\ell,\alpha,1}^{n,1}\cup B_{\ell,\alpha,3,1}^{n,1}),$ there exist mappings $F_{k_{j}}$ and $q_{\bar{k}_{j}}\in V_{0}$ $(j=1,2,...,\alpha)$ such that $e_{0}^{1}=F_{k_{j}}q_{\bar{k}_{j}},$ where $k_{1}<k_{2}<\cdots<k_{\alpha}$ and $\bar{k}_{1}>\bar{k}_{2}>\cdots>\bar{k}_{\alpha}.$ We have
\begin{equation}\label{eq:3.8}
h(e_{0}^{1})=\sum_{j\in\{\bar{k}_{i}\}}\bar{b}_{j,\alpha}h(q_{j})+\bar{b}_{\alpha}\sum_{\tilde{\ell}\notin\{\bar{k}_{i}\}}h(q_{\tilde{\ell}}),
\end{equation}
where $0<\bar{b}_{\alpha}<\bar{b}_{j,\alpha}<1,$ $\bar{b}_{i,\alpha}=\max_{j\in\{\bar{k}_{j}\}}\bar{b}_{j,\alpha}$ and the index $i$ indicates that the point $e_{0}^{1}$ is adjacent to the boundary point  $q_{i}.$ Moreover, $\sum_{j\in\{\bar{k}_{j}\}}\bar{b}_{j,\alpha}+(n+1-\alpha)\bar{b}_{\alpha}=1.$

When $\alpha=n+1,$ $V_{n+1,n+1}^{n,1}$ consists of only one point. This unique point satisfies $e_{0}^{1}=F_{k_{i}}q_{i}$ for all $i=0,1,...,n,$ where $q_{i}\in V_{0}.$ This implies:
\begin{equation}\label{eq:3.9}
h(e_{0}^{1})=\frac{1}{n+1}\sum_{i=0}^{n}h(q_{i}).
\end{equation}

Next, we characterize the representation of edges in $E_{\ell,1}^{n,m+1}$ under harmonic $1$-forms. For any harmonic $1$-form $h_{1}^{m}\in\mathcal{H}_{\ell,1}^{n,m},$ by\cite[Theorem 2.7]{Strichartz_2002}, we have $h_{1}^{m}=d_{0}^{m}f_{0}^{m}$ for some harmonic mapping $f_{0}^{m}.$ For any $[p,q]\in E_{\ell,1}^{n,m},$ we have
$$h_{1}^{m}([p,q])=f_{0}^{m}(q)-f_{0}^{m}(p).$$
Let $p_{j}\in V_{\ell,2}^{n,m+1}$ satisfy $p_{j}=F_{i}q_{j}=F_{k}q_{i},$ and let $p_{j}$ be connected to $q_{i}\in V_{0}.$ For any edge $e_{1}^{m+1}=[q_{i},p_{j}]\in E_{\ell,1}^{n,m+1}$ where $p_{j}$ is an endpoint and $q_{i}\in V_{0},$ equation $(\ref{eq:3.4})$ yields
\begin{equation}\label{eq:3.10}
h_{1}^{m+1}([q_{i},p_{j}])=\sum_{\tilde{\ell}=0, \tilde{\ell}\neq i}^{n}\tilde{a}_{\tilde{\ell}}h_{1}^{m}([q_{i},q_{\tilde{\ell}}])
\end{equation}
where
\begin{equation*}
\tilde{a}_{\tilde{\ell}}=\begin{cases}
a_{2}, \quad & \tilde{\ell}=j;\\
a_{3}, \quad & \tilde{\ell}\neq i,j.
\end{cases}
\end{equation*}

For any edge $e_{1}^{m+1}=[p_{j},p_{k}]\in E_{\ell,1}^{n,m+1},$ if $p_{k}\in V_{\ell,2}^{n,m+1},$ and both $p_{j}$ and $p_{k}$ are connected to $q_{i}\in V_{0},$ given $p_{k}=F_{i}q_{k}=F_{k_{2}}q_{i},$ then equation $(\ref{eq:3.4})$ yields
\begin{equation}\label{eq:3.11}
h_{1}^{m+1}([p_{j},p_{k}])=(a_{2}-a_{3})h_{1}^{m}([q_{j},q_{k}]).
\end{equation}

For any edge $e_{1}^{m+1}=[p_{j},p_{k}]\in E_{\ell,1}^{n,m+1},$ if $p_{k}\in V_{\ell,2}^{n,m+1}$ is not connected to any point in $V_{0},$ given $p_{k}=F_{k_{1}}q_{j}=F_{k_{2}}q_{i},$ then equations $(\ref{eq:3.4})$ and $(\ref{eq:3.5})$ yield
\begin{equation}\label{eq:3.12}
h_{1}^{m+1}([p_{j},p_{k}])=\sum_{\tilde{\ell}=0, \tilde{\ell}\neq i}^{n}\tilde{a}_{\tilde{\ell}}^{1}h_{1}^{m}([q_{i},q_{\tilde{\ell}}]),
\end{equation}
where
\begin{equation*}
\tilde{a}_{\tilde{\ell}}^{1}=\begin{cases}
\bar{a}_{2}-a_{2}, \quad & \tilde{\ell}=j;\\
\bar{a}_{3}-a_{3}, \quad & \tilde{\ell}\neq i,j.
\end{cases}
\end{equation*}

For any edge $e_{1}^{m+1}=[p_{j},p_{k}]\in E_{\ell,1}^{n,m+1},$ if $p_{k}\in C_{\ell,3,1}^{n,m+1}$ and $p_{k}=F_{k_{1}}q_{\bar{\ell}}=F_{k_{2}}q_{j}=F_{k_{3}}q_{i},$ then equations $(\ref{eq:3.4})$ and $(\ref{eq:3.7})$ yield
\begin{equation}\label{eq:3.13}
h_{1}^{m+1}([p_{j},p_{k}])=\sum_{\tilde{\ell}=0, \tilde{\ell}\neq i}^{n}\tilde{a}_{\tilde{\ell}}^{2}h_{1}^{m}([q_{i},q_{\tilde{\ell}}]),
\end{equation}
where
\begin{equation*}
\tilde{a}_{\tilde{\ell}}^{2}=\begin{cases}
b_{2,3}-a_{2}, \quad & \tilde{\ell}=j;\\
b_{2,3}-a_{3}, \quad & \tilde{\ell}=\bar{\ell};\\
b_{3,3}-a_{3}, \quad & \tilde{\ell}\neq i,j,\bar{\ell}.
\end{cases}
\end{equation*}

Let $p_{i}\in V_{\ell,2}^{n,m+1}$ be a point not connected to any point in $V_{0},$ and let $p_{i}=F_{k_{1}}q_{i}=F_{k_{2}}q_{j}.$ For any edge $e_{1}^{m+1}=[p_{i},p_{j}]\in E_{\ell,1}^{n,m+1},$ if $p_{j}$ is not connected to any point in $V_{0},$ and $p_{j}=F_{k_{1}}q_{j}=F_{k_{3}}q_{i},$ then by using equation $(\ref{eq:3.5}),$ we have
\begin{equation}\label{eq:3.14}
h_{1}^{m+1}([p_{i},p_{j}])=\left\{
\begin{aligned}
(\bar{a}_{1}-\bar{a}_{2})h_{1}^{m}([q_{i},q_{j}]) \\
\sum_{\tilde{\ell}=0, \tilde{\ell}\neq i}^{n}\tilde{a}_{\tilde{\ell}}^{3}h_{1}^{m}([q_{\tilde{\ell}},q_{i}])
\end{aligned}
\right.
\end{equation}
where
\begin{equation*}
\tilde{a}_{\tilde{\ell}}^{3}=\begin{cases}
\bar{a}_{2}-\bar{a}_{2}, \quad & \tilde{\ell}=j;\\
\bar{a}_{3}-\bar{a}_{3}, \quad & \tilde{\ell}\neq i,j.
\end{cases}
\end{equation*}
Note that $\bar{a}_{i}$ $(i=2,3)$ are not equal.

For any edge $e_{1}^{m+1}=[p_{i},p_{j}]\in E_{\ell,1}^{n,m+1}$ where $p_{j}$ is the midpoint of the edge $[q_{i},q_{j}]$ and $p_{j}=F_{i}q_{j}=F_{j}q_{i},$ equations $(\ref{eq:3.5})$ and $(\ref{eq:3.6})$ yield
\begin{equation}\label{eq:3.15}
h_{1}^{m+1}=\sum_{\tilde{\ell}=0, \tilde{\ell}\neq i}^{n}\tilde{a}_{\tilde{\ell}}^{4}h_{1}^{m}([q_{i},q_{\tilde{\ell}}]),
\end{equation}
where
\begin{equation*}
\tilde{a}_{\tilde{\ell}}^{4}=\begin{cases}
\hat{a}_{1}-\bar{a}_{2}, \quad & \tilde{\ell}=j;\\
\hat{a}_{2}-\bar{a}_{3}, \quad & \tilde{\ell}\neq i,j.
\end{cases}
\end{equation*}

For any edge $e_{1}^{m+1}=[p_{i},p_{j}]\in E_{\ell,1}^{n,m+1},$ if $p_{j}\in C_{\ell,3,1}^{n,m+1}$ satisfies $p_{j}=F_{k_{1}}q_{\bar{\ell}}=F_{k_{3}}q_{j}=F_{k_{4}}q_{i},$ then by applying \eqref{eq:3.5} and \eqref{eq:3.7}, we obtain
\begin{equation}\label{eq:3.16}
h_{1}^{m+1}=\sum_{\tilde{\ell}=0, \tilde{\ell}\neq i}^{n}\tilde{a}_{\tilde{\ell}}^{5}h_{1}^{m}([q_{\tilde{\ell}},q_{i}]),
\end{equation}
where
\begin{equation*}
\tilde{a}_{\tilde{\ell}}^{5}=\begin{cases}
\bar{a}_{2}-b_{2,3}, \quad & \tilde{\ell}=j;\\
\bar{a}_{3}-b_{2,3}, \quad & \tilde{\ell}=\bar{\ell};\\
\bar{a}_{3}-b_{3,3}, \quad & \tilde{\ell}\neq i,j,\bar{\ell}.\\
\end{cases}
\end{equation*}

For any edge $e_{1}^{m+1}=[p_{i},p_{j}]\in E_{\ell,1}^{n,m+1}$ where $p_{j}\in C_{\ell,3,2}^{n,m+1}$ and $p_{j}=F_{k_{1}}q_{\bar{\ell}}=F_{k_{3}}q_{j}=F_{k_{4}}q_{i},$ by applying equations $(\ref{eq:3.5})$ and $(\ref{eq:3.8}),$ we have
\begin{equation}\label{eq:3.17}
h_{1}^{m+1}=\sum_{\tilde{\ell}=0, \tilde{\ell}\neq i}^{n}\tilde{a}_{\tilde{\ell}}^{6}h_{1}^{m}([q_{i},q_{\tilde{\ell}}]),
\end{equation}
where
\begin{equation*}
\tilde{a}_{\tilde{\ell}}^{6}=\begin{cases}
\bar{b}_{2,3}-\bar{a}_{2}, \quad & \tilde{\ell}=j;\\
\bar{b}_{3,3}-\bar{a}_{3}, \quad & \tilde{\ell}=\bar{\ell};\\
\bar{b}_{4,3}-\bar{a}_{3}, \quad & \tilde{\ell}\neq i,j,\bar{\ell}.\\
\end{cases}
\end{equation*}

Let $p_{ij}$ denote the midpoint of the edge $[q_{i},q_{j}].$ For any edge $e_{1}^{m+1}=[p_{ij},p_{j}]\in E_{\ell,1}^{n,m+1}$ where $p_{j}\in C_{\ell,3,1}^{n,m+1},$ we consider the following conditions:
\begin{enumerate}[label=(\alph*)]
\item when $\ell=4,$ $p_{j}$ is connected to the vertices $F_{i}q_{j}$ and $F_{j}q_{\bar{\ell}}$;
\item when $\ell\neq 4,$ the connection pattern is reversed (opposite to Case (\textrm{a})).
\end{enumerate}
Given $p_{j}=F_{k_{1}}q_{\bar{\ell}}=F_{k_{3}}q_{j}=F_{k_{4}}q_{i},$ combining equations $(\ref{eq:3.6}),$ $(\ref{eq:3.7})$ and $(\ref{eq:3.8}),$ we have
\begin{equation}\label{eq:3.18}
h_{1}^{m+1}([p_{ij},p_{j}])=\left\{
\begin{aligned}
\sum_{\tilde{\ell}=0, \tilde{\ell}\neq j}^{n}\tilde{a}_{\tilde{\ell}}^{7}h_{1}^{m}([q_{j},q_{\tilde{\ell}}]) \quad & \ell=4,\\
\sum_{\tilde{\ell}=0, \tilde{\ell}\neq j}^{n}\tilde{a}_{\tilde{\ell}}^{8}h_{1}^{m}([q_{j},q_{\tilde{\ell}}]) \quad & \ell>4,
\end{aligned}
\right.
\end{equation}
where when $\ell=4,$
\begin{equation*}
\tilde{a}_{\tilde{\ell}}^{7}=\begin{cases}
b_{2,3}-\hat{a}_{1}, \quad & \tilde{\ell}=i;\\
b_{2,3}-\hat{a}_{3}, \quad & \tilde{\ell}=\bar{\ell};\\
b_{3,3}-\hat{a}_{3}, \quad & \tilde{\ell}\neq i,j,\bar{\ell};
\end{cases}
\end{equation*}
when $\ell>4,$
\begin{equation*}
\tilde{a}_{\tilde{\ell}}^{8}=\begin{cases}
\bar{b}_{1,3}-\hat{a}_{1}, \quad & \tilde{\ell}=i;\\
\bar{b}_{3,3}-\hat{a}_{3}, \quad & \tilde{\ell}=\bar{\ell};\\
\bar{b}_{4,3}-\hat{a}_{3}, \quad & \tilde{\ell}\neq i,j,\bar{\ell}.
\end{cases}
\end{equation*}

For any edge $e_{1}^{m+1}=[p_{i},p_{j}]\in E_{\ell,1}^{n,m+1}$ with starting point $p_{i}\in V_{\ell,\alpha}^{n,m+1}$ $(\text{where }\alpha>2),$  we classify the analysis into three cases through the application of equations $(\ref{eq:3.7})$ and $(\ref{eq:3.8}).$

\noindent\textit{Case $I.$} The endpoints of the edge are expressed as in $(\ref{eq:3.7}).$
\begin{enumerate}[label=(\roman*)]
\item Both endpoints belong to $C_{\ell,\alpha,1}^{n,m+1}.$
\item Both endpoints belong to $B_{\ell,\alpha,1}^{n,m+1}.$
\item Both endpoints belong to $B_{\ell,\alpha,3,1}^{n,m+1}.$
\end{enumerate}

In all these cases, for any edge $e_{1}^{m+1}=[p_{i},p_{j}]\in E_{\ell,1}^{n,m+1},$ we have
\begin{equation}\label{eq:3.19}
h_{1}^{m+1}([p_{i},p_{j}])=(b_{2,\alpha}-b_{3,\alpha})h_{1}^{m}([q_{\bar{j}},q_{\bar{k}}]),
\end{equation}
where the indexes $\bar{j}$ and $\bar{k}$ refer to the subscripts of two endpoints at different iteration points under the IFS maps.

However, there is a special case, namely, if $p_{i}\in C_{\ell,\alpha,1}^{n,m+1}$ and $p_{j}\in C_{\ell,\alpha+1,1}^{n,m+1},$ then we have
\begin{equation}\label{eq:3.20}
h_{1}^{m+1}([p_{i},p_{j}])=\sum_{\tilde{\ell}=0, \tilde{\ell}\neq i}^{n}r_{\tilde{\ell}}^{\alpha,1}h_{1}^{m}([q_{i},q_{\tilde{\ell}}]),
\end{equation}
where
\begin{equation*}
r_{\tilde{\ell}}^{\alpha,1}=\begin{cases}
b_{2,\alpha}-b_{2,\alpha+1}, \quad & \tilde{\ell}\in\{\bar{k}_{j}\}\setminus\{i\};\\
b_{3,\alpha}-b_{2,\alpha+1}, \quad & \tilde{\ell}=\hat{k};\\
b_{3,\alpha}-b_{3,\alpha+1}, \quad & \tilde{\ell}\notin\{\bar{k}_{j}\cup\{\bar{k}\}.
\end{cases}
\end{equation*}
Note that the set $\{\bar{k}_{j}\}$ denotes the index subscripts of the points $q_{\bar{k}{j}}$ satisfying $p_{i}=F_{\omega_{j}}q_{\bar{k}_{j}},$ while $\hat{k}$ represents the incremental index value of the terminal point relative to the starting point.

\noindent\textit{Case $II.$} The two endpoints of the edge are given by the expressions in $(\ref{eq:3.7})$ and $(\ref{eq:3.8}).$ In this case, we have two situations.
\begin{enumerate}[label=(\roman*),itemsep=0pt, parsep=0pt]
\item The starting and terminal points are given in $(\ref{eq:3.7})$ and $(\ref{eq:3.8}),$ respectively.
\begin{enumerate}[label=(\alph*)]
\item The starting point belongs to $C_{\ell,\alpha,2}^{n,m+1}$ and is connected to the points in $B_{\ell,\alpha,1}^{n,m+1},$ and the terminal point belongs to $B_{\ell,\alpha,1}^{n,m+1}.$
\item The starting point belongs to $B_{\ell,\alpha,2}^{n,m+1},$ and the terminal point belongs to $B_{\ell,\alpha,3,1}^{n,m+1}.$

In both Cases $({\rm a})$ and $({\rm b}),$ we have
\begin{equation}\label{eq:3.21}
h_{1}^{m+1}(e_{1}^{m+1})=\sum_{\tilde{\ell}=0, \tilde{\ell}\neq i}^{n}r_{\tilde{\ell}}^{\alpha,2}h_{1}^{m}([q_{i},q_{\tilde{\ell}}]),
\end{equation}
where
\begin{equation*}
r_{\tilde{\ell}}^{\alpha,2}=\begin{cases}
b_{2,\alpha}-\bar{b}_{\tilde{\ell},\alpha}, \quad & \tilde{\ell}\in\{\bar{k}_{j}\}\setminus\{i\};\\
b_{3,\alpha}-\bar{b}_{\alpha}, \quad & \tilde{\ell}\notin\{\bar{k}_{j}\}.
\end{cases}
\end{equation*}

\item The starting point belongs to $C_{\ell,\alpha,2}^{n,m+1}$ and is connected to the points in $B_{\ell,\alpha,1}^{n,m+1};$ the terminal point belongs to $C_{\ell,\alpha+1,1}^{n,m+1},$ we have
\begin{equation}\label{eq:3.22}
h_{1}^{m+1}(e_{1}^{m+1})=\sum_{\tilde{\ell}=0, \tilde{\ell}\neq i}^{n}r_{\tilde{\ell}}^{\alpha,3}h_{1}^{m}([q_{i},q_{\tilde{\ell}}]),
\end{equation}
where
\begin{equation*}
r_{\tilde{\ell}}^{\alpha,3}=\begin{cases}
b_{2,\alpha+1}-\bar{b}_{\tilde{\ell},\alpha}, \quad & \tilde{\ell}\in\{\bar{k}_{j}\}\setminus\{i\};\\
 b_{2,\alpha+1}-\bar{b}_{\alpha}, \quad & \tilde{\ell}=\hat{k};\\
b_{3,\alpha+1}-\bar{b}_{\alpha}, \quad & \tilde{\ell}\notin\{\bar{k}_{j}\}\cup\{\hat{k}\}.
\end{cases}
\end{equation*}
Note that $\hat{k}$ denotes an index belonging to the terminal point's index set but not to the starting point's iteration indices.
\end{enumerate}

\item The starting and terminal points are given in $(\ref{eq:3.8})$ and $(\ref{eq:3.7}),$ respectively.
\begin{enumerate}[label=(\alph*),itemsep=0pt, parsep=0pt]
\item The starting point belongs to $C_{\ell,\alpha,1}^{n,m+1},$ and the terminal point belongs to $C_{\ell,\alpha,2}^{n,m+1}$ and is connected to the points in $B_{\ell,\alpha,1}^{n,m+1}.$
\item The starting point belongs to $B_{\ell,\alpha,1}^{n,m+1},$ and the terminal point belongs to $B_{\ell,\alpha,2}^{n,m+1}.$
\item The starting point belongs to $B_{\ell,\alpha,3,1}^{n,m+1},$ and the terminal point belongs to $B_{\ell,\alpha,3,2}^{n,m+1}.$

For Cases $({\rm a})-({\rm c}),$ we have
\begin{equation}\label{eq:3.23}
h_{1}^{m+1}(e_{1}^{m+1})=\sum_{\tilde{\ell}=0, \tilde{\ell}\neq i}^{n}r_{\tilde{\ell}}^{\alpha,4}h_{1}^{m}([q_{i},q_{\tilde{\ell}}]),
\end{equation}
where
\begin{equation*}
r_{\tilde{\ell}}^{\alpha,4}=\begin{cases}
\bar{b}_{\tilde{\ell},\alpha}-b_{2,\alpha}, \quad & \tilde{\ell}\in\{\bar{k}_{j}\}\setminus\{i\};\\
\bar{b}_{\alpha}-b_{3,\alpha}, \quad & \tilde{\ell}\notin\{\bar{k}_{j}\}.
\end{cases}
\end{equation*}

\item The starting point belongs to $C_{\ell,\alpha,1}^{n,m+1},$ and the terminal point belongs to $C_{\ell,\alpha+1,2}^{n,m+1}.$
\item The starting point belongs to $B_{\ell,\alpha,1}^{n,m+1},$ and the terminal point belongs to $C_{\ell,\alpha+1,2}^{n,m+1}.$
\item The starting point belongs to $B_{\ell,\alpha,3,1}^{n,m+1},$ and the terminal point belongs to $C_{\ell,\alpha+1,2}^{n,m+1}.$

For Cases $({\rm d})-({\rm{f}}),$ we have
\begin{equation}\label{eq:3.24}
h_{1}^{m+1}(e_{1}^{m+1})=\sum_{\tilde{\ell}=0, \tilde{\ell}\neq i}^{n}r_{\tilde{\ell}}^{\alpha,5}h_{1}^{m}([q_{\tilde{\ell}},q_{i}]),
\end{equation}
where
\begin{equation*}
r_{\tilde{\ell}}^{\alpha,5}=\begin{cases}
\bar{b}_{\tilde{\ell},\alpha+1}-b_{2,\alpha}, \quad & \tilde{\ell}\in\{\bar{k}_{j}\}\setminus\{i\};\\
\bar{b}_{\hat{k},\alpha+1}-b_{3,\alpha}, \quad & \tilde{\ell}=\hat{k};\\
\bar{b}_{\alpha+1}-b_{3,\alpha}, \quad & \tilde{\ell}\notin\{\bar{k}_{j}\}\cup\{\hat{k}\}.
\end{cases}
\end{equation*}
Note that $\hat{k}$ is defined as above.
\end{enumerate}
\end{enumerate}

\noindent\textit{Case $III.$} The endpoints of the edge are the expressions of $(\ref{eq:3.8}).$
\begin{enumerate}[label=(\alph*),itemsep=0pt, parsep=0pt]
\item The starting point belongs to $C_{\ell,\alpha,2}^{n,m+1}$ and is connected to points in $B_{\ell,\alpha,1}^{n,m+1}$ and $B_{\ell,\alpha,2}^{n,m+1}$; the terminal point belongs to $C_{\ell,\alpha,2}^{n,m+1}$ and is connected to points in $B_{\ell,\alpha,2}^{n,m+1}$.
\item The starting point belongs to $C_{\ell,\alpha,2}^{n,m+1}$ and is connected to points in $B_{\ell,\alpha,1}^{n,m+1}$ and $B_{\ell,\alpha,2}^{n,m+1}$; the terminal point belongs to $B_{\ell,\alpha,2}^{n,m+1}$.
\item The starting point belongs to $C_{\ell,\alpha,2}^{n,m+1}$ and is connected to points in $B_{\ell,\alpha,2}^{n,m+1}$; the terminal point belongs to $B_{\ell,\alpha,2}^{n,m+1}$.
\item The starting point belongs to $B_{\ell,\alpha,2}^{n,m+1}$; the terminal point belongs to $B_{\ell,\alpha,3,2}^{n,m+1}$.

In Cases $({\rm a})-({\rm d}),$ we have
\begin{equation}\label{eq:3.25}
h_{1}^{m+1}(e_{1}^{m+1})=\sum_{\tilde{\ell}=0, \tilde{\ell}\neq i}^{n}r_{\tilde{\ell}}^{\alpha,6}h_{1}^{m}([q_{i},q_{\tilde{\ell}}]),
\end{equation}
where
\begin{equation*}
r_{\tilde{\ell}}^{\alpha,6}=\begin{cases}
\bar{b}_{\tilde{\ell},\alpha}-\bar{b}_{\tilde{\ell},\alpha}, \quad & \tilde{\ell}\in\{\bar{k}_{j}\}\setminus\{i\};\\
\bar{b}_{\alpha}-\bar{b}_{\alpha}, \quad & \tilde{\ell}\notin\{\bar{k}_{j}\}.
\end{cases}
\end{equation*}
Since the endpoints belong to different sets, their associated coefficients are distinct, and this ensures that $r_{\tilde{\ell}}^{\alpha,6}$ is nonzero.

\item The starting point belongs to $C_{\ell,\alpha,2}^{n,m+1}$ and is connected to points in $B_{\ell,\alpha,1}^{n,m+1}$ and $B_{\ell,\alpha,2}^{n,m+1}$; the terminal point belongs to $C_{\ell,\alpha+1,2}^{n,m+1}$.
\item The starting point belongs to $C_{\ell,\alpha,2}^{n,m+1}$ and is connected to points in $B_{\ell,\alpha,2}^{n,m+1}$; the terminal point belongs to $C_{\ell,\alpha+1,2}^{n,m+1}$.
\item The starting point belongs to $B_{\ell,\alpha,2}^{n,m+1}$; the terminal point belongs to $C_{\ell,\alpha+1,2}^{n,m+1}$.
\item The starting point belongs to $B_{\ell,\alpha,3,2}^{n,m+1}$; the terminal point belongs to $C_{\ell,\alpha+1,2}^{n,m+1}$.

In Cases $({\rm e})-({\rm h}),$ we have
\begin{equation}\label{eq:3.26}
h_{1}^{m+1}(e_{1}^{m+1})=\sum_{\tilde{\ell}=0, \tilde{\ell}\neq i}^{n}r_{\tilde{\ell}}^{\alpha,7}h_{1}^{m}([q_{i},q_{\tilde{\ell}}]),
\end{equation}
where
\begin{equation*}
r_{\tilde{\ell}}^{\alpha,7}=\begin{cases}
\bar{b}_{\tilde{\ell},\alpha+1}-\bar{b}_{\tilde{\ell},\alpha}, \quad & \tilde{\ell}\in\{\bar{k}_{j}\}\setminus\{i\};\\
\bar{b}_{\hat{k},\alpha+1}-\bar{b}_{\alpha}, \quad & \tilde{\ell}=\hat{k};\\
\bar{b}_{\alpha+1}-\bar{b}_{\alpha}, \quad & \tilde{\ell}\notin\{\bar{k}_{j}\}\cup\{\hat{k}\}.
\end{cases}
\end{equation*}
Note that $\hat{k}$ represents the incremental index value of the terminal point relative to the starting point's indices.
\end{enumerate}

In Case $\textrm{III},$ there are three situations, namely:
\begin{enumerate}[label=(\alph*),itemsep=0pt,parsep=0pt]
\item The endpoints belong to $C_{\ell,\alpha,2}^{n,m+1}$ and are connected to the points in $B_{\ell,\alpha,1}^{n,m+1}.$
\item The endpoints belong to $B_{\ell,\alpha,2}^{n,m+1}.$
\item The endpoints belong to $B_{\ell,\alpha,3,2}^{n,m+1}.$
\end{enumerate}
In all cases above, we have
\begin{equation}\label{eq:3.27}
h_{1}^{m+1}(e_{1}^{m+1})=(\bar{b}_{i,\alpha}-\bar{b}_{j,\alpha})h_{1}^{m}([q_{i},q_{j}]),
\end{equation}
where the indices $i$ and $j$ correspond to the endpoints of the edge $e_1^{m+1}$ and are generated by $F_{k}q_{i}$ and $F_{k}q_{j},$ respectively.

Before proving Theorem \ref{thm:1.1}, we first establish several propositions concerning $1$-forms obtained through harmonic extension from generation $m$ to generation $m+1.$ These propositions characterize the behavior of such $1$-forms on both edges of triangles and edges constructed by connecting pairs of vertices, thereby laying the necessary foundation for our main theorem.

For any triangle $e_{2}^{m+1}=[p_{i},p_{j},p_{\bar{\ell}}]\in E_{\ell,2}^{n,m+1},$ by the definition of de Rham derivative, we have
\begin{equation}\label{eq:3.28}
\begin{split}
d_{1}^{m+1}h_{1}^{m+1}(e_{2}^{m+1})&=\sum_{e_{1}^{m+1}\in E_{\ell,1}^{n,m+1}}{\rm sgn}(e_{1}^{m+1},e_{2}^{m+1})h_{1}^{m+1}(e_{1}^{m+1})\\
&=\sum_{e_{1}^{m+1}\subset e_{2}^{m+1}}{\rm sgn}(e_{1}^{m+1},e_{2}^{m+1})h_{1}^{m+1}(e_{1}^{m+1})\\
&=h_{1}^{m+1}([p_{i},p_{j}])+h_{1}^{m+1}([p_{j},p_{\bar{\ell}}])+h_{1}^{m+1}([p_{\bar{\ell}},p_{i}]).
\end{split}
\end{equation}
The last equality holds because the orientation of these edges $[p_{i},p_{j}],[p_{j},p_{\bar{\ell}}]$ and $[p_{\bar{\ell}},p_{i}]$ follows the positive direction on triangle $e_{2}^{m+1}.$ By the definition of parity function, we have ${\rm sgn}=1.$

\begin{prop}\label{prop:1}
For any $h_{1}^{m}\in\mathcal{H}_{\ell,1}^{n,m},$ let $h_{1}^{m+1}$ be a $1$-form obtained by extending $h_{1}^{m}$ to the graph $G_{\ell}^{n,m+1}.$ For any triangle $e_{2}^{m+1}=[p_{i},p_{j},p_{\bar{\ell}}]\in E_{\ell,2}^{n,m+1},$ if $p_{i}\in E_{\ell,0}^{n,m},$ then
$$d_{1}^{m+1}h_{1}^{m+1}(e_{2}^{m+1})=0.$$
\end{prop}
\begin{proof}
Since $p_{i}\in E_{\ell,0}^{n,m},$ we have $p_{j},p_{\bar{\ell}}\in V_{\ell,2}^{n,m+1}.$ By using equations $(\ref{eq:3.10})$ and $(\ref{eq:3.11}),$ we can extend the right-hand side of $(\ref{eq:3.28})$ as follows:
\begin{equation*}
\begin{aligned}
&\qquad h_{1}^{m+1}([p_{i},p_{j}])+h_{1}^{m+1}([p_{j},p_{\bar{\ell}}])+h_{1}^{m+1}([p_{\bar{\ell}},p_{i}])\\
&=\sum_{\tilde{\ell}\neq i}\tilde{a}_{\tilde{\ell}}h_{1}^{m}([p_{i},q_{\tilde{\ell}}])+(a_{2}-a_{3})h_{1}^{m}([q_{j},q_{\bar{\ell}}])
+\sum_{\tilde{\ell}\neq i}\tilde{a}_{\tilde{\ell}}h_{1}^{m}([q_{\tilde{\ell}},p_{i}])\\
&=\sum_{\tilde{\ell}\neq i,j,\bar{\ell}}a_{3}h_{1}^{m}([p_{i},q_{\tilde{\ell}}])+a_{2}h_{1}^{m}([p_{i},q_{j}])+a_{3}h_{1}^{m}([p_{i},q_{\bar{\ell}}])+(a_{2}-a_{3})h_{1}^{m}([q_{j},q_{\bar{\ell}}])\\
&\quad +\sum_{\tilde{\ell}\neq i,j,\bar{\ell}}a_{3}h_{1}^{m}([q_{\tilde{\ell}},p_{i}])+a_{2}h_{1}^{m}([q_{\bar{\ell}},p_{i}])+a_{3}h_{1}^{m}([q_{j},p_{i}])\\
&=(a_{2}-a_{3})h_{1}^{m}([p_{i},q_{j}])+(a_{2}-a_{3})h_{1}^{m}([q_{j},q_{\bar{\ell}}])+(a_{2}-a_{3})h_{1}^{m}([q_{\bar{\ell}},p_{i}])\\
&=(a_{2}-a_{3})\sum_{\substack{e_{1}^{m}\subset e_{2}^{m}=[p_{i},q_{j},q_{\bar{\ell}}]\\e_{2}^{m+1}=F_{k}e_{2}^{m}}}h_{1}^{m}(e_{1}^{m})\\
&=(a_{2}-a_{3})d_{1}^{m}h_{1}^{m}(e_{2}^{m})\qquad\qquad(\text{the definition of $d_{1}^{m}$})\\
&=0.\qquad\qquad (h_{1}^{m}\in\mathcal{H}_{\ell,1}^{n,m})
\end{aligned}
\end{equation*}
\end{proof}

\setlist[enumerate]{itemsep=0pt, topsep=0pt, partopsep=0pt}
\begin{prop}\label{prop:2}
For $h_{1}^{m}\in\mathcal{H}_{\ell,1}^{n,m},$ let $h_{1}^{m+1}$ be a $1$-form obtained by extending $h_{1}^{m}$ to the graph $G_{\ell}^{n,m+1}.$ For any triangle $e_{2}^{m+1}=[p_{i},p_{j},p_{\bar{\ell}}]\in E_{\ell,2}^{n,m+1},$ if $p_{i},p_{j},p_{\bar{\ell}}\in V_{\ell,\alpha}^{n,m+1},$ then one of the following six cases must hold:
\begin{enumerate}[label=(\roman*),itemsep=0pt, parsep=0pt]
\item one vertex belongs to $C_{\ell,\alpha,1}^{n,m+1},$ and the other two vertices belong to $C_{\ell,\alpha,2}^{n,m+1};$
\item one vertex belongs to $B_{\ell,\alpha,1}^{n,m+1},$ and the other two vertices belong to $B_{\ell,\alpha,2}^{n,m+1};$
\item one vertex belongs to $B_{\ell,\alpha,3,1}^{n,m+1},$ and the other two vertices belong to $B_{\ell,\alpha,3,2}^{n,m+1};$
\item two vertices belong to $C_{\ell,\alpha,2}^{n,m+1},$ and the other vertex belongs to $B_{\ell,\alpha,2}^{n,m+1};$
\item two vertices belong to $B_{\ell,\alpha,2}^{n,m+1},$ and the other vertex belongs to $B_{\ell,\alpha,3,2}^{n,m+1};$
\item three vertices belong to $B_{\ell,\alpha,3,2}^{n,m+1}.$
\end{enumerate}
If any of these cases holds, then
\begin{equation*}
d_{1}^{m+1}h_{1}^{m+1}(e_{2}^{m+1})=0.
\end{equation*}
\end{prop}
\begin{proof}
For Cases $(\textrm{i})-(\textrm{iii}),$ the proofs all employ equations $(\ref{eq:3.23})$ and $(\ref{eq:3.27})$ to extend the right-hand side of $(\ref{eq:3.28}),$ following the same method as in Proposition \ref{prop:1}. We omit the details.

As Cases $(\textrm{iv})-(\textrm{vi})$ follow similar arguments, we present only Case $(\textrm{iv}).$ The edges formed by vertices satisfy the conditions in Case $\textrm{III}.$ Therefore, by using equations $(\ref{eq:3.25})$ and $(\ref{eq:3.27}),$ we extend the right-hand side of $(\ref{eq:3.28})$ as follows:
\begin{align*}
&\qquad h_{1}^{m+1}([p_{i},p_{j}])+h_{1}^{m+1}([p_{j},p_{\bar{\ell}}])+h_{1}^{m+1}([p_{\bar{\ell}},p_{i}])\\
&=(\bar{b}_{i,\alpha}-\bar{b}_{j,\alpha})h_{1}^{m}([q_{i},q_{j}])+\sum_{\tilde{\ell}\neq j}r_{\tilde{\ell}}^{\alpha,6}h_{1}^{m}([q_{j},q_{\tilde{\ell}}])+\sum_{\tilde{\ell}\neq i}r_{\tilde{\ell}}^{\alpha,6}h_{1}^{m}([q_{\tilde{\ell}},q_{i}])\\
&=(\bar{b}_{i,\alpha}-\bar{b}_{j,\alpha})h_{1}^{m}([q_{i},q_{j}])
+\sum_{\tilde{\ell}\in\{\bar{k}_{j}\}\setminus\{i,j,\bar{\ell}\}}(\bar{b}_{\tilde{\ell},\alpha}-\bar{b}_{\tilde{\ell},\alpha})h_{1}^{m}([q_{j},q_{\tilde{\ell}}])
+(\bar{b}_{i,\alpha}-\bar{b}_{j,\alpha})h_{1}^{m}([q_{j},q_{i}])\\
&\quad+(\bar{b}_{\bar{\ell},\alpha}-\bar{b}_{\bar{\ell},\alpha})h_{1}^{m}([q_{j},q_{\bar{\ell}}])
+\sum_{\tilde{\ell}\notin\{\bar{k}_{j}\}}(\bar{b}_{\alpha}-\bar{b}_{\alpha})h_{1}^{m}([q_{j},q_{\tilde{\ell}}])
+\sum_{\tilde{\ell}\in\{\bar{k}_{j}\}\setminus\{i,j,\bar{\ell}\}}(\bar{b}_{\tilde{\ell},\alpha}-\bar{b}_{\tilde{\ell},\alpha})h_{1}^{m}([q_{\tilde{\ell}},q_{i}])\\
&\quad+(\bar{b}_{j,\alpha}-\bar{b}_{j,\alpha})h_{1}^{m}([q_{j},q_{i}])+(\bar{b}_{\bar{\ell},\alpha}-\bar{b}_{\bar{\ell},\alpha})h_{1}^{m}([q_{\bar{\ell}},q_{i}])
+\sum_{\tilde{\ell}\notin\{\bar{k}_{j}\}}(\bar{b}_{\alpha}-\bar{b}_{\alpha})h_{1}^{m}([q_{\tilde{\ell}},q_{i}])\\
&=(\bar{b}_{i,\alpha}+\bar{b}_{j,\alpha}-\bar{b}_{i,\alpha}-\bar{b}_{j,\alpha})h_{1}^{m}([q_{i},q_{j}])+(\bar{b}_{\bar{\ell},\alpha}-\bar{b}_{\bar{\ell},\alpha})h_{1}^{m}([q_{j},q_{\bar{\ell}}])
+(\bar{b}_{\bar{\ell},\alpha}-\bar{b}_{\bar{\ell},\alpha})h_{1}^{m}([q_{\bar{\ell}},q_{i}])\\
&\quad +\sum_{\tilde{\ell}\in\{\bar{k}_{j}\}\setminus\{i,j,\bar{\ell}\}}(\bar{b}_{\tilde{\ell},\alpha}-\bar{b}_{\tilde{\ell},\alpha})h_{1}^{m}([q_{j},q_{\tilde{\ell}}])
+\sum_{\tilde{\ell}\notin\{\bar{k}_{j}\}}(\bar{b}_{\alpha}-\bar{b}_{\alpha})h_{1}^{m}([q_{j},q_{\tilde{\ell}}])\\
&\quad +\sum_{\tilde{\ell}\in\{\bar{k}_{j}\}\setminus\{i,j,\bar{\ell}\}}(\bar{b}_{\tilde{\ell},\alpha}-\bar{b}_{\tilde{\ell},\alpha})h_{1}^{m}([q_{\tilde{\ell}},q_{i}])
+\sum_{\tilde{\ell}\notin\{\bar{k}_{j}\}}(\bar{b}_{\alpha}-\bar{b}_{\alpha})h_{1}^{m}([q_{\tilde{\ell}},q_{i}])\\
&=(\bar{b}_{i,\alpha}+\bar{b}_{j,\alpha}-\bar{b}_{i,\alpha}-\bar{b}_{j,\alpha})\Big(f_{0}^{m}(q_{j})-f_{0}^{m}(q_{i})\Big)+(\bar{b}_{\bar{\ell},\alpha}-\bar{b}_{\bar{\ell},\alpha})\Big(f_{0}^{m}(q_{\bar{\ell}})-f_{0}^{m}(q_{j})\Big)\\
&\quad +(\bar{b}_{\bar{\ell},\alpha}-\bar{b}_{\bar{\ell},\alpha})\Big(f_{0}^{m}(q_{i})-f_{0}^{m}(q_{\bar{\ell}})\Big)
+\sum_{\tilde{\ell}\in\{\bar{k}_{j}\}\setminus\{i,j,\bar{\ell}\}}(\bar{b}_{\tilde{\ell},\alpha}-\bar{b}_{\tilde{\ell},\alpha})\Big(f_{0}^{m}(q_{\tilde{\ell}})-f_{0}^{m}(q_{j})\Big)\\
&\quad +\sum_{\tilde{\ell}\notin\{\bar{k}_{j}\}}(\bar{b}_{\alpha}-\bar{b}_{\alpha})\Big(f_{0}^{m}(q_{\tilde{\ell}})-f_{0}^{m}(q_{j})\Big)
+\sum_{\tilde{\ell}\in\{\bar{k}_{j}\}\setminus\{i,j,\bar{\ell}\}}(\bar{b}_{\tilde{\ell},\alpha}-\bar{b}_{\tilde{\ell},\alpha})\Big(f_{0}^{m}(q_{i})-f_{0}^{m}(q_{\tilde{\ell}})\Big)\\
&\quad +\sum_{\tilde{\ell}\notin\{\bar{k}_{j}\}}(\bar{b}_{\alpha}-\bar{b}_{\alpha})\Big(f_{0}^{m}(q_{i})-f_{0}^{m}(q_{\tilde{\ell}})\Big)\\
&=\Big[\big(\bar{b}_{i,\alpha}+\bar{b}_{j,\alpha}-\bar{b}_{i,\alpha}-\bar{b}_{j,\alpha}\big)+(\bar{b}_{\bar{\ell},\alpha}-\bar{b}_{\bar{\ell},\alpha})
+\sum_{\tilde{\ell}\in\{\bar{k}_{j}\}\setminus\{i,j,\bar{\ell}\}}(\bar{b}_{\tilde{\ell},\alpha}-\bar{b}_{\tilde{\ell},\alpha})\\
&\qquad+(n+1-\alpha)(\bar{b}_{\alpha}-\bar{b}_{\alpha})\Big]f_{0}^{m}(q_{i})\\
&\quad+\Big[\big(\bar{b}_{i,\alpha}+\bar{b}_{j,\alpha}-\bar{b}_{i,\alpha}-\bar{b}_{j,\alpha}\big)-(\bar{b}_{\bar{\ell},\alpha}-\bar{b}_{\bar{\ell},\alpha})+\sum_{\tilde{\ell}\in\{\bar{k}_{j}\}\setminus\{i,j,\bar{\ell}\}}(\bar{b}_{\tilde{\ell},\alpha}-\bar{b}_{\tilde{\ell},\alpha})\\
&\qquad+(n+1-\alpha)(\bar{b}_{\alpha}-\bar{b}_{\alpha})\Big]f_{0}^{m}(q_{j})\\
&\quad +\Big((\bar{b}_{\bar{\ell},\alpha}-\bar{b}_{\bar{\ell},\alpha})-(\bar{b}_{\bar{\ell},\alpha}-\bar{b}_{\bar{\ell},\alpha})\Big)f_{0}^{m}(q_{\bar{\ell}})
+\sum_{\tilde{\ell}\in\{\bar{k}_{j}\}\setminus\{i,j,\bar{\ell}\}}\Big((\bar{b}_{\tilde{\ell},\alpha}-\bar{b}_{\tilde{\ell},\alpha})-(\bar{b}_{\tilde{\ell},\alpha}-\bar{b}_{\tilde{\ell},\alpha})\Big)f_{0}^{m}(q_{\tilde{\ell}})\\
&\quad +\sum_{\tilde{\ell}\notin\{\bar{k}_{j}\}}\Big((\bar{b}_{\alpha}-\bar{b}_{\alpha})-(\bar{b}_{\alpha}-\bar{b}_{\alpha})\Big)f_{0}^{m}(q_{\tilde{\ell}})\\
&=\Big[\Big(\sum_{\tilde{\ell}\in\{\bar{k}_{j}\}}\bar{b}_{\tilde{\ell},\alpha}
+(n+1-\alpha)\bar{b}_{\alpha}\Big)-\Big(\sum_{\tilde{\ell}\in\{\bar{k}_{j}\}}\bar{b}_{\tilde{\ell},\alpha}+(n+1-\alpha)\bar{b}_{\alpha}\Big)\Big]f_{0}^{m}(q_{i})\\
&\quad+\Big[\Big(\sum_{\tilde{\ell}\in\{\bar{k}_{j}\}}\bar{b}_{\tilde{\ell},\alpha}+(n+1-\alpha)\bar{b}_{\alpha}\Big) -\Big(\sum_{\tilde{\ell}\in\{\bar{k}_{j}\}}\bar{b}_{\tilde{\ell},\alpha}+(n+1-\alpha)\bar{b}_{\alpha}\Big)\Big]f_{0}^{m}(q_{j})\\
&=(1-1)f_{0}^{m}(q_{i})+(1-1)f_{0}^{m}(q_{j})\qquad\qquad\Big(\sum_{\tilde{\ell}\in\{\bar{k}_{j}\}}\bar{b}_{\tilde{\ell},\alpha}+(n+1-\alpha)\bar{b}_{\alpha}=1\Big)\\
&=0.
\end{align*}
\end{proof}

\begin{prop}\label{prop:3}
\setlist[enumerate]{itemsep=0pt, topsep=0pt, partopsep=0pt}
For any $h_{1}^{m}\in\mathcal{H}_{\ell,1}^{n,m},$ let $h_{1}^{m+1}$ be a $1$-form obtained by extending $h_{1}^{m}$ to the graph $G_{\ell}^{n,m+1}.$ For any triangle $e_{2}^{m+1}=[p_{i},p_{j},p_{\bar{\ell}}]\in E_{\ell,2}^{n,m+1},$ if $p_{i},p_{j}\in V_{\ell,\alpha}^{n,m+1},$ $p_{\bar{\ell}}\in V_{\ell,\alpha+1}^{n,m+1},$ then one of the following cases must hold:
\begin{enumerate}[label=(\roman*),itemsep=0pt, parsep=0pt]
\item one vertex belongs to $C_{\ell,\alpha,1}^{n,m+1},$ one vertex belongs to $C_{\ell,\alpha,2}^{n,m+1},$ and another vertex belongs to $C_{\ell,\alpha+1,1}^{n,m+1};$
\item two vertices belong to $C_{\ell,\alpha,2}^{n,m+1},$ and another vertex belongs to $C_{\ell,\alpha+1,1}^{n,m+1};$
\item three vertices satisfy the following cases:
  \begin{enumerate}[label=(\alph*),itemsep=0pt, parsep=0pt]
  \item one vertex belongs to $C_{\ell,\alpha,2}^{n,m+1},$ one vertex belongs to $B_{\ell,\alpha,1}^{n,m+1},$ and another vertex belongs to $C_{\ell,\alpha+1,2}^{n,m+1};$
  \item one vertex belongs to $B_{\ell,\alpha,2}^{n,m+1},$ one vertex belongs to $B_{\ell,\alpha,3,1}^{n,m+1},$ and another vertex belongs to $C_{\ell,\alpha+1,2}^{n,m+1};$
  \item one vertex belongs to $B_{\ell,\alpha,3,1}^{n,m+1},$ one vertex belongs to $B_{\ell,\alpha,3,2}^{n,m+1},$ and another vertex belongs to $C_{\ell,\alpha+1,2}^{n,m+1}.$
  \end{enumerate}
\item three vertices satisfy the following cases:
  \begin{enumerate}[label=(\alph*),itemsep=0pt, parsep=0pt]
  \item two vertices belong to $C_{\ell,\alpha,2}^{n,m+1},$ and another vertex belongs to $C_{\ell,\alpha+1,2}^{n,m+1};$
  \item two vertices belong to $B_{\ell,\alpha,2}^{n,m+1},$ and another vertex belongs to $C_{\ell,\alpha+1,2}^{n,m+1};$
  \item one vertex belongs to $B_{\ell,\alpha,2}^{n,m+1},$ one vertex belongs to $B_{\ell,\alpha,3,2}^{n,m+1},$ and another vertex belongs to $C_{\ell,\alpha+1,2}^{n,m+1};$
  \item two vertices belong to $B_{\ell,\alpha,3,2}^{n,m+1},$ and another vertex belongs to $C_{\ell,\alpha+1,2}^{n,m+1}.$
  \end{enumerate}
\end{enumerate}
If any of these cases holds, then
\begin{equation*}
d_{1}^{m+1}h_{1}^{m+1}(e_{2}^{m+1})=0.
\end{equation*}
\end{prop}
\begin{proof}
In the proofs of Cases $(\textrm{i})-(\textrm{iv}),$ we follow the same approach as in Case $(\textrm{iv})$ of Proposition \ref{prop:2}. We employ equations $(\ref{eq:3.20})$ $(\ref{eq:3.22})$ and $(\ref{eq:3.23})$ to extend the right-hand of $(\ref{eq:3.28}).$ For the third-to-last equation, we utilize the conditions
$$\sum_{\tilde{\ell}\in \{\bar{k}_{j}\}}\bar{b}_{\tilde{\ell},\alpha}+(n+1-\alpha)\bar{b}_{\alpha}=1$$
and
$$b_{1,\alpha+1}+\alpha b_{2,\alpha+1}+(n-\alpha)b_{3,\alpha+1}=1$$
to obtain the result.
\end{proof}

In the above propositions, we discuss the $1$-form $h_{1}^{m+1}$ on the triangles obtained by extending the harmonic $1$-form $h_{1}^{m}.$ Now, we turn our attention to $\delta_{1}^{m+1}h_{1}^{m+1}$ on the vertices. By the definition of the dual de Rham derivative $\delta_{1}^{m},$ we have
\begin{equation}\label{eq:3.29}
\begin{aligned}
\delta_{1}^{m+1}h_{1}^{m+1}(e_{0}^{m+1})&=\sum_{e_{1}^{m+1}\in E_{\ell,1}^{n,m+1}}\frac{\mu_{1}(e_{1}^{m+1})}{\mu_{0}(e_{0}^{m+1})}{\rm sgn}(e_{0}^{m+1},e_{1}^{m+1})h_{1}^{m+1}(e_{1}^{m+1})\\
&=\sum_{e_{1}^{m+1}\in E_{\ell,1}^{n,m+1}}\frac{\mu_{1}(e_{1}^{m+1})}{\mu_{0}(e_{0}^{m+1})}\times(-1)\times h_{1}^{m+1}(e_{1}^{m+1}),
\end{aligned}
\end{equation}
where the last equality holds because $e_{0}^{m+1}$ is the starting vertex of $e_{1}^{m+1},$ and thus $\rm{sgn} = -1$ by Definition \ref{def:2.1}. Next, we focus on the right-hand side of $(\ref{eq:3.29}).$

\begin{prop}\label{prop:4}
For any $h_{1}^{m}\in\mathcal{H}_{\ell,1}^{n,m},$ let $h_{1}^{m+1}$ be a $1$-form obtained by extending $h_{1}^{m}$ to the graph $G_{\ell}^{n,m+1}.$ For any vertex $e_{0}^{m+1}\in E_{\ell,0}^{n,m+1}\cap E_{\ell,0}^{n,m},$ we have
$$\delta_{1}^{m+1}h_{1}^{m+1}(e_{0}^{m+1})=0.$$
\end{prop}
\begin{proof}
For any $e_{0}^{m+1}\in E_{\ell,0}^{n,m+1}\cap E_{\ell,0}^{n,m},$ any vertex connected to $e_{0}^{m+1}$ belongs to $V_{\ell,2}^{n,m+1}.$ According to equation $(\ref{eq:3.10}),$ the right-hand side of $(\ref{eq:3.29})$ can be evaluated as follows:
\begin{equation}\label{eq:3.43}
\delta_{1}^{m+1}h_{1}^{m+1}(e_{0}^{m+1})=\sum_{e_{1}^{m+1}\in E_{\ell,1}^{n,m+1}}\frac{\mu_{1}(e_{1}^{m+1})}{\mu_{0}(e_{0}^{m+1})}\times(-1)\times\Big(\sum_{\tilde{\ell}\neq i}\tilde{a}_{\tilde{\ell}}h_{1}^{m}([q_{i},q_{\tilde{\ell}}])\Big).
\end{equation}

To expand the summation on the right-hand side of $(\ref{eq:3.43}),$ we proceed as follows:
\begin{equation*}
\sum_{\tilde{\ell}\neq i}\tilde{a}_{\tilde{\ell}}h_{1}^{m}([q_{i},q_{\tilde{\ell}}])+\sum_{\tilde{\ell}\neq i}\tilde{a}_{\tilde{\ell}}h_{1}^{m}([q_{i},q_{\tilde{\ell}}])+\cdots+\sum_{\tilde{\ell}\neq i}\tilde{a}_{\tilde{\ell}}h_{1}^{m}([q_{i},q_{\tilde{\ell}}])
\end{equation*}
There are a total of $n$ terms, and the coefficients in front of each edge are unequal. The above expression can be written as
\begin{align*}
&\quad a_{2}h_{1}^{m}([q_{i},q_{0}])+\sum_{\tilde{\ell}\neq 0,i}a_{3}h_{1}^{m}([q_{i},q_{\tilde{\ell}}])+a_{2}h_{1}^{m}([q_{i},q_{1}])+\sum_{\tilde{\ell}\neq 1,i}a_{3}h_{1}^{m}([q_{i},q_{\tilde{\ell}}])+\cdots\\
&\quad+a_{2}h_{1}^{m}([q_{i},q_{i-1}])+\sum_{\tilde{\ell}\neq i-1,i}a_{3}h_{1}^{m}([q_{i},q_{\tilde{\ell}}])+a_{2}h_{1}^{m}([q_{i},q_{i+1}])+\sum_{\tilde{\ell}\neq i,i+1}a_{3}h_{1}^{m}([q_{i},q_{\tilde{\ell}}])\\
&\quad +\cdots+a_{2}h_{1}^{m}([q_{i},q_{n}])+\sum_{\tilde{\ell}\neq i,n}a_{3}h_{1}^{m}([q_{i},q_{\tilde{\ell}}])\\
&=\Big(a_{2}+(n-1)a_{3}\Big)h_{1}^{m}([q_{i},q_{0}])+\Big(a_{2}+(n-1)a_{3}\Big)h_{1}^{m}([q_{i},q_{1}])+\cdots\\
&\quad+\Big(a_{2}+(n-1)a_{3}\Big)h_{1}^{m}([q_{i},q_{i-1}])+\Big(a_{2}+(n-1)a_{3}\Big)h_{1}^{m}([q_{i},q_{i+1}])\\
&\quad+\cdots+\Big(a_{2}+(n-1)a_{3}\Big)h_{1}^{m}([q_{i},q_{n}])\\
&=\Big(a_{2}+(n-1)a_{3}\Big)\sum_{\tilde{\ell}\neq i}h_{1}^{m}([q_{i},q_{\tilde{\ell}}]).
\end{align*}
Substitute the above equality into $(\ref{eq:3.29}),$ we have
\begin{align*}
\delta_{1}^{m+1}h_{1}^{m+1}(e_{0}^{m+1})&=-\frac{\mu_{1}(e_{1}^{m+1})}{\mu_{0}(e_{0}^{m+1})}\Big(a_{2}+(n-1)a_{3}\Big)\sum_{\tilde{\ell}\neq i}h_{1}^{m}([q_{i},q_{\tilde{\ell}}])\\
&=-\Big(\frac{\mu_{1}(e_{1}^{0})}{\mu_{0}(e_{0}^{0})}\Big)^{m+1}\Big(a_{2}+(n-1)a_{3}\Big)\sum_{\tilde{\ell}\neq i}h_{1}^{m}([q_{i},q_{\tilde{\ell}}])\\
&=\frac{\mu_{1}(e_{1}^{0})}{\mu_{0}(e_{0}^{0})}\Big(a_{2}+(n-1)a_{3}\Big)\sum_{e_{1}^{m}\in E_{\ell,1}^{n,m}}\frac{\mu_{1}(e_{1}^{m})}{\mu_{0}(e_{0}^{m})}{\rm sgn}(e_{0}^{m},e_{1}^{m})h_{1}^{m}(e_{1}^{m})\\
&=\frac{\mu_{1}(e_{1}^{0})}{\mu_{0}(e_{0}^{0})}\Big(a_{2}+(n-1)a_{3}\Big)\delta_{1}^{m}h_{1}^{m}(e_{0}^{m})\quad(\text{by the definition of $\delta_{1}^{m}$})\\
&=0.\quad(h_{1}^{m}\in\mathcal{H}_{\ell,1}^{n,m})
\end{align*}
\end{proof}

\begin{prop}\label{prop:5}
\setlist[enumerate]{itemsep=0pt, topsep=0pt, partopsep=0pt}
For any $h_{1}^{m}\in\mathcal{H}_{\ell,1}^{n,m},$ let $h_{1}^{m+1}$ be a $1$-form obtained by extending $h_{1}^{m}$ to the graph $G_{\ell}^{n,m+1}.$ For any vertex $e_{0}^{m+1}\in V_{\ell,2}^{n,m+1},$ we consider the following cases based on its connections:
\begin{enumerate}[label=(\roman*),itemsep=0pt, parsep=0pt]
\item $e_{0}^{m+1}$ is connected to at least one vertex in $V_{0};$
\item $e_{0}^{m+1}$ is not connected to any vertex in $V_{0}.$
\end{enumerate}
If either of the cases holds, then
$$\delta_{1}^{m+1}h_{1}^{m+1}(e_{0}^{m+1})=0.$$
\end{prop}
\begin{proof}
Let $y$ be a vertex connected to $e_{0}^{m+1}.$ We first prove Case $(\textrm{i}).$ By assumption, the vertices connected to $e_{0}^{m+1}$ belong to $V_{0}\cup V_{\ell,2}^{n,m+1}\cup V_{\ell,3}^{n,m+1}.$ Note that when $\ell\geq 3,$ $e_{0}^{m+1}$ is connected to the vertices in $V_{\ell,3}^{n,m+1}.$ Therefore, we divide the proof into following two parts.

\noindent\textit{Part $1.$} Let $\ell=2.$ The number of vertices in $V_{0}$ connected to $e_{0}^{m+1}$ is $2,$ and the number of vertices in $V_{\ell,2}^{n,m+1}$ connected to $e_{0}^{m+1}$ is $2(n-1).$ Moreover, when $\ell=2,$ $a_{1}=a_{2}.$

By using equations $(\ref{eq:3.10})$ and $(\ref{eq:3.11}),$ we expand the right-hand side of $(\ref{eq:3.29})$ as follows:
\begin{align*}
&\qquad\sum_{y\in V_{0}}h_{1}^{m+1}(e_{1}^{m+1})+\sum_{y\in V_{\ell,2}^{n,m+1}}h_{1}^{m+1}(e_{1}^{m+1})\\
&=\sum_{y\in V_{0}}\sum_{\tilde{\ell}\neq i}\tilde{a}_{\tilde{\ell}}h_{1}^{m}([q_{\tilde{\ell}},q_{i}])+\sum_{y\in V_{\ell,2}^{n,m+1}}(a_{2}-a_{3})h_{1}^{m}([q_{j},q_{k}])\\
&=a_{2}h_{1}^{m}([q_{j},q_{i}])+\sum_{\tilde{\ell}\neq i,j}a_{3}h_{1}^{m}([q_{\tilde{\ell}},q_{i}])+a_{1}h_{1}^{m}([q_{i},q_{j}])+\sum_{\tilde{\ell}\neq i,j}a_{3}h_{1}^{m}([q_{\tilde{\ell}},q_{j}])\\
&\quad+\sum_{\tilde{\ell}\neq i,j}(a_{2}-a_{3})h_{1}^{m}([q_{j},q_{\tilde{\ell}}])+\sum_{\tilde{\ell}\neq i,j}(a_{1}-a_{3})h_{1}^{m}([q_{i},q_{\tilde{\ell}}])\\
&=a_{2}h_{1}^{m}([q_{j},q_{i}])+\sum_{\tilde{\ell}\neq i}a_{3}h_{1}^{m}([q_{\tilde{\ell}},q_{i}])-a_{3}h_{1}^{m}([q_{j},q_{i}])+a_{1}h_{1}^{m}([q_{i},q_{j}])\\
&\quad+\sum_{\tilde{\ell}\neq j}a_{3}h_{1}^{m}([q_{\tilde{\ell}},q_{j}])-a_{3}h_{1}^{m}([q_{i},q_{j}])+\sum_{\tilde{\ell}\neq j}(a_{2}-a_{3})h_{1}^{m}([q_{j},q_{\tilde{\ell}}])-(a_{2}-a_{3})h_{1}^{m}([q_{j},q_{i}])\\
&\quad+\sum_{\tilde{\ell}\neq i}(a_{1}-a_{3})h_{1}^{m}([q_{i},q_{\tilde{\ell}}])-(a_{1}-a_{3})h_{1}^{m}([q_{i},q_{j}])\\
&=\sum_{\tilde{\ell}\neq i}a_{3}h_{1}^{m}([q_{\tilde{\ell}},q_{i}])+\sum_{\tilde{\ell}\neq j}a_{3}h_{1}^{m}([q_{\tilde{\ell}},q_{j}])+\sum_{\tilde{\ell}\neq j}(a_{2}-a_{3})h_{1}^{m}([q_{j},q_{\tilde{\ell}}])+\sum_{\tilde{\ell}\neq i}(a_{1}-a_{3})h_{1}^{m}([q_{i},q_{\tilde{\ell}}]).
\end{align*}
Substitute the above equality into $(\ref{eq:3.29}),$ we have
\begin{align*}
\delta_{1}^{m+1}h_{1}^{m+1}(e_{0}^{m+1})&=-\frac{\mu_{1}(e_{1}^{m+1})}{\mu_{0}(e_{0}^{m+1})}\Big(\sum_{\tilde{\ell}\neq i}a_{3}h_{1}^{m}([q_{\tilde{\ell}},q_{i}])+\sum_{\tilde{\ell}\neq j}a_{3}h_{1}^{m}([q_{\tilde{\ell}},q_{j}])\\
&\quad+\sum_{\tilde{\ell}\neq j}(a_{2}-a_{3})h_{1}^{m}([q_{j},q_{\tilde{\ell}}])+\sum_{\tilde{\ell}\neq i}(a_{1}-a_{3})h_{1}^{m}([q_{i},q_{\tilde{\ell}}])\Big)\\
&=-\Big(\frac{\mu_{1}(e_{1}^{0})}{\mu_{0}(e_{0}^{0})}\Big)^{m+1}\Big(\sum_{\tilde{\ell}\neq i}a_{3}h_{1}^{m}([q_{\tilde{\ell}},q_{i}])+\sum_{\tilde{\ell}\neq j}a_{3}h_{1}^{m}([q_{\tilde{\ell}},q_{j}])\\
&\quad+\sum_{\tilde{\ell}\neq j}(a_{2}-a_{3})h_{1}^{m}([q_{j},q_{\tilde{\ell}}])+\sum_{\tilde{\ell}\neq i}(a_{1}-a_{3})h_{1}^{m}([q_{i},q_{\tilde{\ell}}])\Big)\\
&=\frac{\mu_{1}(e_{1}^{0})}{\mu_{0}(e_{0}^{0})}\Big(-a_{3}\delta_{1}^{m}h_{1}^{m}(q_{i})-a_{3}\delta_{1}^{m}h_{1}^{m}(q_{j})\\
&\qquad+(a_{2}-a_{3})\delta_{1}^{m}h_{1}^{m}(q_{j})
+(a_{1}-a_{3})\delta_{1}^{m}h_{1}^{m}(q_{i})\Big)\quad(\text{by the definition of $\delta_{1}^{m}$})\\
&=0.\quad(h_{1}^{m}\in\mathcal{H}_{\ell,1}^{n,m})
\end{align*}

\noindent\textit{Part $2.$} Let $\ell\geq 3.$ We have $a_{1}\neq a_{2}.$ In this case, the number of vertices in $V_{0}$ connected to $e_{0}^{m+1}$ is $1,$ the number of vertices in $V_{\ell,2}^{n,m+1}$ connected to $e_{0}^{m+1}$ is $n,$ and the number of vertices in $V_{\ell,3}^{n,m+1}$ connected to $e_{0}^{m+1}$ is $(n-1).$

According to equations $(\ref{eq:3.10}),(\ref{eq:3.11}),(\ref{eq:3.12})$ and $(\ref{eq:3.13}),$ we extend the right-hand side of $(\ref{eq:3.29})$ as follows:
\begin{align*}
&\qquad\sum_{y\in V_{0}}h_{1}^{m}(e_{1}^{m+1})+\sum_{y\in V_{\ell,2}^{n,m+1}}h_{1}^{m}(e_{1}^{m+1})+\sum_{y\in V_{\ell,3}^{n,m+1}}h_{1}^{m}(e_{1}^{m+1})\\
&=\sum_{\tilde{\ell}\neq i}\tilde{a}_{\tilde{\ell}}h_{1}^{m}([q_{\tilde{\ell}},q_{i}])+\sum_{\substack{y\in V_{\ell,2}^{n,m+1}\\y\sim z,z\in V_{0}}}(a_{2}-a_{3})h_{1}^{m}([q_{j},q_{k}])
+\sum_{\tilde{\ell}\neq i}\tilde{a}_{\tilde{\ell}}^{1}h_{1}^{m}([q_{i},q_{\tilde{\ell}}])\\
&\quad+\sum_{y\in V_{\ell,3}^{n,m+1}}\sum_{\tilde{\ell}\neq i}\tilde{a}_{\tilde{\ell}}^{2}h_{1}^{m}([q_{i},q_{\tilde{\ell}}])\\
&=a_{2}h_{1}^{m}([q_{j},q_{i}])+\sum_{\tilde{\ell}\neq i,j}a_{3}h_{1}^{m}([q_{\tilde{\ell}},q_{i}])+\sum_{\tilde{\ell}\neq i,j}(a_{2}-a_{3})h_{1}^{m}([q_{j},q_{\tilde{\ell}}])+(\bar{a}_{2}-a_{2})h_{1}^{m}([q_{i},q_{j}])\\
&\quad+\sum_{\tilde{\ell}\neq i,j}(\bar{a}_{3}-a_{3})h_{1}^{m}([q_{i},q_{\tilde{\ell}}])+\sum_{\bar{\ell}\neq i,j}\Big[(b_{2,3}-a_{2})h_{1}^{m}([q_{i},q_{j}])+(b_{2,3}-a_{3})h_{1}^{m}([q_{i},q_{\bar{\ell}}])\\
&\quad+\sum_{\tilde{\ell}\neq i,j,\bar{\ell}}(b_{3,3}-a_{3})h_{1}^{m}([q_{i},q_{\tilde{\ell}}])\Big]\\
&=a_{2}\Big(f_{0}^{m}(q_{i})-f_{0}^{m}(q_{j})\Big)+\sum_{\tilde{\ell}\neq i,j}a_{3}\Big(f_{0}^{m}(q_{i})-f_{0}^{m}(q_{\tilde{\ell}})\Big)+\sum_{\tilde{\ell}\neq i,j}(a_{2}-a_{3})\Big(f_{0}^{m}(q_{\tilde{\ell}})-f_{0}^{m}(q_{j})\Big)\\
&\quad+(\bar{a}_{2}-a_{2})\Big(f_{0}^{m}(q_{j})-f_{0}^{m}(q_{i})\Big)+\sum_{\tilde{\ell}\neq i,j}(\bar{a}_{3}-a_{3})\Big(f_{0}^{m}(q_{\tilde{\ell}})-f_{0}^{m}(q_{i})\Big)\\
&\quad+\sum_{\bar{\ell}\neq i,j}\Big[(b_{2,3}-a_{2})\Big(f_{0}^{m}(q_{j})-f_{0}^{m}(q_{i})\Big)+(b_{2,3}-a_{3})\Big(f_{0}^{m}(q_{\bar{\ell}})-f_{0}^{m}(q_{i})\Big)\\
&\qquad+\sum_{\tilde{\ell}\neq i,j,\bar{\ell}}(b_{3,3}-a_{3})\Big(f_{0}^{m}(q_{\tilde{\ell}})-f_{0}^{m}(q_{i})\Big)\Big]\\
&=\Big[a_{2}+(n-1)a_{3}-(\bar{a}_{2}-a_{2})-(n-1)(\bar{a}_{3}-a_{3})+(n-1)\Big(-(b_{2,3}-a_{2})\\
&\qquad-(b_{2,3}-a_{3})-(n-2)(b_{3,3}-a_{3})\Big)\Big]f_{0}^{m}(q_{i})\\
&\quad+\Big[-a_{2}-(n-1)(a_{2}-a_{3})+(\bar{a}_{2}-a_{2})+(n-1)(b_{2,3}-a_{2})\Big]f_{0}^{m}(q_{j})\\
&\quad+\sum_{\tilde{\ell}\neq i,j}\Big[-a_{3}+(a_{2}-a_{3})+(\bar{a}_{3}-a_{3})+(b_{2,3}-a_{3})+(n-2)(b_{3,3}-a_{3})\Big]f_{0}^{m}(q_{\tilde{\ell}})\\
&=\Big(\bar{a}_{1}+(n-1)b_{1,3}+1-(n+1)a_{1}\Big)f_{0}^{m}(q_{i})
+\Big(\bar{a}_{2}+(n-1)a_{3}+(n-1)b_{2,3}\\
&\quad-2na_{2}\Big)f_{0}^{m}(q_{j})+\sum_{\tilde{\ell}\neq i,j}\Big(a_{2}+\bar{a}_{3}+b_{2,3}+(n-2)b_{3,3}-(n+2)a_{3}\Big)f_{0}^{m}(q_{\tilde{\ell}}).\\
\end{align*}

From the solution of the linear system of equations $AX = B,$ we know that the value at any vertex in $E_{\ell,0}^{n,1}\setminus V_{0}$ can be written as
$$h(e_{0}^{m}) = \sum_{i=0}^{n} s_{i} h(q_{i}),$$
where $s_{i}$ are coefficients determined by the matrix $C.$ The solution $X$ can therefore be represented as $X = CB_{1},$ where $B_{1}=(h(q_{0}),h(q_{1}),...,h(q_{n})).$ The coefficients in front of $h(q_{\tilde{\ell}}),$ excluding the vertex $q_{i},$ in the solution correspond to a certain row of $-AC.$ The coefficient in front of $f_{0}^{m}(q_{i})$ involves only the sum of $a_{1},\bar{a}_{1}$ and $b_{1,3},$ which is the element of $-AC.$ By construction, this row is equal to the corresponding row of $-C_{1},$ where $C_{1}$ is a diagonal matrix with:
\begin{equation*}
(C_{1})_{ii}=\begin{cases}
1, &\quad\quad\text{if the vertex $q_{i}$ is connected to any vertex in $V_{0},$}\\
0, &\quad\quad\text{otherwise.}
\end{cases}
\end{equation*}

Since this vertex is connected to the vertex in $V_{0},$ thus $(C_{1})_{ii}=1,$ namely,
$$\bar{a}_{1}+(n-2)b_{1,3}-(n+1)a_{1}=-1.$$
Therefore, we have
$$\bar{a}_{1}+(n-1)b_{1,3}+1-(n+1)a_{1}=-1+1=0,$$
$$\bar{a}_{2}+(n-1)a_{3}+(n-1)b_{2,3}-2na_{2}=0,$$
and
$$a_{2}+\bar{a}_{3}+b_{2,3}+(n-2)b_{3,3}-(n+2)a_{3}=0.$$
Therefore, we obtain
$$\delta_{1}^{m+1}h_{1}^{m+1}(e_{0}^{m+1})=0.$$

Next, we prove Case $(\textrm{ii}).$ We divide the proof into three subcases:

\noindent\textit{Case $1.$} One vertex in $V_{\ell,2}^{n,m+1}$ connected to $e_{0}^{m+1}$ is not connected to any vertex in $V_{0},$ while the other vertex in $V_{\ell,2}^{n,m+1}$ connected to $e_{0}^{m+1}$ is connected to some vertex in $V_{0}.$

In Case $1,$ we only consider the case $\ell$ is odd. By using equations $(\ref{eq:3.12}),(\ref{eq:3.14})$ and $(\ref{eq:3.16}),$ we can extend the right-hand side of $(\ref{eq:3.29})$ as follows:
\begin{equation}\label{eq:3.42}
\begin{aligned}
&\qquad\sum_{y\in V_{\ell,2}^{n,m+1}}h_{1}^{m+1}(e_{1}^{m+1})+\sum_{y\in V_{\ell,3}^{n,m+1}}h_{1}^{m+1}(e_{1}^{m+1})\\
&=\sum_{\tilde{\ell}\neq i}\tilde{a}_{\tilde{\ell}}^{1}h_{1}^{m}([q_{\tilde{\ell}},q_{i}])+(\bar{a}_{1}-\bar{a}_{2})h_{1}^{m}([q_{i},q_{j}])
+\sum_{y\in V_{\ell,3}^{n,m+1}}\sum_{\tilde{\ell}\neq i}\tilde{a}_{\tilde{\ell}}^{5}h_{1}^{m}([q_{\tilde{\ell}},q_{i}])\\
&=\Big(a_{1}+\bar{a}_{2}+(n-1)b_{1,3}-(n+1)\bar{a}_{1}\Big)f_{0}^{m}(q_{i})+\Big(\bar{a}_{1}+a_{2}+(n-1)b_{2,3}\\
&\quad-(n+1)\bar{a}_{2}\Big)f_{0}^{m}(q_{j})+\sum_{\tilde{\ell}\neq i,j}\Big(a_{3}+b_{2,3}+(n-2)b_{3,3}-n\bar{a}_{3}\Big)f_{0}^{m}(q_{\tilde{\ell}}).
\end{aligned}
\end{equation}

By analyzing the proof of Part $2$ of Case $(\textrm{i}),$ we observe that each coefficient preceding $f_{0}^{m}(q_{\tilde{\ell}})$ in the above formula corresponds to an entry in a row of the matrix $AC.$ Since $e_{0}^{m+1}$ is not connected to any vertex in $V_{0},$ it follows that
$$a_{1}+\bar{a}_{2}+(n-1)b_{1,3}-(n+1)\bar{a}_{1}=0,$$
$$\bar{a}_{1}+a_{2}+(n-1)b_{2,3}-(n+1)\bar{a}_{2}=0,$$
and
$$a_{3}+b_{2,3}+(n-2)b_{3,3}-n\bar{a}_{3}=0.$$
Therefore, the last expression of (\ref{eq:3.42}) equals zero.

\noindent\textit{Case $2.$} Neither of the vertices in $V_{\ell,2}^{n,m+1}$ connected to $e_{0}^{m+1}$ is not connected to any vertex in $V_{0}.$

\noindent\textit{Case $3.$} $\ell$ is even, and $e_{0}^{m+1}$ is the midpoint of the edge $[q_{i},q_{j}].$

Since the proofs of the remaining cases follow the same approach, we omit the details.
\end{proof}

\begin{prop}\label{prop:6}
For any $h_{1}^{m}\in\mathcal{H}_{\ell,1}^{n,m},$ let $h_{1}^{m+1}$ be a $1$-form obtained by extending $h_{1}^{m}$ to the graph $G_{\ell}^{n,m+1}.$ For any vertex $e_{0}^{m+1}\in V_{\ell,\alpha}^{n,m+1},$ if $e_{0}^{m+1}\in C_{\ell,\alpha,1}^{n,m+1},$ then
$$\delta_{1}^{m+1}h_{1}^{m+1}(e_{0}^{m+1})=0.$$
\end{prop}
\begin{proof}
Let $y$ be a vertex connected to $e_{0}^{m+1}.$ By assumption, there are $\alpha n$ vertices connected to $e_{0}^{m+1},$ and these vertices belong to $V_{\ell,\alpha-1}^{n,m+1}\cup V_{\ell,\alpha}^{n,m+1}\cup V_{\ell,\alpha+1}^{n,m+1}.$ Precisely, the number of vertices in $V_{\ell,\alpha-1}^{n,m+1}$ connected to $e_{0}^{m+1}$ is $(\alpha-1)^{2},$ the number of vertices in $V_{\ell,\alpha}^{n,m+1}$ connected to $e_{0}^{m+1}$ is $(\alpha-1)(n+2-\alpha),$ and the number of vertices in $V_{\ell,\alpha+1}^{n,m+1}$ connected to $e_{0}^{m+1}$ is $(n+1-\alpha).$

By using equations $(3,19),(\ref{eq:3.20}),(\ref{eq:3.22})$ and $(\ref{eq:3.23}),$ we extend the right-hand side of $(\ref{eq:3.29})$ as the approach in Case $(\textrm{ii})$ in Proposition \ref{prop:5}.
\end{proof}

\begin{prop}\label{prop:7}
\setlist[enumerate]{itemsep=0pt, topsep=0pt, partopsep=0pt}
For any $h_{1}^{m}\in\mathcal{H}_{\ell,1}^{n,m},$ let $h_{1}^{m+1}$ be a $1$-form obtained by extending $h_{1}^{m}$ to the graph $G_{\ell}^{n,m+1}.$ For any vertex $e_{0}^{m+1}\in C_{\ell,\alpha,2}^{n,m+1},$ we consider the following cases based on its connections:
\begin{enumerate}[label=(\roman*),itemsep=0pt, parsep=0pt]
\item $e_{0}^{m+1}$ is connected to some vertex in $B_{\ell,\alpha,1}^{n,m+1}$ but is not connected to any vertex in $B_{\ell,\alpha,2}^{n,m+1};$
\item $e_{0}^{m+1}$ is connected to some vertex in $B_{\ell,\alpha,1}^{n,m+1}\cup B_{\ell,\alpha,2}^{n,m+1};$
\item $e_{0}^{m+1}$ is connected to some vertex in $B_{\ell,\alpha,2}^{n,m+1}$ but is not connected to any vertex in $B_{\ell,\alpha,1}^{n,m+1}.$
\end{enumerate}
If any of these cases holds, then
$$\delta_{1}^{m+1}h_{1}^{m+1}(e_{0}^{m+1})=0.$$
\end{prop}
\begin{proof}
The vertices connected to $e_{0}^{m+1}$ belong to $V_{\ell,\alpha-1}^{n,m+1}\cup V_{\ell,\alpha}^{n,m+1}\cup V_{\ell,\alpha+1}^{n,m+1}.$  In Case $(\textrm{i}),$ we use equations $(\ref{eq:3.22})-(\ref{eq:3.27});$ in Case $(\textrm{ii}),$ we employ equations $(\ref{eq:3.21}),$ and $(\ref{eq:3.24})-(\ref{eq:3.27});$ for Case $(\textrm{iii}),$ we utilize equations $(\ref{eq:3.25})-(\ref{eq:3.27}).$ We can use the method in the proof of Case $(\textrm{ii})$ in Proposition \ref{prop:5} to complete the subsequent calculations.
\end{proof}

\begin{prop}\label{prop:8}
For any $h_{1}^{m}\in\mathcal{H}_{\ell,1}^{n,m},$ let $h_{1}^{m+1}$ be a $1$-form obtained by extending $h_{1}^{m}$ onto the graph $G_{\ell}^{n,m+1}.$ If $e_{0}^{m+1}\in B_{\ell,\alpha,1}^{n,m+1},$ then
$$\delta_{1}^{m+1}h_{1}^{m+1}(e_{0}^{m+1})=0.$$
\end{prop}
\begin{proof}
Let $y$ be a vertex connected to $e_{0}^{m+1}.$ By assumption, there are $\alpha n$ vertices connected to $e_{0}^{m+1},$ and they belong to $ V_{\ell,\alpha}^{n,m+1}\cup V_{\ell,\alpha+1}^{n,m+1}.$ Precisely, the number of vertices in $V_{\ell,\alpha}^{n,m+1}$ connected to $e_{0}^{m+1}$ is $\alpha(\alpha-1),$ and the number of vertices in $V_{\ell,\alpha+1}^{n,m+1}$ connected to $e_{0}^{m+1}$ is $\alpha(n+1-\alpha).$

By using equations $(\ref{eq:3.21}),(\ref{eq:3.23})$ and $(\ref{eq:3.24}),$ we extend the right-hand side of $(\ref{eq:3.29})$ as follows:
\begin{align*}
&\qquad\sum_{y\in V_{\ell,\alpha}^{n,m+1}}h_{1}^{m+1}(e_{1}^{m+1})+\sum_{y\in V_{\ell,\alpha+1}^{n,m+1}}h_{1}^{m+1}(e_{1}^{m+1})\\
&=\sum_{y\in C_{\ell,\alpha,2}^{n,m+1}}h_{1}^{m+1}(e_{1}^{m+1})+\sum_{y\in B_{\ell,\alpha,2}^{n,m+1}}h_{1}^{m+1}(e_{1}^{m+1})
+\sum_{y\in C_{\ell,\alpha+1,2}^{n,m+1}}h_{1}^{m+1}(e_{1}^{m+1})\\
&=\sum_{y\in C_{\ell,\alpha,2}^{n,m+1}}\sum_{\tilde{\ell}\neq i}r_{\tilde{\ell}}^{\alpha,2}h_{1}^{m}([q_{\tilde{\ell}},q_{i}])+\sum_{y\in C_{\ell,\alpha,2}^{n,m+1}}\sum_{\tilde{\ell}\neq i}r_{\tilde{\ell}}^{\alpha,4}h_{1}^{m}([q_{\tilde{\ell}},q_{i}])\\
&\quad+\sum_{y\in B_{\ell,\alpha,2}^{n,m+1}}\sum_{\tilde{\ell}\neq i}r_{\tilde{\ell}}^{\alpha,4}h_{1}^{m}([q_{i},q_{\tilde{\ell}}])
+\sum_{y\in C_{\ell,\alpha+1,2}^{n,m+1}}\sum_{\tilde{\ell}\neq i}r_{\tilde{\ell}}^{\alpha,5}h_{1}^{m}([q_{i},q_{\tilde{\ell}}])\\
&=\Big[(\alpha-1)\bar{b}_{i,\alpha}+(\alpha-2)(\alpha-1)\bar{b}_{i,\alpha}+(\alpha-1)\bar{b}_{i,\alpha}+\alpha(n+1-\alpha)\bar{b}_{i,\alpha+1}-\alpha nb_{1,\alpha}\Big]f_{0}^{m}(q_{i})\\
&\quad+\sum_{\tilde{\ell}\in\{\bar{k}_{j}\}\setminus\{i\}}\Big[\sum_{\hat{\ell}\in\{\bar{k}_{j}\}\setminus\{i\}}\bar{b}_{\hat{\ell},\alpha}
+(\alpha-2)\sum_{\hat{\ell}\in\{\bar{k}_{j}\}\setminus\{i\}}\bar{b}_{\hat{\ell},\alpha}+\sum_{\hat{\ell}\in\{\bar{k}_{j}\}\setminus\{i\}}\bar{b}_{\hat{\ell},\alpha}
+\alpha(n+1-\alpha)\bar{b}_{\tilde{\ell},\alpha+1} \\
&\qquad-\alpha nb_{2,\alpha}\Big]f_{0}^{m}(q_{\tilde{\ell}})\\
&\quad+\sum_{\tilde{\ell}\notin\{\bar{k}_{j}\}}\Big[(\alpha-1)\bar{b}_{\alpha}+(\alpha-2)(\alpha-1)\bar{b}_{\alpha}+(\alpha-1)\bar{b}_{\alpha}+\alpha\bar{b}_{\tilde{\ell},\alpha+1}+\alpha(n-\alpha)\bar{b}_{\alpha+1}\\
&\qquad-\alpha nb_{3,\alpha}\Big]f_{0}^{m}(q_{\tilde{\ell}}).
\end{align*}

The following calculations follow the same method in Proposition \ref{prop:5}, so we omit the details.
\end{proof}

\begin{prop}\label{prop:9}
\setlist[enumerate]{itemsep=0pt, topsep=0pt, partopsep=0pt}
For any $h_{1}^{m}\in\mathcal{H}_{\ell,1}^{n,m},$ let $h_{1}^{m+1}$ be a $1$-form obtained by extending $h_{1}^{m}$ to the graph $G_{\ell}^{n,m+1}.$ For any vertex $e_{0}^{m+1}\in B_{\ell,\alpha,2}^{n,m+1},$ we consider the following cases based on its connections:
\begin{enumerate}[label=(\roman*),itemsep=0pt, parsep=0pt]
\item $e_{0}^{m+1}$ is connected to some vertex in $B_{\ell,\alpha,1}^{n,m+1};$
\item $e_{0}^{m+1}$ is not connected to any vertex in $B_{\ell,\alpha,1}^{n,m+1};$
\item $e_{0}^{m+1}$ is connected to some vertex in $B_{\ell,\alpha,3,1}^{n,m+1};$
\item $e_{0}^{m+1}$ is connected to some vertex in $B_{\ell,\alpha,3,2}^{n,m+1}.$
\end{enumerate}
If any of these four cases holds, then
$$\delta_{1}^{m+1}h_{1}^{m+1}(e_{0}^{m+1})=0.$$
\end{prop}
\begin{proof}
The connectivity analysis for $e_{0}^{m+1}\in B_{\ell,\alpha,2}^{n,m+1}$ is consistent with that of Proposition \ref{prop:8}; we omit the proof.
\end{proof}

\begin{prop}\label{prop:10}
For any $h_{1}^{m}\in\mathcal{H}_{\ell,1}^{n,m},$ let $h_{1}^{m+1}$ be a $1$-form obtained by extending $h_{1}^{m}$ to the graph $G_{\ell}^{n,m+1}.$ For any vertex $e_{0}^{m+1}\in B_{\ell,\alpha,3,1}^{n,m+1},$ we have
$$\delta_{1}^{m+1}h_{1}^{m+1}(e_{0}^{m+1})=0.$$
\end{prop}
\begin{proof}
The connectivity analysis for $e_{0}^{m+1}\in B_{\ell,\alpha,3,1}^{n,m+1}$ is consistent with that of Proposition \ref{prop:8}. However, in this case, the vertex is not connected to any vertex in $C_{\ell,\alpha}^{n,m+1}.$ Applying the method of Proposition \ref{prop:5} to extend (\ref{eq:3.29}) via $(\ref{eq:3.21})-(\ref{eq:3.24})$ yields the result; we omit the details.
\end{proof}

\begin{prop}\label{prop:11}
\setlist[enumerate]{itemsep=0pt, topsep=0pt, partopsep=0pt}
For any $h_{1}^{m}\in\mathcal{H}_{\ell,1}^{n,m},$ let $h_{1}^{m+1}$ be a $1$-form obtained by extending $h_{1}^{m}$ to the graph $G_{\ell}^{n,m+1}.$ For any vertex $e_{0}^{m+1}\in B_{\ell,\alpha,3,2}^{n,m+1},$ we consider the following cases based on its connections:
\begin{enumerate}[label=(\roman*),itemsep=0pt, parsep=0pt]
\item $e_{0}^{m+1}$ is connected to some vertex in $B_{\ell,\alpha,3,1}^{n,m+1};$
\item $e_{0}^{m+1}$ is not connected to any vertex in $B_{\ell,\alpha,3,1}^{n,m+1};$
\item $e_{0}^{m+1}$ is not connected to some vertex in $B_{\ell,\alpha,2}^{n,m+1}.$
\end{enumerate}
If any of these cases holds, then
$$\delta_{1}^{m+1}h_{1}^{m+1}(e_{0}^{m+1})=0.$$
\end{prop}
\begin{proof}
The connectivity analysis for the vertex $e_{0}^{m+1}\in B_{\ell,\alpha,3,2}^{n,m+1}$ is consistent with that of Proposition \ref{prop:8}. To extend the right-hand side of equation $(\ref{eq:3.29}),$ in Case $(\textrm{i}),$ we apply equations $(\ref{eq:3.23})$ and $(\ref{eq:3.25})-(\ref{eq:3.27});$ in Cases $(\textrm{ii})$ and $(\textrm{iii}),$ we use equations $(\ref{eq:3.25})-(\ref{eq:3.27}).$ The proofs follow the same methodology as Proposition \ref{prop:5}; we omit the details.
\end{proof}

From these propositions, we can prove Theorem \ref{thm:1.1}.
\begin{proof}[Proof of Theorem \ref{thm:1.1}]
By Proposition \ref{prop:1} to \ref{prop:3}, for any triangle $e_{2}^{m+1}\in E_{\ell,2}^{n,m+1},$ we have
$$d_{1}^{m+1}h_{1}^{m+1}(e_{2}^{m+1})=0.$$
By Proposition \ref{prop:4} to \ref{prop:11}, for any vertex $e_{0}^{m+1}\in E_{\ell,0}^{n,m+1},$ we have
$$\delta_{1}^{m+1}h_{1}^{m+1}(e_{0}^{m+1})=0.$$
Thus, $h_{1}^{m+1}$ is a harmonic $1$-form on $G_{\ell}^{n,m+1}.$ Let $\{p_{ij}^{k}\}_{k=1}^{\ell-1}$ denote the set of newly iterated vertices insert into the edge $[p_{i},p_{j}].$ Since $h_{1}^{m+1}$ is harmonic, we have
$$h_{1}^{m+1}([p_{i},p_{ij}^{1}])+h_{1}^{m+1}([p_{ij}^{1},p_{ij}^{2}])+\cdots+h_{1}^{m+1}([p_{ij}^{\ell-1},p_{j}])=h_{1}^{m}([p_{i},p_{j}]).$$
\end{proof}

We now describe a basis for $\mathcal{H}_{\ell,1}^{n,m}.$ Let $h$ be an original harmonic $1$-form. We have the following proposition.
\begin{prop}
Let $h$ be a harmonic $1$-form on $G_{\ell}^{n,1}.$ For any word $\omega$ with length $k,$ if we harmonically extend $h$ to the cell $F_{\omega}G_{m-k,l}^{n},$ and define a new form on it
\begin{equation}\label{eq:h1}
h_{\omega}(e):=\begin{cases}
h\circ F_{\omega}^{-1}(e),\quad &\text{$e\in F_{\omega}E_{\ell,1}^{n,1}$};\\
0,\quad &\text{elsewhere}.
\end{cases}
\end{equation}
Then $h_{\omega}$ is a harmonic $1$-form on $G_{k+1,l}^{n}.$
\label{prop:3.13}
\end{prop}
\begin{proof}
For any $e_{2}^{k+1}\in E_{2,l}^{k+1,n},$ there exist three edges $e_{1}^{k+1}\in E_{1,l}^{k+1,n}$ such that
$${\rm sgn}(e_{1}^{k+1},e_{2}^{k+1})\neq 0.$$
Furthermore, for any $e\in F_{\omega}E_{\ell,1}^{n,1},$ the form satisfies $h_{\omega}(e)=h\circ F_{\omega}^{-1}(e).$ In the set $E_{1,l}^{k+1,n},$ there are $\binom{n+1}{2}N_{\ell}^{n}$ edges contained in $F_{\omega}E_{\ell,1}^{n,1},$ and these edges form $\binom{n+1}{3}N_{\ell}^{n}$ triangles $e_{2}^{k+1}.$ For these triangles, we have
\begin{align*}
d_{1}^{k+1}h_{\omega}(e_{2}^{k+1})&=\sum_{e_{1}^{k+1}\in F_{\omega}E_{\ell,1}^{n,1}}{\rm sgn}(e_{1}^{k+1},e_{2}^{k+1})h_{\omega}(e_{1}^{k+1})\quad(\text{by equation (\ref{eq:2.3})})\\
&=\sum_{\substack{e_{1}^{1}\in E_{\ell,1}^{n,1}\\e_{2}^{k+1}=F_{\omega}e_{2}^{1}}}{\rm sgn}(e_{1}^{k+1},e_{2}^{k+1})h(e_{1}^{1})\quad(\text{by equation (\ref{eq:h1})})\\
&=\sum_{\substack{e_{1}^{1}\in E_{\ell,1}^{n,1}\\e_{1}^{1}\subset e_{2}^{1}}}{\rm sgn}(e_{1}^{1},e_{2}^{1})h(e_{1}^{1})=d_{1}^{1}h(e_{2}^{1})=0.\quad(h\text{ is a harmonic $1$-form})\nonumber
\end{align*}
For other triangles $e_{2}^{k+1},$ by equation (\ref{eq:h1}), we have
$$d_{1}^{k+1}h_{\omega}(e_{2}^{k+1})=\sum_{e_{1}^{k+1}\in E_{n,1}^{k+1}\setminus{F_{\omega}E_{n,1}^{1}}}{\rm sgn}(e_{1}^{k+1},e_{2}^{k+1})h_{\omega}(e_{1}^{k+1})=0.$$

We now prove that $\delta_{1}^{k+1}h_{\omega}=0.$ Let $e_{0}^{k+1}\in E_{\ell,0}^{n,k+1}$ be an ending vertex of some edge in $F_{\omega}E_{\ell,1}^{n,1};$ there are exactly $M_{\ell}^{n,1}$ of them.

We first address the case that $e_{0}^{k+1}\in F_{\omega}E_{\ell,0}^{n,1}.$ Based on the relative orientations between vertices and edges, we decompose the expression $\delta_{1}^{k+1}h_{\omega}$ into two parts: one with ${\rm sgn}=1$ and the other with ${\rm sgn}=-1.$
\begin{align*}
\delta_{1}^{k+1}h_{\omega}(e_{0}^{k+1})
&=\sum_{e_{1}^{k+1}\in F_{\omega}E_{\ell,1}^{n,1}}{\rm sgn}(e_{0}^{k+1},e_{1}^{k+1})h_{\omega}(e_{1}^{k+1})\quad(\text{by equation (\ref{eq:2.6})})\\
&=\sum_{e_{1}^{k+1}\in F_{\omega}E_{\ell,1}^{n,1}}h_{\omega}(e_{1}^{k+1})-\sum_{e_{1}^{(k+1)^{\prime}}\in F_{\omega}E_{\ell,1}^{n,1}}h_{\omega}(e_{1}^{(k+1)^{\prime}})\\
&=\sum_{e_{1}^{1}\in E_{\ell,1}^{n,1}}h(e_{1}^{1})-\sum_{e_{1}^{1^{\prime}}\in E_{\ell,1}^{n,1}}h(e_{1}^{1^{\prime}})\quad(\text{by equation (\ref{eq:h1})})\\
&=\delta_{1}^{1}h(e_{0}^{1})=0.\quad(h\text{ is a harmonic $1$-form})
\end{align*}

For any vertex not contained in $F_{\omega}E_{\ell,1}^{n,1},$ equation (\ref{eq:h1}) directly yields:
$$\delta_{1}^{k+1}h_{\omega}(e_{0}^{k+1})=\sum_{e_{1}^{k+1}\in E_{n,1}^{k+1}\setminus{F_{\omega}E_{\ell,1}^{n,1}}}{\rm sgn}(e_{0}^{k+1},e_{1}^{k+1})h_{\omega}(e_{1}^{k+1})=0.$$
\begin{figure}[htbp]
  \centering
  \begin{tikzpicture}[scale=2]
    \draw (0,0) --(1,0)--(0.5,0.87) --(0,0);
    \draw (1/3,0) --(1/6,0.29) --(0.5,0.29) --(5/6,0.29) --(2/3,0) --(0.5,0.29) --(1/3,0);
    \draw (1/3,0.58) --(2/3,0.58) --(0.5,0.29) --(1/3,0.58);
  \end{tikzpicture}\qquad\qquad
  \begin{tikzpicture}[scale=2]
    \draw (0,0) --(1,0)--(0.5,0.87) --(0,0);
    \draw (1/3,0) --(1/6,0.29) --(0.5,0.29) --(5/6,0.29) --(2/3,0) --(0.5,0.29) --(1/3,0);
    \draw (1/3,0.58) --(2/3,0.58) --(0.5,0.29) --(1/3,0.58);
    \draw (0,0) --(1/9,0) --(1/18,0.096) --(5/18,0.096) --(2/9,0) --(1/6,0.096) --(1/9,0);
    \draw (1/6,0.096) --(1/9,0.19) --(2/9,0.19) --(1/6,0.096);
    \draw (4/9,0) --(7/18,0.096) --(11/18,0.096) --(5/9,0) --(1/2,0.096) --(4/9,0);
    \draw (1/2,0.096) --(4/9,0.19) --(5/9,0.19) --(0.5,0.096);
    \draw (7/9,0) --(13/18,0.096) --(17/18,0.096) --(8/9,0) --(5/6,0.096) --(7/9,0);
    \draw (5/6,0.096) --(7/9,0.19) --(8/9,0.19) --(5/6,0.096);
    \draw (5/18,0.29) --(4/18,0.38) --(8/18,0.38) --(7/18,0.29) --(1/3,0.38) --(5/18,0.29);
    \draw (1/3,0.38) --(5/18,0.48) --(7/18,0.48) --(1/3,0.38);
    \draw (11/18,0.29) --(10/18,0.38) --(14/18,0.38) --(13/18,0.29) --(12/18,0.38) --(11/18,0.29);
    \draw (12/18,0.38) --(11/18,0.48) --(13/18,0.48) --(12/18,0.38);
    \draw [red] (4/9,0.58) --(7/18,0.67) --(11/18,0.67) --(5/9,0.58) --(1/2,0.67) --(4/9,0.58);
    \draw [red] (1/2,0.67) --(8/18,0.77) --(10/18,0.77) --(1/2,0.67);
    \draw [red] (1/3,0.58) --(1/2,0.87) --(2/3,0.58) --(1/3,0.58);
  \end{tikzpicture}
\end{figure}
\end{proof}

Let $\gamma$ be an upside-down triangle on the graph $G_{\ell}^{n,1}.$ By the values of edges and Definition \ref{def:2.10}, $\gamma$ is a $1$-cycle. On each triangle $e_{2}^{0},$ there are $\frac{l(\ell-1)}{2}$ cycles $\gamma.$
\begin{rema}\label{re:1}
Let $\gamma_{i}$ be a $1$-cycle in some $e_{2}^{0}\in E_{\ell,2}^{n,0}.$ For any $h_{1}\in\mathcal{H}_{\ell,1}^{n,1},$ we have
\begin{equation}\label{eq:3.31}
\sum_{i}\int_{\gamma_{i}}h_{1}=0,
\end{equation}
since the sum of the integrals over $\gamma_{i}$ is the difference between the integral over $\partial e_{2}^{0}$ and the sum of the integrals over $\partial e_{2}^{1},$ where $e_{2}^{1}\subset e_{2}^{0}.$ By the definition of harmonic $1$-form and Theorem \ref{thm:2.10}, the sum of all integrals vanishes.
\end{rema}

We define the boundary edges as the outermost edges in the graph $G_{\ell}^{n,m}.$ On $G_{\ell}^{n,1},$ we prescribe the following values for the edges. By the harmonic conditions, \textit{i.e.}, $d_{1}^{1}h(e_{2}^{1})=0$ and $\delta_{1}^{1}h(e_{0}^{1})=0,$ the value of a boundary edge that contains at least one boundary vertex is set to $1.$ If a triangle contains two edges that share the same boundary vertex, then the value of the third edge is set to $-2$ or $0.$ The value of a boundary edge that does not contain any boundary vertex is set to $2.$ If a triangle contains a boundary edge that does not contain any boundary vertex, then the values of the other two edges are set to $-1.$ The values of the remaining edges are set to $0.$

\begin{proof}[Proof of Theorem \ref{thm:1.2}]
If the cells corresponding to $\omega$ and $\omega^{\prime}$ are disjoint, then by equation (\ref{eq:h1}), the harmonic $1$-forms $h_{\omega}$ and $h_{\omega^{\prime}}$ are orthogonal, \textit{i.e.},
$$\langle h_{\omega},h_{\omega^{\prime}}\rangle_{1}^{m}=\sum_{e_{1}^{m}\in E_{\ell,1}^{n,m+1}}\mu_{1}(e_{1}^{m})h_{\omega}(e_{1}^{m})\overline{h_{\omega^{\prime}}(e_{1}^{m+1})}=0.$$
The last equality holds due to the following cases. When $e_{1}^{m+1}\notin F_{\omega}E_{\ell,1}^{n,1}\cup F_{\omega^{\prime}}E_{\ell,1}^{n,1},$ $h_{\omega}(e_{1}^{m+1})=h_{\omega^{\prime}}(e_{1}^{m+1})=0.$ When $e_{1}^{m+1}\in F_{\omega}E_{\ell,1}^{n,1},$ $h_{\omega^{\prime}}(e_{1}^{m+1})=0.$ When $e_{1}^{m+1}\in F_{\omega^{\prime}}E_{\ell,1}^{n,1},$ $h_{\omega}(e_{1}^{m+1})=0.$ In all cases, the product $h_{\omega}(e_{1}^{m})\overline{h_{\omega^{\prime}}(e_{1}^{m+1})}$ vanishes.

Next, assume that the $\omega$-cell is contained in the $\omega^{\prime}$-cell, where $|\omega|=m.$ In $E_{\ell,1}^{n,m+1},$ there are $\binom{n+1}{2}N_{\ell}^{n}$ edges such that $h_{\omega}\neq 0.$ We focus on the triangle formed by three boundary vertices under the IFS maps. Consequently, there are $\binom{n+1}{2}l$ edges on the boundary of the triangle and $\binom{n+1}{2}(N_{\ell}^{n}-l)$ edges in the interior of the triangle. We denote the sets of boundary and interior edges as ${\rm O}_{\ell,1}^{n,m+1}$ and ${\rm I}_{\ell,1}^{n,m+1},$ respectively. By the values of the edges and equations (\ref{eq:2.2}) and(\ref{eq:h1}), we obtain
\begin{align*}
&\qquad\langle h_{\omega},h_{\omega^{\prime}}\rangle_{1}^{m+1}=\sum_{e_{1}^{m+1}\in F_{\omega}E_{\ell,1}^{n,1}}\mu_{1}(e_{1}^{m+1})h_{\omega}(e_{1}^{m+1})\overline{h_{\omega^{\prime}}(e_{1}^{m+1})}\\
&=\sum_{\substack{e_{1}^{m+1}\in E_{\ell,1}^{n,m+1}\\e_{1}^{m+1}\in{\rm O}_{\ell,1}^{n,m+1}}}\mu_{1}(e_{1}^{m+1})h_{\omega}(e_{1}^{m+1})\overline{h_{\omega^{\prime}}(e_{1}^{m+1})}
+\sum_{\substack{e_{1}^{m+1}\in E_{\ell,1}^{n,m+1}\\e_{1}^{m+1}\in{\rm I}_{\ell,1}^{n,m+1}}}\mu_{1}(e_{1}^{m+1})h_{\omega}(e_{1}^{m+1})\overline{h_{\omega^{\prime}}(e_{1}^{m+1})}\\
&=\sum_{\substack{e_{1}^{m+1}\in{\rm O}_{\ell,1}^{n,m+1}\\e_{1}^{1}\in E_{\ell,1}^{n,1}\cap{\rm O}_{\ell,1}^{n,1}}}\mu_{1}(e_{1}^{m+1})h(e_{1}^{1})\overline{h_{\omega^{\prime}}(e_{1}^{m+1})}
+\sum_{\substack{e_{1}^{m+1}\in{\rm I}_{\ell,1}^{n,m+1}\\e_{1}^{1}\in E_{\ell,1}^{n,1}\cap{\rm I}_{\ell,1}^{n,1}}}\mu_{1}(e_{1}^{m+1})h(e_{1}^{1})\overline{h_{\omega^{\prime}}(e_{1}^{m+1})}\\
&=\sum_{\substack{e_{1}^{m+1}\in{\rm O}_{\ell,1}^{n,m+1}\\e_{1}^{1}\in{\rm O}_{\ell,1}^{n,1}\\h(e_{1}^{1})=1}}\mu_{1}(e_{1}^{m+1})\overline{h_{\omega^{\prime}}(e_{1}^{m+1})}
+\sum_{\substack{e_{1}^{m+1}\in{\rm O}_{\ell,1}^{n,m+1}\\e_{1}^{1}\in{\rm O}_{\ell,1}^{n,1}\\h(e_{1}^{1})=2}}\mu_{1}(e_{1}^{m+1})2\times\overline{h_{\omega^{\prime}}(e_{1}^{m+1})}\\
&\quad+\sum_{\substack{e_{1}^{m+1}\in{\rm I}_{\ell,1}^{n,m+1}\\e_{1}^{1}\in{\rm I}_{\ell,1}^{n,1}\\h(e_{1}^{1})=-1}}\mu_{1}(e_{1}^{m+1})(-1)\times\overline{h_{\omega^{\prime}}(e_{1}^{m+1})}
+\sum_{\substack{e_{1}^{m+1}\in{\rm I}_{\ell,1}^{n,m+1}\\e_{1}^{1}\in{\rm I}_{\ell,1}^{n,1}\\h(e_{1}^{1})=-2}}\mu_{1}(e_{1}^{m+1})(-2)\times\overline{h_{\omega^{\prime}}(e_{1}^{m+1})}\\
&\quad+\sum_{\substack{e_{1}^{m+1}\in{\rm I}_{\ell,1}^{n,m+1}\\e_{1}^{1}\in{\rm I}_{\ell,1}^{n,1}\\h(e_{1}^{1})=0}}\mu_{1}(e_{1}^{m+1})(0)\times\overline{h_{\omega^{\prime}}(e_{1}^{m+1})}\quad(\text{the value of $h(e_{1}^{1}),e_{1}^{1}\in E_{\ell,1}^{n,1}$})\\
&=\mu_{1}(e_{1}^{m+1})d_{1}^{m+1}h_{\omega^{\prime}}(e_{2}^{m+1})+\sum_{\substack{e_{1}^{m+1}\in{\rm O}_{\ell,1}^{n,m+1}\\e_{1}^{1}\in{\rm O}_{\ell,1}^{n,1}\\h(e_{1}^{1})=2}}\mu_{1}(e_{1}^{m+1})\overline{h_{\omega^{\prime}}(e_{1}^{m+1})}\\
&\quad-2\Big[\sum_{\substack{e_{1}^{m+1}\in{\rm I}_{\ell,1}^{n,m+1}\\e_{1}^{1}\in{\rm I}_{\ell,1}^{n,1}\\h(e_{1}^{1})=-1}}\mu_{1}(e_{1}^{m+1})\times\overline{h_{\omega^{\prime}}(e_{1}^{m+1})}+\sum_{\substack{e_{1}^{m+1}\in{\rm I}_{\ell,1}^{n,m+1}\\e_{1}^{1}\in{\rm I}_{\ell,1}^{n,1}\\h(e_{1}^{1})=-2}}\mu_{1}(e_{1}^{m+1})\overline{h_{\omega^{\prime}}(e_{1}^{m+1})}\\
&\quad+\sum_{\substack{e_{1}^{m+1}\in{\rm I}_{\ell,1}^{n,m+1}\\e_{1}^{1}\in{\rm I}_{\ell,1}^{n,1}\\h(e_{1}^{1})=0}}\mu_{1}(e_{1}^{m+1})\overline{h_{\omega^{\prime}}(e_{1}^{m+1})}\Big]+\sum_{\substack{e_{1}^{m+1}\in{\rm I}_{\ell,1}^{n,m+1}\\e_{1}^{1}\in{\rm I}_{\ell,1}^{n,1}\\h(e_{1}^{1})=-1}}\mu_{1}(e_{1}^{m+1})\overline{h_{\omega^{\prime}}(e_{1}^{m+1})}\\
&\quad+2\sum_{\substack{e_{1}^{m+1}\in{\rm I}_{\ell,1}^{n,m+1}\\e_{1}^{1}\in{\rm I}_{\ell,1}^{n,1}\\h(e_{1}^{1})=0}}\mu_{1}(e_{1}^{m+1})\overline{h_{\omega^{\prime}}(e_{1}^{m+1})}\quad(\text{by definition of $d_{1}^{m}$})\\
&=\mu_{1}(e_{1}^{m+1})d_{1}^{m+1}h_{\omega^{\prime}}(e_{2}^{m+1})-2\Big[\sum_{\substack{e_{2}^{m+1}\in E_{\ell,2}^{n,m+1}\\e_{2}^{m+1}\supset e_{1}^{m+1}\in F_{\omega}E_{\ell,1}^{n,1}}}\mu_{1}(e_{1}^{m+1})d_{1}^{m+1}h_{\omega^{\prime}}(e_{2}^{m+1})-d_{1}^{m}h_{\omega^{\prime}}(e_{2}^{m})\Big]\\
&\quad+\sum_{\substack{e_{2}^{m+1}\in E_{\ell,2}^{n,m+1}\\e_{2}^{m+1}\supset e_{1}^{m+1}\in{\rm O}_{\ell,1}^{n,m+1}\\h_{\omega}(e_{1}^{m+1})=2}}\mu_{1}(e_{1}^{m+1})d_{1}^{m+1}h_{\omega^{\prime}}(e_{2}^{m+1})+\sum_{\substack{e_{2}^{m+1}\in E_{\ell,2}^{n,m+1}\\e_{2}^{m+1}\supset e_{1}^{m+1}\in{\rm I}_{\ell,1}^{n,m+1}\\h_{\omega}(e_{m+1}^{1})=0}}\mu_{1}(e_{1}^{m+1})d_{1}^{m+1}h_{\omega^{\prime}}(e_{2}^{m+1})\\
&=0. \quad(\text{$h_{\omega},h_{\omega^{\prime}}\in\mathcal{H}_{1,l}^{m+1,n}$})
\end{align*}

Let $M$ be the total number of independent $1$-cycles in $G_{\ell}^{n,1}.$ Since the dimension of de Rham cohomology space $H^{1}$ equals the number of independent $1$-cycles in $G_{\ell}^{n,1},$ combining this with Theorem \ref{thm:2.9} (Hodge decomposition), we deduce that the dimension of $\mathcal{H}_{\ell,1}^{n,1}$ is
$$\dim\mathcal{H}_{\ell,1}^{n,1}=\dim H^{1}=\sharp\{\gamma:\gamma\text{ is an independent $1$-cycles in } G_{\ell}^{n,1}\}=M.$$

Given a basis $\{h_{1,j}\}_{j=1}^{M}$ of $\mathcal{H}_{\ell,1}^{n,1},$ by using the IFS maps $\{F_{i}\},$ we can construct a family of $1$-forms $\{h_{1,j}\circ F_{\omega}^{-1}\}$ on $G_{\ell}^{n,m},$ where $|\omega|\leq m-1.$ By Proposition \ref{prop:3.13}, we know that $h_{1,j}\circ F_{\omega}^{-1}$ is harmonic.

Next, we prove the linear independence of $\{h_{1,j}\circ F_{\omega}^{-1}\}.$ Let $|\omega|=1,$ and let $\sum c_{j}h_{1,j}=0,$ where $c_{j}\in\mathbb{R}.$ For any $1$-cycle $\gamma_{i},$ we have $\sum c_{j}\int_{\gamma_{i}}h_{1,j}=0.$ By Remark \ref{re:1}, \textit{i.e.}, $\int_{\gamma_{i}}h_{1,j}=\delta_{i,j},$ all of the coefficients $c_{j}$ are zero. Assume that $\{h_{1,j}\circ F_{\omega}^{-1}:|\omega|=m\}$ is linearly independent. Let $|\omega|=m+1,$ and let
$$\sum c_{\omega,j}h_{1,j}\circ F_{\omega}^{-1}=0,$$
where $c_{\omega,j}\in\mathbb{R}$ and $F_{\omega}\gamma_{j}$ is an independent $1$-cycle. By Theorem \ref{thm:1.1} and Proposition \ref{prop:3.13}, we have $\int_{F_{\omega^{\prime}}\gamma_{j}}h_{1,j}\circ F_{\omega}^{-1}=0$ for any $\omega^{\prime}\neq\omega.$ In fact, by Definition \ref{def:2.10}, we have
\begin{equation*}
\begin{split}
\int_{F_{\omega}\gamma_{j}}h_{1,j}\circ F_{\omega}^{-1}&=\sum_{e_{1}^{m+1}\in F_{\omega}\gamma_{j}}h_{1,j}\circ F_{\omega}^{-1}(e_{1}^{m+1})=\sum_{e_{1}^{m+1}\in F_{\omega}\gamma_{j}}h_{1,j}(F_{\omega}^{-1}e_{1}^{m})\\
&=\sum_{e_{1}^{1}\in\gamma_{j}}h_{1,j}(e_{1}^{1})=\int_{\gamma_{j}}h_{1,j}=1.
\end{split}
\end{equation*}
When integrating over each $1$-cycle $F_{\omega^{\prime}}\gamma_{j},$ only the term corresponding to $\omega^{\prime}=\omega$ contributes a nonzero value, while all other terms vanish. Consequently, each coefficient $c_{\omega,j}$ must be zero. This establishes the linear independence of $\{h_{1,j}\circ F_{\omega}\}.$

Let $h_{\omega}=\sum c_{\omega^{\prime}}h_{1,j}\circ F_{\omega^{\prime}}^{-1},$ where $c_{\omega^{\prime}}\in\mathbb{R}$ and $|\omega^{\prime}|=m.$ By the definition of harmonic $1$-forms and Proposition \ref{prop:3.13}, we deduce that $h_{\omega}$ is harmonic. Therefore,
$$\mathcal{H}_{\ell,1}^{n,m}={\rm span}\{h_{1,j}\circ F_{\omega}^{-1}\},$$
and
$$\dim\mathcal{H}_{\ell,1}^{n,m}=M+M\times N_{\ell}^{n}+\cdots+M\times (N_{\ell}^{n})^{m-1}=M\frac{(N_{\ell}^{n})^{m}-1}{N^{n}_{\ell}-1}.$$
\end{proof}

\section{2-FORMS ON $2$-DIMENSIONAL LEVEL-$3$ SIERPI\'{N}SKI GASKET}\label{S:4}
\setcounter{equation}{0}
In this section, we discuss a theory of $2$-forms on $SG_{3}^{2}$  that involves the derivatives $d_{2}$ and $\delta_{2}.$ We consider two cases. The first one applies to $2$-forms $f_{2}$ that are absolutely continuous with respect to the Kusuoka measure (with results extendable to $SG_{\ell}^{n}$); the second one applies to $2$-forms $f_{2}$ that are absolutely continuous with respect to the standard self-similar measure on $SG_{3}^{2},$ with continuous Radon-Nikodym derivative.

For any word $\omega$ of length $m,$ let $\Omega=\bigcup_{\omega}F_{\omega}SG_{\ell}^{n}$ be a union of cells on $SG_{\ell}^{n}.$ Let $h$ be a harmonic function. An energy measure $\nu_{h}$ is defined by
$$\nu_{h}(\Omega):=\sum\E_{0}^{m}(h|_{F_{\omega}SG_{\ell}^{n}}):=\sum_{e_{2}^{m}\subset\Omega}\nu_{h}(e_{2}^{m}).$$
The Kusuoka measure is defined as
$$\nu:=\nu_{h}+\nu_{h^{\bot}}$$
where $\{h,h^{\bot}\}$ is an orthogonal basis of harmonic functions modulo constants and the definition of $\nu_{h^{\bot}}$ is similar with $\nu_{h}.$ This definition is independent of the orthogonal basis used.

For simplicity, we take an example on $SG_{3}^{2}.$ Let $\nu_{h}$ be a $2$-form for some harmonic function $h,$ and let $\sigma_{1}$ be a line segment joining $q_{1}$ and $q_{2}.$ We consider the formula $\nu_{h}(\Omega_{n})=\sum\nu_{h}(\sigma_{2}^{n}),$ where $\Omega_{n}=\cup_{|\omega|=n} F_{\omega}SG_{3}^{2}$ with $\omega_{j}=1,2,3,$ and $\sigma_{2}^{n}=F_{\omega}SG_{3}^{2}$ is a cell. From \cite[p.2124]{Azzam-Hall-Strichartz_2008} and \cite[Theorem 2.8]{Owen-Strichartz_2012}, we obtain that
\begin{equation*}
\left(
\begin{array}{c}
\E_{0}(h|_{F_{\omega 10}K})\\
\E_{0}(h|_{F_{\omega 11}K})\\
\E_{0}(h|_{F_{\omega 12}K})
\end{array}\right)
=E_{1}\left( \begin{array}{c}
\E_{0}(h|_{F_{\omega 0}K})\\
\E_{0}(h|_{F_{\omega 1}K})\\
\E_{0}(h|_{F_{\omega 2}K})
\end{array} \right),
\left( \begin{array}{c}
\E_{0}(h|_{F_{\omega 20}K})\\
\E_{0}(h|_{F_{\omega 21}K})\\
\E_{0}(h|_{F_{\omega 22}K})
\end{array}\right)
=E_{2}\left( \begin{array}{c}
\E_{0}(h|_{F_{\omega 0}K})\\
\E_{0}(h|_{F_{\omega 1}K})\\
\E_{0}(h|_{F_{\omega 2}K})
\end{array} \right),
\end{equation*}
\begin{equation*}
\left( \begin{array}{c}
\E_{0}(h|_{F_{\omega 30}K})\\
\E_{0}(h|_{F_{\omega 31}K})\\
\E_{0}(h|_{F_{\omega 32}K})
\end{array}\right)
=E_{3}\left( \begin{array}{c}
\E_{0}(h|_{F_{\omega 0}K})\\
\E_{0}(h|_{F_{\omega 1}K})\\
\E_{0}(h|_{F_{\omega 2}K})
\end{array} \right).
\end{equation*}
where
\begin{equation*}
E_{1}
=\frac{98}{5^{4}3^{3}}\left[ \begin{array}{ccc}
287 & 962 & -283\\
-49 & 3701 & -49\\
-283 & 962 & 287
\end{array} \right],\quad
E_{2}
=\frac{98}{5^{4}3^{3}}\left[ \begin{array}{ccc}
287 & -283 & 962\\
-283 & 287 & 962\\
-49 & -49 & 3701
\end{array} \right],
\end{equation*}
\begin{equation*}
E_{3}
=\frac{1}{5^{4}3^{3}4}\left[ \begin{array}{ccc}
1174 & 49 & 49\\
-962 & 3613 & 1213\\
-962 & 1213 & 3613
\end{array} \right].
\end{equation*}
The relation between $\nu_{h}(F_{i}K)(i=0,1,2)$ and $\nu_{h}(F_{i}K)(i=3,4,5)$ is that
\begin{equation*}
\left( \begin{array}{ccc}
\nu_{h}(F_{3}K)\\
\nu_{h}(F_{4}K)\\
\nu_{h}(F_{5}K)
\end{array}\right)=
\frac{1}{60}\left[ \begin{array}{ccc}
-2 & 13 & 13\\
13 & -2 & 13\\
13 & 13 & -2
\end{array} \right]
\left( \begin{array}{ccc}
\nu_{h}(F_{0}K)\\
\nu_{h}(F_{1}K)\\
\nu_{h}(F_{2}K)
\end{array}\right)
=:C\left( \begin{array}{ccc}
\nu_{h}(F_{0}K)\\
\nu_{h}(F_{1}K)\\
\nu_{h}(F_{2}K)
\end{array}\right).
\end{equation*}
Let $e=(\nu_{h}(F_{0}K),\nu(F_{1}K),\nu(F_{2}K))^{T}.$ We obtain
\begin{equation}
\nu_{h}(F_{\omega}K)=\E_{0}^{m}(h\circ F_{\omega})=\left( \begin{array}{ccc}
1 &
1 &
1
\end{array}\right)
(I+C)E_{\omega_{m}}\cdots E_{\omega_{1}}e,
\end{equation}
where $I$ is the $3\times 3$ identity matrix.

\begin{prop}
Let $\nu$ be a Kusuoka measure on $SG_{3}^{2}.$ For any harmonic function $h,$ the value of $\nu(\Omega_{n})$ can be written as
$$\nu(\Omega_{n})=a\biggr{(}\frac{15235+21\sqrt{257505}}{3375}\biggr{)}^{n}+b\biggr{(}\frac{15235-21\sqrt{257505}}{3375}\biggr{)}^{n}$$
where $\Omega_{n}=\cup_{|\omega|=n}F_{\omega}SG_{3}^{2}$ and the constants $a$ and $b$ are explicitly determined by the values of $h$ in $E_{0,3}^{1,2}.$
\label{prop:4.1}
\end{prop}
\begin{proof}
We first compute $\sum_{|\omega|=j}\E_{0}(h|_{F_{\omega}K})$ when $\omega_{j}=1,2,3.$

When $j=1,$ we have
\begin{equation}
\begin{split}
\sum_{|\omega|=1}\E_{0}(h|_{F_{\omega}K})&=\E_{0}(h|_{F_{1}K})+\E_{0}(h|_{F_{2}K})+\E_{0}(h|_{F_{3}K})\\
&=\left( \begin{array}{ccc}
1 &
1 &
1
\end{array}\right)
(I+C)(E_{1}+E_{2}+E_{3})
\left( \begin{array}{c}
e_{0}\\
e_{1}\\
e_{2}
\end{array}\right).
\end{split}
\end{equation}
When $j=2,$ we have
\begin{equation}
\begin{split}
\sum_{|\omega|=2}\E_{0}(h|_{F_{\omega}K})&=\E_{0}(h|_{F_{1}K})+\E_{0}(h|_{F_{2}K})+\E_{0}(h|_{F_{3}K})\\
&=\left( \begin{array}{ccc}
1 &
1 &
1
\end{array}\right)
(I+C)(E_{1}+E_{2}+E_{3})^{2}
\left( \begin{array}{c}
e_{0}\\
e_{1}\\
e_{2}
\end{array}\right).
\end{split}
\end{equation}
Next, we consider the eigenvectors and eigenvalues of the matrix $E_{1}+E_{2}+E_{3}.$ Let
$$\mathcal{A}=E_{1}+E_{2}+E_{3}=\frac{1}{5^{4}3^{2}4}\left[ \begin{array}{ccc}
75394 & 94619 & 94619\\
-37822 & 522303 & 119703\\
-37822 & 119703 & 522303
\end{array}\right].$$
We can compute that one of the eigenvalues is $\lambda_{1}=\frac{1342}{75}$ with the eigenvector $\left( \begin{array}{ccc}
0 &
1 &
1
\end{array}\right),$ while the form of other eigenvectors is $\left( \begin{array}{ccc}
c &
d &
d
\end{array}\right).$ To compute the other eigenvector, we denote
$$\mathcal{A}\left( \begin{array}{ccc}
c\\
d\\
d
\end{array}\right):=\frac{1}{5^{4}3^{2}4}\left[ \begin{array}{cc}
75394 & 189238\\
-37822 & 534006
\end{array}\right]:=\mathcal{B}$$
Therefore, we only need to compute the eigenvalues and eigenvectors of the matrix $\mathcal{B}.$ By computing, we obtain that $\lambda_{2,3}=\frac{15235\pm 21\sqrt{257505}}{3375}$ and the eigenvectors
$$\left( \begin{array}{cc}
93737\\
114751\pm 210\sqrt{257505}
\end{array}\right).$$
Let
$$\left( \begin{array}{ccc}
e_{0}\\
e_{1}\\
e_{2}
\end{array}\right)=X\left( \begin{array}{ccc}
93737\\
114751+210\sqrt{257505}\\
114751+210\sqrt{257505}
\end{array}\right)+Y\left( \begin{array}{ccc}
93737\\
114751-210\sqrt{257505}\\
114751-210\sqrt{257505}
\end{array}\right)+Z\left( \begin{array}{ccc}
0\\
1\\
-1
\end{array}\right).$$
Then, we have
\begin{align*}
\sum_{|\omega|=j}\E_{0}(h|_{F_{\omega}K})&=\left( \begin{array}{ccc}
1&
1&
1
\end{array}\right)(I+C)(E_{1}+E_{2}+E_{3})^{j}X\left( \begin{array}{ccc}
93737\\
114751+210\sqrt{257505}\\
114751+210\sqrt{257505}
\end{array}\right)\\
&\quad+\left( \begin{array}{ccc}
1&
1&
1
\end{array}\right)(I+C)(E_{1}+E_{2}+E_{3})^{j}Y\left( \begin{array}{ccc}
93737\\
114751-210\sqrt{257505}\\
114751-210\sqrt{257505}
\end{array}\right)\\
&\quad+\left( \begin{array}{ccc}
1&
1&
1
\end{array}\right)(I+C)(E_{1}+E_{2}+E_{3})^{j}Z\left( \begin{array}{ccc}
0\\
1\\
-1
\end{array}\right)\\
&=\frac{7}{5}(323239+420\sqrt{257505})X\biggr{(}\frac{15235+21\sqrt{257505}}{3375}\biggr{)}^{j}\\
&\quad+\frac{7}{5}(323239-420\sqrt{257505})Y\biggr{(}\frac{15235-21\sqrt{257505}}{3375}\biggr{)}^{j}.
\end{align*}
Thus, by the above equation and the definition of Kusuoka measure, we have
$$\nu(\Omega_{n})=\sum\nu(\sigma_{2}^{n})
=a\biggr{(}\frac{15235+21\sqrt{257505}}{3375}\biggr{)}^{n}+b\biggr{(}\frac{15235-21\sqrt{257505}}{3375}\biggr{)}^{n}$$
\end{proof}

From Proposition \ref{prop:4.1}, we can define
$$\delta_{2}^{\prime}(fd\nu)(\sigma_{1}):=\lim_{n\rightarrow\infty}\biggr{(}\frac{3375}{15235+21\sqrt{257505}}\biggr{)}^{n}\sum_{\sigma_{2}^{n}\cap\sigma_{1}}\int_{\sigma_{2}^{n}}fd\nu.$$

As for $SG_{\ell}^{n},$ assign values $x_{i}$ to the boundary points. Using the measure $\nu$ and the harmonic extension matrix $A_{\ell,i}^{n}$ which can be calculated by solving $AX=B,$ we obtain the value of $\nu(F_{i}K)$ and $\nu(F_{j}F_{i}K),$ where $i,j\in\{1,2,...,N_{\ell}^{n}\}.$ By analyzing their relationships, we derive a sequence of matrices $\{E_{i}\}_{i=0}^{N_{\ell}^{n}}.$ Continuing this process yields an expression for $\delta_{2}^{m}$ in terms of the Kusuoka measure.

We now focus on the second case for $SG_{3}^{2},$ analyzing the mappings $d_{1}$ and $\delta_{2}.$ Let $\mu$ be the standard balanced measure on $SG_{3}^{2},$ assigning mass $\frac{1}{6^{m}}$ to each $m$-cell. For any continuous function $f$ on $SG_{3}^{2},$ the measure $fd\mu$ is absolutely continuous with respect to $\mu$ and has a continuous Radon-Nikodym derivative. If $\sigma_{1}^{m}$ is an arbitrary generation-$m$ edge, then
\begin{equation}\label{eq:4.4}
\delta_{2}(fd\mu)(\sigma_{1}^{m})=\int_{\sigma_{1}^{m}}f
\end{equation}
is well defined. Precisely, it can be expressed as the renormalized limit
\begin{equation}\label{eq:4.5}
\delta_{2}(fd\mu)(\sigma_{1}^{m})=\lim_{n\rightarrow\infty}\Big (\frac{6}{3}\Big )^{n}\sum_{\sigma_{2}^{n}\cap\sigma_{1}^{m}\neq\emptyset}\int_{\sigma_{2}^{n}}fd\mu,
\end{equation}
which represents the measure of $\sigma_{1}^{m},$ where $\sigma_{1}^{m}$ satisfies $\sigma_{2}^{n}\cap\sigma_{1}^{m}\neq\emptyset.$ The corresponding definition for $d_{1}$ is
\begin{equation}\label{eq:4.6}
d_{1}f_{1}(\sigma_{2}^{m})=\lim_{n\rightarrow\infty}\Big (\frac{3}{6}\Big )^{n}\sum_{\sigma_{1}^{n}\subset\sigma_{2}^{m}}f_{1}(\sigma_{1}^{n}).
\end{equation}

\begin{prop}\label{prop:4.2}
If we assume $f_{1}=\delta_{2}(fd\mu)$ for some continuous function $f,$ then $d_{1}f_{1}$ exists and
$$d_{1}\delta_{2}(fd\mu)=3fd\mu.$$
\end{prop}
\begin{proof}
Let $f_{1}=\delta_{2}(fd\mu).$ We have
\begin{align*}
&d_{1}f_{1}(\sigma_{2}^{m})=d_{1}\delta_{2}(fd\mu)(\sigma_{2}^{m})\\
&=\lim_{n\rightarrow\infty}\Big (\frac{3}{6}\Big )^{n}\sum_{\sigma_{1}^{n}\subset\sigma_{2}^{m}}\delta_{2}(fd\mu)(\sigma_{1}^{n})\quad(\text{by (\ref{eq:4.6})})\\
&=\lim_{n\rightarrow\infty}\Big (\frac{3}{6}\Big )^{n}\sum_{\sigma_{1}^{n}\subset\sigma_{2}^{m}}\lim_{k\rightarrow\infty}\Big (\frac{6}{3}\Big )^{k}\sum_{\sigma_{2}^{k}\cap\sigma_{1}
^{n}\neq\emptyset}\int_{\sigma_{2}^{k}}fd\mu\quad(\text{by (\ref{eq:4.5})})\\
&=\lim_{n\rightarrow\infty}\lim_{k\rightarrow\infty}\Big (\frac{3}{6}\Big )^{n-k}\sum_{\sigma_{1}^{n}\subset\sigma_{2}^{m}}\sum_{\sigma_{2}^{k}\cap\sigma_{1}
^{n}\neq\emptyset}\int_{\sigma_{2}^{k}}fd\mu\\
&=\lim_{n\rightarrow\infty}\lim_{k\rightarrow\infty}\Big (\frac{3}{6}\Big )^{n-k}\sum_{\sigma_{1}^{n}\subset\sigma_{2}^{m}}3^{k-n}\cdot\int_{\sigma_{2}^{k}}fd\mu.\\
\end{align*}
Upon iterating the edge $\sigma_{1}^{n},$ we get three cells $\sigma_{2}^{n+1}$ satisfying the intersection condition $\sigma_{2}^{k}\cap\sigma_{1}^{n}\neq\emptyset.$ Performing $k-n$ iterations of this edge operation yields the last equality. Similarly, when iterating over the cell $\sigma_{2}^{m},$ we identify three edges $\sigma_{1}^{m+1}$ meeting the containment condition $\sigma_{1}^{n}\subset\sigma_{2}^{m}.$ Each such $\sigma_{2}^{m}$ cell admits exactly three qualifying edges. Consequently, after $n-m$ iterations of this cell operation, the last expression equals
\begin{align*}
&\quad\lim_{n\rightarrow\infty}\lim_{k\rightarrow\infty}\Big (\frac{3}{6}\Big )^{n-k}\cdot 3^{n-m+1}\cdot 3^{k-n}\cdot\int_{\sigma_{2}^{k}}fd\mu\\
&=\lim_{n\rightarrow\infty}\lim_{k\rightarrow\infty}\Big (\frac{1}{6}\Big )^{n-k}\cdot 3^{n-m}\cdot 3\int_{\sigma_{2}^{m}}fd\mu\\
&=3\int_{\sigma_{2}^{m}}fd\mu=3(fd\mu)(\sigma_{2}^{m}).\quad(\text{by (\ref{eq:2.16})})
\end{align*}
\end{proof}
By Proposition \ref{prop:4.2} and Definition \ref{def:2.7}, this yields $\Delta_{2}=3I,$ while the restriction of $\Delta_{1}$ to this class of $1$-forms is also $3I.$

\section*{Acknowledgements}
Part of this  work was carried out while the first author was visiting Beijing Institute of Mathematical Sciences and Applications (BIMSA). He thanks the institute for its hospitality and support.


	\end{document}